\documentclass{article}
\usepackage{bm}
\usepackage{tikz,mathpazo}
\usepackage{pgf}
\usepackage[top=2.5cm, bottom=2.5cm, left=3cm, right=3cm]{geometry}   
\usepackage{indentfirst}
\usepackage{graphicx}
\usepackage{graphics}
\usepackage[toc,page,title,titletoc,header]{appendix}
\usepackage{bm}
\usepackage{listings}  
\usepackage{amsmath,amsthm}
\usepackage{setspace} 
\usepackage{indentfirst}
\usepackage{caption}
\usepackage{multirow} 
\usepackage{lipsum,multicol}
\usepackage{pdfpages}
\usepackage{float}
\usepackage{pdfpages}
\usepackage{url}
\usepackage{colortbl}
\usepackage{subfigure}
\usepackage{epsfig}
\usepackage{epstopdf}
\usetikzlibrary{shapes.geometric, arrows}
\usepackage{fancyhdr}
\usepackage{abstract}
\usepackage{amssymb}
\usepackage{latexsym}
\usepackage{verbatim}
\usepackage[numbers]{natbib} 
\usepackage{booktabs}

\allowdisplaybreaks[4]

\newcommand{\R}{\mathbb{R}}
\newcommand{\Spp}{\mathbb{S}^{p\times p}}

\newcommand{\Rn}{\mathbb{R}^{n}}
\newcommand{\Rp}{\mathbb{R}^{p}}
  
\newcommand{\Rnp}{\mathbb{R}^{n\times p}} 
\newcommand{\Rpp}{\mathbb{R}^{p\times p}}

\newcommand{\inner}[1]{\left\langle #1 \right\rangle}
\newcommand{\norm}[1]{\left\Vert #1\right\Vert}
\newcommand{\bb}[1]{\mathbb{#1}}

\newcommand{\SR}[1]{\Omega_{#1}}

\newcommand{\ca}[1]{\mathcal{#1}}
\newcommand{\tr}[0]{\mathrm{tr}}

\newcommand{\tX}[0]{\tilde{X}}
\newcommand{\tY}[0]{\tilde{Y}}

\newcommand{\Xdoti}{X_{ i \cdot }}
\newcommand{\Xdotiz}{X_{ i \cdot }\tp}
\newcommand{\Xo}[0]{{X_{ \mathrm{orth} }} }

\newcommand{\ff}{_{\mathrm{F}}}
\newcommand{\tf}{_2}

\newcommand{\fs}{^2_{\mathrm{F}}}
\newcommand{\tp}{^\top}

\newcommand{\A}{\ca{A}}

\newcommand{\prox}{{\mathrm{prox}}}
\newcommand{\Apen}[1]{\left( \frac{3}{2} I_p - \frac{1}{2} {#1}\tp {#1} \right)}

\newcommand{\proxk}{\prox_{\eta_k}}
\newcommand{\SLPG}{{SLPG}}
\newcommand{\SLPGs}{{SLPG }}

\newcommand{\Xk}{{X_{k} }}
\newcommand{\Yk}{{Y_{k} }}
\newcommand{\Zk}{{Z_{k} }}
\newcommand{\Xkp}{{X_{k+1} }}

\newcommand{\st}{\mbox{s.\, t.}}

\newcommand{\revise}[1]{{#1}}

\newtheorem{theo}{Theorem}
\newtheorem{lem}{Lemma}
\newtheorem{prop}{Proposition}
\newtheorem{examples}{Problem}
\newtheorem{cond}{Condition}

\newtheorem{defin}{Definition}
\newtheorem{rmk}{Remark}
\newtheorem{assumpt}{Assumption}

\usepackage{caption}
\usepackage{algorithm}
\usepackage{algpseudocode}
\usepackage{longtable}
\usepackage{appendix}

\numberwithin{equation}{section}

\begin{document}
	\title{A Penalty-free Infeasible Approach for a Class of Nonsmooth Optimization Problems over the Stiefel Manifold}
	
	\author{
		{Nachuan Xiao\thanks{State Key Laboratory of Scientific and Engineering Computing, Academy of Mathematics and Systems Science, Chinese  Academy of Sciences, and University of Chinese Academy of Sciences, China, (email: xnc@lsec.cc.ac.cn).} ~
			Xin Liu\thanks{
				State Key Laboratory of Scientific and Engineering Computing, Academy of Mathematics and Systems Science, Chinese Academy of Sciences, and University of Chinese Academy of Sciences, China (email: liuxin@lsec.cc.ac.cn). }~
			and Ya-xiang Yuan\thanks{ State Key Laboratory of Scientific and Engineering Computing, Academy of Mathematics and Systems Science, Chinese Academy of Sciences, China, 
				(email: yyx@lsec.cc.ac.cn). }}
	}
	
	\maketitle
	\begin{abstract}
		Transforming into an exact penalty function model with convex compact constraints yields efficient infeasible approaches for optimization problems with orthogonality constraints. For smooth and $\ell_{2,1}$-norm regularized cases, these infeasible approaches adopt simple and orthonormalization-free updating schemes and show high efficiency in some numerical experiments. However, to avoid orthonormalization while enforcing the feasibility of the final solution, these infeasible approaches introduce a quadratic penalty term, where an inappropriate penalty parameter can lead to numerical inefficiency. 
		Inspired by penalty-free approaches for smooth optimization problems, we proposed a 
		sequential linearized proximal gradient method (\SLPG) for a class of optimization problems with orthogonality constraints and nonsmooth regularization term. This approach alternatively takes tangential steps and normal steps to improve the optimality and feasibility respectively. In \SLPG, the orthonormalization process is invoked only once at the last step if high precision for feasibility is needed, showing that main iterations in \SLPGs are orthonormalization-free. Besides, both the tangential steps and normal steps do not involve the penalty parameter, and thus \SLPGs is penalty-free and avoids the inefficiency caused
		by possible inappropriate penalty parameter. 
		We analyze the global convergence properties of \SLPGs where the tangential steps are inexactly computed. By inexactly computing tangential steps, for smooth cases and $\ell_{2,1}$-norm regularized cases, 
		\SLPGs has a closed-form updating scheme, which leads to cheap tangential steps. 
		Numerical experiments illustrate the advantages of \SLPGs when compared with existing first-order methods. 
	\end{abstract}
	
	\section{Introduction}
	\subsection{Problem description}
	In this paper, we focus on a class of composite 
	optimization problems with orthogonality constraints,
	\begin{equation}\tag{COS}
		\label{Prob_Ori}
		\begin{aligned}
			\min_{X \in \bb{R}^{n\times p} } \quad & f(X) + r(X) \\
			\text{s.t.}\quad &X\tp X = I_p,
		\end{aligned}
	\end{equation}
	where the objective is the summation of
		two functions $f,\, r: \Rnp \mapsto \R$
		satisfying the following blanket assumption.
	\begin{assumpt}[\bf blanket assumption]
		\label{Assumption_1}
		\begin{itemize}
			\item Function $f$ is differentiable  and $\nabla f(X)$ is locally Lipschitz continuous in $\bb{R}^{n\times p}$;
			\item Function $r$ is convex and Lipschitz continuous in $\bb{R}^{n\times p}$;
			\item For any $X, G \in \bb{R}^{n\times p}$ and any $\eta > 0$, the problem 
			\begin{equation*}
				\min_{D \in \bb{R}^{n\times p}}\quad \inner{G, D} + r(D) + \frac{1}{2\eta} \norm{D-X}\fs 
			\end{equation*}
			is of closed-form solution or can be solved efficiently
			by certain iterative approach. 
		\end{itemize}
	\end{assumpt}
	
    The feasible region of the 
		orthogonality constraints $X\tp X = I_p$ 
		is the Stiefel manifold embedded in real matrix space 
		$X \in \ca{S}_{n,p} := \{X \in \bb{R}^{n\times p}| X\tp  X = I_p \}$. 
		We also call it as the Stiefel manifold for brevity. 
	
	The optimization problems of the form \eqref{Prob_Ori} have wide applications in data science and statistics. We mention a few of them in the following.

	\begin{examples}[{\bf Sparse Principal Component Analysis}]
		\label{Example_SPCA}
		Principal component analysis (PCA) is a basic tool 
			in data processing and dimensional reduction.
			It pursues the $p$ leading eigenvectors of the empirical 
			covariance matrix $L$ associated with $N$ 
			samples in $\Rn$. Contemporary datasets often have 
			a new feature that the dimension $n$ is comparable with 
			or even much larger than the samples $N$. At this point,
			we need to take into account the sparsity in the principal 
			components
			for  better representation and consistency. 
			Mathematically, we consider the following sparse PCA 
			model \cite{cai2013sparse,ma2013sparse}, which
			admits a nonsmooth $\ell_1$ norm regularization term.
		\begin{equation}
				\begin{aligned}
					\min_{X \in \bb{R}^{n\times p}}\quad& -\frac{1}{2}\tr\left( X\tp LX \right) + \gamma\norm{X}_1\\
					\st \quad &X\tp X = I_p,
				\end{aligned}
		\end{equation}
		where $\gamma$ is a positive parameter 
			controlling the sparsity.
	\end{examples}

	\begin{examples}[{\bf $\ell_{2,1}$-norm regularized PCA }]
			\label{Example_l21PCA}
			To pursue the sparsity in the features (variables) of the principal 
			components, we can impose the row sparsity to the classical PCA model
			and arrive at the following $\ell_{2,1}$-norm regularized PCA problem
			\cite{ulfarsson2008sparse,cai2013sparse}.
			\begin{equation}\label{SVPCA}
					\begin{aligned}
						\min_{X\in \bb{R}^{n\times p}} \quad &-\frac{1}{2}\tr\left( X\tp LX \right) + \sum_{i = 1}^n \gamma_i \norm{\Xdoti}_{2}\\
						\st \quad & X\tp X = I_p,
					\end{aligned}
			\end{equation}
			where
			$X_{i\cdot}$ and $\gamma_i$ are
			the $i$-th row of matrix $X\in \bb{R}^{n\times p}$ 
			and a positive parameter controlling the row sparsity, respectively, 
			for all $i=1,...,n$.  
			The problem \eqref{SVPCA} is also known as the Coordinate-independent Sparse Estimation \cite{chen2010coordinate}.
		\end{examples}
		
		When the nonsmooth part vanishes, i.e. $r = 0$, the objective function 
		of \eqref{Prob_Ori} reduces to a smooth function. There are
		many applications in this scenario as well, for instance, the 
		discretized Kohn-Sham energy minimization problem arising in material sciences.

	\begin{examples}[Discretized Kohn-Sham Energy Minimization]
		\label{Example_KS}
		Kohn-Sham density functional theory (KSDFT) \citep{Kohn1965Self} is widely used in electronic structure calculation. 
			In the last step of KSDFT, it requires to minimize the 
			following discretized Kohn-Sham energy function 
			over the Stiefel manifold.
		\begin{equation}
			\begin{aligned}
				\min_{X \in \bb{R}^{n\times p}}\quad &  
				\frac{1}{4} \tr\left( X\tp LX \right) + \frac{1}{2}\tr\left( X\tp V_{ion}X \right) + \frac{1}{4} \rho\tp L^{\dagger}\rho + \frac{1}{2}\rho\tp \epsilon_{\mathrm{xc}}(\rho)\\
				\text{s.t.} \quad & X\tp X = I_p,
			\end{aligned}
		\end{equation} 
		where $L\in\bb{R}^{n\times n}$ and diagonal matrix $V_{\mathrm{ion}}\in\bb{R}^{n\times n}$
		refers to the Laplace operator in the planewave basis and discretized local ionic potential, respectively, $\rho := \mathrm{diag}(XX\tp)$ denotes the charge density, and
		$\epsilon_{\mathrm{xc}}:\Re^n\mapsto \Re^n$ stands for the exchange correlation function. 
	\end{examples}
	
		\begin{rmk}
			The blanket assumption is not strict at all, since
			it holds at all the instances we listed above.
			Moreover, it is
			the same as those imposed in \cite{chen2018proximal,huang2019extending,huang2019riemannian}.
		\end{rmk}

	\subsection{Existing methods}
	On minimizing smooth objectives over the 
	Stiefel manifold, there exist several efficient approaches, such as 
	gradient-based methods \cite{manton2002optimization,nishimori2005learning,abrudan2008steepest}, conjugate gradient methods \cite{edelman1998geometry,abrudan2009conjugate}, projection-based methods \cite{Absil2009optimization,dai2019adaptive}, constraint preserving updating scheme \cite{wen2013feasible,jiang2015framework}, Newton methods \cite{hu2018adaptive}, trust-region methods \cite{absil2007trust}, first-order methods with multipliers correction framework \cite{gao2018new}, infeasible methods \cite{gao2019parallelizable,xiao2020class}, etc.
	Interested readers are referred to the book \cite{Absil2009optimization},
	the survey paper \cite{hu2020} and the references therein. 
	It is worth mentioning that several infeasible approaches 
	have been proposed and show their high efficiency 
		in solving optimization problems over the Stiefel manifold. 
		The ALM-based approaches PLAM and PCAL\cite{gao2019parallelizable} 
		update the Lagrangian multipliers by an explicit expression
		derived by the first-order stationarity conditions.
		Such explicit expression involves the gradient of the objective,
		and hence these algorithms can only tackle the problems
		with smooth objective function. Gao et al. \cite{GHKL2020}
		provide a clear route of applying PCAL to the electronic structure
		calculation. 
		Xiao et al. \cite{xiao2020class} present 
		a novel penalty function with compact convex constraints (PenC). 
		In the framework of PenC, they propose approximate
		projected gradient and Newton methods PenCF and PenCS, respectively. 
		Hu et al. \cite{hu2020anefficiency} propose 
			an unconstrained penalty model for sparse dictionary learning and dual principal component pursuit.
	
		However, most of the above-mentioned approaches can hardly be
		applied to the problem with nonsmooth objective function 
		directly.
		The approaches for solving \eqref{Prob_Ori} with $r\neq 0$
		are not as many as those for smooth minimization.
		We review a few representative ones in the following.
	
	The first class of approaches 		
		are based on the splitting and alternating. 
		The splitting method for orthogonality constrained problem (SOC) \cite{lai2014splitting} introduces auxiliary variables to split the objective function and the orthogonality constraints, and then 
		applies the alternating direction method of multipliers (ADMM)
		to solve the equivalent splitting model. 
		The subproblem related to the objective function lacks
		closed-form solution in general which is a main limit of SOC.
		Meanwhile, Rosman et al.
		\cite{rosman2014augmented-lagrangian} propose a variable splitting framework based on augmented Lagrangian method for problems on imaging processing, which can also be extended to solve optimization problems on $\ca{S}_{n,n}$. 
	Besides,  Chen et al. \cite{chen2016augmented} propose a proximal alternating minimization approach based on augmented Lagrangian method (PAMAL).
	Different from SOC, PAMAL develops an equivalent model by introducing two blocks of variables to split the orthogonality constraints, smooth and nonsmooth terms apart. PAMAL invokes the augmented Lagrangian method (ALM) framework
		and block coordinate descent (BCD) method  to solve the split model and the subproblems related to the primal variables, respectively.
	
	The second classes of approaches apply the proximal gradient
		method to tackle the nonsmooth term in \eqref{Prob_Ori}. 
		Chen et al. \cite{chen2018proximal} propose the Riemannian proximal gradient method (ManPG) and its accelerated version, ManPG-Ada.  
		The main iteration, which occupies the main computational cost
		of ManPG or ManPG-Ada, is to compute the following proximal mapping restricted to the tangent space $\ca{T}_{\Xk}:= \{ \Delta \in \bb{R}^{n\times p} | \Delta\tp \Xk + \Xk\tp \Delta = 0 \}$
		of the Stiefel manifold.
		\begin{equation}
			\label{Eq_Subproblem_ManPG}
			\min_{D \in \Xk + \ca{T}_{\Xk}} \quad \inner{D, \nabla f(\Xk)} + r(D) + \frac{1}{2\eta_k } \norm{D-\Xk}\fs,
		\end{equation}
	where $\eta_k >0$ is the stepsize.
		The subproblem \eqref{Eq_Subproblem_ManPG} is a nonsmooth convex optimization problem without closed-form solution in general and
		can be solved by the  semi-smooth Newton method (SSN) \cite{Sun2002}.
	Their numerical experiments show that both ManPG and ManPG-Ada outperform the existing splitting and alternating based approaches SOC and PAMAL.
	Recently, Huang et al. present a 
	Riemannian version of fast iterative shrinkage-thresholding algorithm with safeguard (AManPG) in \cite{huang2019extending}, which exhibits the accelerated behavior over the Riemannian proximal gradient method. Nevertheless, no convergence rate analysis is presented for AManPG. 
	They also propose
		a modified Riemannian proximal gradient method (RPG) and its accelerated version (ARPG), respectively, in \cite{huang2019riemannian}.
	They show the $\ca{O}(\frac{1}{k})$-convergence rate of RPG and ARPG.
	However, the proximal mapping subproblems in both RPG and ARPG are 
	even more expensive to solve than ManPG
		due to their nonsmoothness and nonconvexiety. 
		Thus, ARPG and RPG are usually slower than AManPG and ManPG-Ada in solving 
		optimization problems on the Stiefel manifold as
		illustrated in \cite{huang2019riemannian}.
	
	The key motivation of PenC is to utilize the explicit expression
		of the Lagrangian multipliers at first-order stationary 
		points, which involves the Euclidian gradient of the objective function.
		Hence, it can hardly be generalized to the nonsmooth case,
		in which the gradient of the objective function is absent.
		Xiao et al. \cite{xiao2020l21} extend PenC to a special 
		case of \eqref{Prob_Ori} in which $r$ takes the $\ell_{2,1}$-norm
		like \eqref{SVPCA} in Problem \ref{Example_l21PCA}.
		Although the subdifferential of $r$ in this case is set-valued,
		the term $X\tp \partial r(X)$ is single-valued. Based on this observation,
		the authors of \cite{xiao2020l21} propose the corresponding PenC model
		and a proximal gradient method called PenCPG.
	In PenCPG, the proximal subproblem is of closed-form solution, which 
		leads to its numerical superiority when compared with the existing Riemannian proximal gradient approaches in solving $\ell_{2,1}$-norm regularized problems.
	
	However, if the nonsmooth term $r$ is not a $\ell_{2,1}$-norm, 
	the term $X\tp \partial r(X)$ is set-valued in general.
	Hence, the Lagrangian multipliers at any stationary point no longer
	have closed-form expression. Therefore, the PenC model does not apply 
	to \eqref{Prob_Ori} in general.
	
	Another limitation of PenC based approaches is 
	that their numerical performances are
	related to the choice of the penalty parameter,
	as reported in \cite{gao2018new,xiao2020class,xiao2020l21}.
		But slow convergence or even divergence occurs, if the penalty parameter
		is out of such range. The authors in \cite{gao2019parallelizable} provide heuristic way to select 
		the penalty parameter without theoretical guarantee.

	\subsection{Motivation}
	
	In order to develop an efficient infeasible approach
		for solving \eqref{Prob_Ori} which is not sensitive
		to the penalty parameter, we borrow the idea from 
		a class of sequential quadratic programming (SQP) approaches 
		developed for solving the equality constrained smooth nonlinear optimization problems. These approaches include
		the inexact-restoration method  proposed
		by \citet{martinez2001inexact}, the nonmonotone trust-region based
		SQP methods proposed by \citet{ulbrich2003non}, \citet{gould2010nonlinear},
		\citet{liu2011a}, and \citet{chen2019a}, respectively. 
		In particular, the authors in \cite{gould2010nonlinear} and
		\cite{liu2011a} provide inexact strategies to tackle the
		SQP subproblems.
		Besides, the approaches presented in \cite{ulbrich2004on}
		and \cite{shen2012a} utilize the  nonmonotone filter techniques.
		However, all of these approaches invoke the second-order oracle 
		or use the first-order information to approximate the Hessian
		of the objective function or its Lagrangian.
		Hence, these approaches are only valid in the smooth problems. 
		To the best of our knowledge, there are few approaches 
		for solving the nonsmooth optimization problems such as \eqref{Prob_Ori}
		by adopting the SQP-like techniques.
	
		Our main idea is to reduce the objective function in the tangent space 
		and to improve the feasibility in the normal space alternatively.
		We first approximate the objective function by a proximal linearized model
		and minimize it on an affine subspace spanned at the current iterate which is parallel to
		a tangent space of the Stiefel manifold. We then invoke 
		a normal step which searches in the range space of the Jacobian of 
		the constraints $X\tp X - I_p = 0$ to
		reduce the feasibility violation.
	
	\subsection{Contributions} 
	We propose a first-order penalty-free infeasible approach,
		called sequential linearized proximal gradient method (\SLPG),
		for solving a class of composite optimization problems with orthogonality constraints \eqref{Prob_Ori}. 
		In each iteration, \SLPGs 
		takes the tangential and the normal steps one after the other,
		both of which do not involve any orthonormalization procedure or updating
		of penalty parameters. Consequently, \SLPGs enjoys high scalability and avoids the numerical inefficiency from inappropriately selected penalty parameters.
		We discuss how to solve the tangential subproblems inexactly, which is different 
		from the existing approaches since the iterates are no longer feasible.
		We provide a novel idea to conduct the normal steps which simultaneously have both low computational cost and
		fast convergence to the feasible region.
		To combine the tangential and normal steps together, the subsequence convergence as well as the worst-case complexity of \SLPGs can be established under  mild assumptions.
		Furthermore, when the nonsmooth term of \eqref{Prob_Ori} has a special structure, 
		i.e. the Lagrange multipliers with respect 
		to the orthogonality constraints are of closed-form expressions, the tangential steps of \SLPGs
		enjoy closed-form approximate solutions, and hence an inner loop to solve the tangential subproblem
		is waived. 
		The efficiency and robustness of \SLPGs are illustrated by a set of numerical
		experiments on the sparse PCA, the $\ell_{2,1}$-norm regularized PCA, and the discretized Kohn-Sham energy minimization
		problems. \SLPGs visibly outperforms the state-of-the-art feasible approaches in solving those
		nonsmooth problems. It exhibits its prominent robustness when compared with the 
		existing infeasible approaches.

	\subsection{Notations and Organization} 
	Let $\Spp:=\{A\mid A\in\Rpp,\, A=A\tp \}$
		be the set containing all the real symmetric $p\times p$ matrices. 
		We use $I_p$ to denote the $p \times p$ identity matrix.
	The entry in the $i$-th row and the $j$-th column of a matrix $X\in\Rnp$
	is denoted by $X_{ij}$. 
		For brevity, we use
		$\norm{X}_{1}$ to represent the component-wise $\ell_1$ norm,
		i.e. $\norm{X}_{1} = \sum_{i,j} |X_{ij}|$.
	The Euclidean inner product of two matrices $X, Y\in \bb{R}^{n\times p}$ is defined as $\inner{X,Y}=\tr(X\tp Y)$,
	where $\tr(A)$ is the trace of a matrix $A\in \bb{R}^{p\times p}$.
	$\norm{\cdot}_2$ and $\norm{\cdot}\ff$ represent the $2$-norm and the Frobenius norm, respectively. 
	For a positive semi-definite matrix $A\in\Rpp$, $A^{\frac{1}{2}}$ refers to the unique positive semi-definite
		matrix satisfying $A^{\frac{1}{2}}A^{\frac{1}{2}}=A$ and $A^{-\frac{1}{2}}$ is its inverse.
	
	The rest of this paper is organized as follows. 
	In Section 2, we present the detailed description of \SLPGs, and introduce the practical implementations on computing the tangential and the normal steps.  
	We establish the convergence analysis for \SLPGs in Section 3. 
		The preliminary numerical experiments are reported in Section 4. Finally, we conclude this paper in Section 5.

	\section{Algorithm Description}

	In this section, we mainly propose the framework of \SLPG.
		We first provide some necessary preliminary definitions. 
		Then we present the mathematical formulations of the tangential and normal subproblems
		and introduce how to solve them respectively.
		Finally, we demonstrate the complete algorithm framework.

	\subsection{Preliminary}
	

	We first review the definition of Clark's subdifferential \cite{clarke1990optimization} for nonsmooth functions.  
	\begin{defin}[\cite{clarke1990optimization,rockafellar2009variational}]
		
		For any Lipschitz continuous $f$ on $\bb{R}^{n\times p}$,  the generalized directional derivative of $f$ in the  direction $D\in \bb{R}^{n\times p}$ is defined by, 
		\begin{equation}
			f^o(X, D) := \mathop{\lim\sup}\limits_{Y \to X, t \to 0^+} \frac{f(Y+tD)-f(Y)}{t}. 
		\end{equation}
		Based on generalized directional derivative of $f$, the Clark's subdifferential 
		(``subdifferential" for brevity)
		of $f$ is defined by,
		\begin{equation}
			\partial f(X) := \left\{ W \in \bb{R}^{n\times p}| \inner{W, D} \leq f^o(X,D) \text{ for any }D \in \bb{R}^{n\times p} \right\}.
		\end{equation}
	\end{defin}

	
	
	Next we describe the stationarity of \eqref{Prob_Ori}, which is same as \cite[Definition 3.3]{chen2018proximal}.
	\begin{defin}
		\label{Defin:FOSP_Ori}
		Under the Assumption \ref{Assumption_1}, we call $X$ a first-order stationary point of \eqref{Prob_Ori} if and only if there exists $W \in \partial r(X)$ such that 
		\begin{equation}\label{new1}
				\left\{
				\begin{aligned}
					&\nabla f(X) + W - X\Phi\left(X\tp \left( \nabla f(X) + W \right)\right) = 0,\\
					& X\tp X = I_p.
				\end{aligned} 
				\right.
		\end{equation}
	\end{defin}

	\begin{defin}
		\label{Defin_operator}
		A operator $T: \bb{R}^{p\times p} \mapsto \bb{R}^{p\times p}$ is nonexpansive 
			if and only if there exists a constant $c\in [0,1]$ such that
			\begin{equation*}
				\norm{T(\Lambda_1) - T(\Lambda_2)}\ff \leq c\norm{\Lambda_1 - \Lambda_2}\ff
			\end{equation*} 
			holds for any $\Lambda_1, \Lambda_2 \in \bb{R}^{p\times p}$.
		%
	\end{defin}

	\subsection{Computing the tangential step}\label{subsec:1}
		Suppose $\Xk$ is the current iterate, we define the affine subspace $\A_k$ as
	$$\A_k :=\left\{ X \in \bb{R}^{n\times p} \left|  \Phi(X\tp \Xk) = \Xk\tp \Xk \right. \right\}.$$
	Here $\Phi: \bb{R}^{p\times p} \to \bb{R}^{p\times p}, \Phi(M) = \frac{M+M\tp}{2}$ is an operator that symmetrize the square matrices in $\bb{R}^{p\times p}$. 
	To reduce the function value, we minimize the following proximal linearized approximation 
	of the objective function with stepsize $\eta_k$  on the space $\A_k$.
		\begin{equation}
			\label{Subproblem_tangential}
			\begin{aligned}
				\mathop{\min}_{D\in \A_k} \inner{\nabla f(\Xk), D} + r\left(D \right) + \frac{1}{2\eta_k}\norm{D - \Xk}\fs.
			\end{aligned}
		\end{equation}
		We call \eqref{Subproblem_tangential} the tangential subproblem for convenience hereinafter.
		Different with the tangential step in \cite{chen2018proximal}, the subproblem 
		\eqref{Subproblem_tangential}
		is constructed on an infeasible point $\Xk$.
		
		By simple calculations, we can obtain 
		the following KKT condition of the convex optimization problem \eqref{Subproblem_tangential}.
		\begin{equation}
			\label{Eq_subproblem_optimality}
			\left\{\begin{aligned}
				&0 \in  \nabla f(\Xk) - \Xk\Lambda + \partial r(D) + \frac{1}{\eta_k}(D - \Xk), \\
				& \Phi( {D} \tp \Xk) = \Xk\tp \Xk, 
			\end{aligned}\right.
		\end{equation}
		where 
		$\Lambda \in \Spp$
		is the Lagrangian multiplier of the linear constraint $D\in \A_k$.
		
		Once $\Lambda$ is fixed, the first relation in \eqref{Eq_subproblem_optimality} determines
		\begin{eqnarray*}
			D = \prox_{\eta_k}(\nabla f(\Xk) - \Xk\Lambda; \Xk),
		\end{eqnarray*}
		where the proximal mapping $\proxk: \Rnp \otimes \Rnp \mapsto \Rnp$ is defined by 
		$$\proxk(G; \Xk):= \mathop{\arg\min}_{D \in \bb{R}^{n\times p}}  \inner{G, D} + r(D ) + \frac{1}{2\eta_k} \norm{D-\Xk}\fs.$$
		Then it is clear that the KKT condition \eqref{Eq_subproblem_optimality} 
		is equivalent to the nonlinear equation $E(\Lambda)=0$, where
		\begin{eqnarray*} 
			E(\Lambda):=\Phi\left( {\left(\prox_{\eta_k}(\nabla f(\Xk) - \Xk\Lambda; \Xk) - \Xk\right)} \tp \Xk\right).
		\end{eqnarray*}
		This equation can be rewritten as the fixed point equation
		\begin{equation}\label{Eq_T}
			\Lambda - tE(\Lambda) = \Lambda, \quad \mbox{where\,} t>0.
		\end{equation}
		
		We adopt the following Arrow-Hurwicz algorithm proposed by \citet{beale1959studies}
		to solve \eqref{Eq_T}.

		\begin{algorithm}[htbp]
			\begin{algorithmic}[1]   
				\Require Input data: current iterate $\Xk$, parameter $\eta_k$;
				\State Choose initial guess $\Lambda_0\in\mathbb{S}^{p\times p}$, set $j:=0$;
				\While{not terminate}
				\State Calculate the proximal mapping: 
				$D_j = \prox_{\eta_k}(\nabla f(\Xk) - \Xk\Lambda_j; \Xk)$;
				\State Main update: $\Lambda_{j+1} = \Lambda_j - \frac{1}{\eta_k}E(\Lambda_j)$;
				\State Set $j:=j+1$;
				\EndWhile
				\State Return $\Yk := D_{j}$. 
			\end{algorithmic}  
			\caption{Fixed point iteration}  
			\label{Alg:FPI}
		\end{algorithm}
		
		Chambolle et al. \cite{chambolle2011a} have provided 
		an $\ca{O}\left( \frac{1}{k} \right)$ convergence rate of the Arrow-Hurwicz algorithm locally.
		Later on, He et al. \cite{he2014on} have shown that the Arrow-Hurwicz algorithm, as a special case of  
		their primal-dual hybrid gradient algorithm (PDHG), enjoys an $\ca{O}\left( \frac{1}{k} \right)$ convergence rate in the ergodic sense under mild conditions containing our case. 
		
		In our infeasible framework, we actually do not need an accurate solution 
		to the tangential subproblem
		\eqref{Subproblem_tangential}. More specifically, in Algorithm \ref{Alg:FPI}, we adopt the 
		following terminating condition for the residual. 
		
		\begin{cond}
			\label{Assumption_2} 
			There exist $C>0$ such that $\Yk\in \Rnp$ returned by
			Algorithm \ref{Alg:FPI} satisfy 
			\begin{equation}
				\label{Eq_errors_in_tangential}
				\norm{\Phi\left(\left(\Yk-\Xk\right)\tp \Xk\right)}\ff \leq c\eta_k 
				\norm{\Xk\tp\Xk - I_p}\ff . 
			\end{equation}
		\end{cond}
	
	\subsection{A practical inexact tangential step in special cases}
	\label{subsection_closed_form}
		In \cite{xiao2020class} and \cite{xiao2020l21}, it is shown that
		for two special cases of \eqref{Prob_Ori} with $r(X)=0$ and $r(X)=\sum_{i = 1}^n \gamma_i \norm{\Xdoti}_{2}$,
		the Lagrangian multipliers have explicit expressions
		\begin{eqnarray}\label{mult:1}
			\Lambda(X) &=& \Phi(X\tp \nabla f(X))\quad \mbox{and}\\
			\label{mult:2}
			\Lambda(X) &=& \Phi(X\tp \nabla f(X)) + \sum_{i = 1}^n \gamma_{i} S(\Xdotiz),
		\end{eqnarray}
		respectively, at any first-order stationary point, where $S$ is defined by
		\begin{eqnarray}\label{eq:Sdef}
			S(x):=\left\{
			\begin{array}{cc}
				\frac{x x\tp}{\norm{x}_2}, & \mbox{if\,} x\neq 0;\\
				0_p, & \mbox{otherwise,}
			\end{array}
			\right.
		\end{eqnarray}
		and $0_p$ is the zero vector in $\Rp$. 
		
		Here we can propose an alternative way to inexactly solve the tangential subproblem \eqref{Subproblem_tangential} other than Algorithm \ref{Alg:FPI} with Condition \ref{Assumption_2} 
		by using the expressions \eqref{mult:1} and \eqref{mult:2} to estimate the multipliers of 
		\eqref{Subproblem_tangential} and then get the proximal mapping. Namely, we adopt the following
		two step algorithm.
		
		\begin{algorithm}[htbp]
			\begin{algorithmic}[1]   
				\Require Input data: current iterate $\Xk$, parameter $\eta_k$;
				\State If $r(X)=0$, calculate $\Lambda$ by \eqref{mult:1} with $X=\Xk$;
				\State If $r(X)=\sum_{i = 1}^n \gamma_i \norm{\Xdoti}_{2}$, 
				calculate $\Lambda_k$ by \eqref{mult:2} with $X=\Xk$;
				\State Calculate the proximal mapping: 
				$Y_k = \prox_{\eta_k}(\nabla f(\Xk) - \Xk\Lambda_k; \Xk)$;
				\State Return $Y_k$. 
			\end{algorithmic}  
			\caption{Using explicit expressions}  
			\label{Alg:EPR}
		\end{algorithm}

	\subsection{Computing the normal step}\label{subsec:2}
After obtaining $\Yk$, an inexact solution of the tangential subproblem, we need 
		to consider a normal step to reduce the feasibility violation.
		A usual way to realize it is to pull this intermediate point back to the Stiefel manifold
		through certain projection, i.e. orthonormalization process, such as the QR decomposition,
		the polar decomposition, and so on.
		As we know, orthonormalization is usually unscalable and expensive when $p$ is large.
		%
		An accurate normal step usually does not help much for the overall
		performance as the tangential step is an inexact solution
		of a linear approximate model. 
		Therefore, we consider to compute the orthonormalization inexactly
		to balance the accuracies of the tangential and the normal steps.
		A parallelizable algorithm proposed in \cite{higham1994APA} computes the polar decomposition 
		by adopting the Pad\'e approximant, whose main computational cost can be attributed to the inverse of  a
		series of $p\times p$ matrices which can be realized by  solving linear equations.

		The Taylor expansion of $z^{-\frac{1}{2}}$ at $z=1$ to order one is 
		$z^{-\frac{1}{2}} = 1 - \frac{1}{2}(z-1) + \ca{O}((z-1)^2).$
		Let $z=X\tp X$, we have
		\begin{equation*}
			\norm{(X\tp X)^{-\frac{1}{2}} - \Apen{X}}\ff = \ca{O}\left(\norm{X\tp X- I_p}\fs\right). 
		\end{equation*}
		Hence, the polar decomposition at the intermediate iterate $\Yk(\Yk\tp \Yk)^{-\frac{1}{2}}$ can be approximated
		by the following normal step
	\begin{equation}
		\label{Subproblem_normal}
		\Xkp = \Yk \Apen{\Yk}.
	\end{equation}
	
	Next we show how the above normal step reduce the feasibilit violation.
	\begin{lem}
		\label{Le_bound_X}
		For any $X \in \bb{R}^{n\times p}$ satisfying $\norm{X\tp X - I_p}\tf \leq \frac{1}{4}$, let $\hat{X}:= X\Apen{X}$, then it holds that
		\begin{equation*}
			\norm{\hat{X}\tp \hat{X} - I_p}\ff \leq \frac{13}{16}\norm{X\tp X - I_p}\fs.
		\end{equation*}
	\end{lem}
	\begin{proof}
		It directly follows from the condition $\norm{X\tp X - I_p}\tf \leq \frac{1}{4}$ that
		\begin{equation*}
			\norm{I_p - \frac{1}{4}X\tp X}_2 = \norm{ \frac{3}{4}I_p + \frac{1}{4}(I_p - X\tp X) }_2 \leq \frac{3}{4} + \frac{1}{4} \norm{X\tp X - I_p}_2 \leq \frac{13}{16}. 
		\end{equation*}
		%
		Together with the definition of $\hat{X}$, we have
		\begin{equation}
			\begin{aligned}
				{}&\norm{\hat{X}\tp \hat{X} - I_p}\ff = \norm{X\tp X - I_p + (I_p - X\tp X) X\tp X + \frac{1}{4}X\tp X(X\tp X - I_p)^2 }\ff\\
				={}& \norm{ \left(I_p - \frac{1}{4}X\tp X\right)(X\tp X - I_p)^2 }\ff
				\leq \frac{13}{16}\norm{(X\tp X - I_p)^2}\ff \leq \frac{13}{16}\norm{X\tp X - I_p}\fs.
			\end{aligned}
		\end{equation}
	\end{proof}

	\subsection{Algorithm}
	
	Now, we are ready to present the framework of our SLPG algorithm which alternatively
		takes the tangential and the normal steps introduced in the Subsections \ref{subsec:1}, \ref{subsec:2}, 
		respectively.
		
		\begin{algorithm}[!htbp]
			\caption{ Sequential Linearized Proximal Gradient method (\SLPG)}   
			\label{Alg_MOrthPen}
			\begin{algorithmic}[1]   
				\Require Input data:  functions $f$ and $r$;
				\State Choose initial guess $X_0$, set $k:=0$; 
				\While{not terminate}
				\State Choose parameter $\eta_k$;
				\State Solve the tangential subproblem \eqref{Subproblem_tangential} inexactly
				to make Condition \ref{Assumption_2} hold by Algorithm \ref{Alg:FPI}, and obtain $\Yk$;
				\State Compute  the normal step  \eqref{Subproblem_normal}, and obtain $\Xkp$;
				\State Set $k:=k+1$;
				\EndWhile
				\If{need post-process}
				\State Compute $X_k:= U_k V_k\tp$, where $\Xk = U_k\Sigma_kV_k\tp$ is the singular value decomposition (SVD)
				of $\Xk$ in economic size;
				\EndIf
				\State Return $X_k$.
			\end{algorithmic}  
		\end{algorithm}

		The post-process stated in Line 9 of Algorithm \ref{Alg_MOrthPen} pursues an accurate feasible solution if necessary.
		As shown later in theoretical and numerical analyses, such post-process does not affect the substationarity much. In addition, it can further reduce the function value 
		while decreasing the feasibility violation to machine precision.

		\begin{rmk}
			It is also worth mentioning 
			that if the normal step in \SLPGs takes the orthonormalization process, 
			the sequence $\{\Xk\}$ generated by \SLPGs is on the Stiefel manifold. In addition, 
			the tangential subproblem is strictly on the tangent space of the Stiefel manifold at $\Xk$.
			By choosing suitable parameter $c$, \SLPGs reduces to the existing approach ManPG \cite{chen2018proximal}. 
			In other word, ManPG can be regarded as a special variant of \SLPGs in which both of the tangential and 
			the normal steps are computed more precisely. 
		\end{rmk}

	\section{Global Convergence of \SLPG}
	In this section, we first establish the global convergence of \SLPGs without the post-process by constructing a merit function and evaluating the sufficient function value reduction. For convenience, when we mention Algorithm \ref{Alg_MOrthPen}
		in the first two subsections in this section, the post-process is
		switched off.
		Then we demonstrate that the post-process provides further function value reduction.  
		For  convenience, we define the following constants at the very beginning:
		\begin{eqnarray*}
			L_f &:=& \sup\limits_{X,\,Y\in\SR{1/2}} 
			\frac{| f(X) - f(Y)|}{\norm{X-Y}\ff} = 
			\sup\limits_{X\in\SR{1/2}} ||\nabla f(X)||\ff,\\
			L_{f'} &:=& \sup\limits_{X,\,Y\in\SR{1/2}} 
			\frac{\norm{\nabla f(X) -\nabla f(Y)}\ff}{\norm{X-Y}\ff},\\
			L_r &:=& \sup\limits_{X,\,Y\in\SR{1/2}} 
			\frac{|r(X) -r(Y)|}{\norm{X-Y}\ff},
		\end{eqnarray*}
		where $\SR{r}:=\{X\mid \norm{X\tp X - I_p}\ff\leq r\}$, for any given
		$r\geq 0$.
		We also introduce a new assumption on the parameter sequence $\{\eta_k\}$.
	   \begin{assumpt}\label{Assumption_3}
			Assume that the parameters in Algorithm \ref{Alg_MOrthPen}
			satisfies $\eta_k\in\left[\frac{\tilde{\eta}}{2},\tilde{\eta}\right]$
			for all $k=0,1,...$, 
			where
			$$\tilde{\eta} \,=\, \min\left\{\frac{1}{6(1+c)(L_f + L_r + c +1)},
			\frac{1}{2L_{f'}+8(L_f + L_r)+3}\right\},$$
			and $c$ is defined in Condition \ref{Assumption_2}.
		\end{assumpt}
	
	\subsection{Preliminary properties of the iterate sequences}
	We first demonstrate some properties of the iterate sequences $\{\Xk\}$ and $\{\Yk\}$ generated by \SLPGs 
		including the boundedness and the reduction on the 
		feasibility under a mild assumption on the initial guess.	
	\begin{lem}
		\label{Le_bound_Y}
		Suppose the sequences $\{\Xk\}$ and $\{\Yk\}$ are generated by Algorithm \ref{Alg_MOrthPen}. 
		Then, it holds that 
		\begin{equation}
			\norm{\Yk\tp \Yk - I_p}\ff \leq \left(1+2c \eta_k\right)\norm{\Xk\tp \Xk - I_p}\ff   + \norm{\Yk-\Xk}\fs.
		\end{equation}
	\end{lem}
	\begin{proof}
		
			Let $\tilde{D}:= \Yk-\Xk$, we have
			\begin{eqnarray*}
				&& \norm{\Yk\tp \Yk - I_p}\ff 
				= \norm{\Xk \tp \Xk +2\Phi(\tilde{D}\tp \Xk)+ \tilde{D}\tp \tilde{D}- I_p}\ff\\
				&\leq& \norm{\Xk\tp \Xk - I_p}\ff+ 2\norm{\Phi(\tilde{D}\tp \Xk)}\ff  
				+ \norm{\tilde{D}\tp \tilde{D}}\ff\\
				&\leq& \left(1+2c \eta_k \right)\norm{\Xk\tp \Xk - I_p}\ff   + \norm{\Yk-\Xk}\fs.
			\end{eqnarray*}
			Here the last inequality follows from Condition \ref{Assumption_2}. 
	\end{proof}
	
	Lemma \ref{Le_bound_Y} shows that the tangential step 
	may increase the feasibility violation.
	But fortunately, it can be controlled in some senses. Next, we investigate the boundedness of the iterate sequences with a suitable initial guess.

	\begin{lem}
		\label{Le_subprob_descrease}
			Suppose that Assumption \ref{Assumption_3} holds and 
			$\Xk\in\SR{\tau}$ with 
			\begin{eqnarray}\label{eq:add1}
				\tau=\frac{1}{4(1+c)^2}
			\end{eqnarray}
		    Then, we have
			\begin{equation}
				\label{eq:fun-red}
					\begin{aligned}
						&\inner{\Yk, \nabla f(\Xk)} + r(\Yk) + \frac{1}{2\eta_k} \norm{\Yk-\Xk}\fs  \\
						\leq{}&  \inner{\Xk, \nabla f(\Xk)} + r(\Xk) + \frac{6c(5L_f+5L_r + c)\eta_k}{25}\norm{\Xk\tp \Xk - I_p}\ff.
					\end{aligned}
			\end{equation}
	\end{lem}
	\begin{proof}
		For convenience, we denote $p_k(D) := \inner{D, \nabla f(\Xk)} + r(D) + \frac{1}{2\eta_k} \norm{D-\Xk}\fs $. 
		It directly follows from the fact $\Xk\in\SR{\tau}$ 
		that $\Xk\tp \Xk$ is non-singular and hence we can define $\Zk := \Xk + \Xk(\Xk\tp \Xk)^{-1}\Phi((\Yk-\Xk)\tp \Xk)$. By the definition of $\Zk$, we first obtain
		\begin{equation*}
			\Phi((\Zk - \Xk)\tp \Xk) = \Phi((\Yk - \Xk)\tp \Xk).
		\end{equation*}
		Besides, together with Condition \ref{Assumption_2}, the distance between $\Zk$ and $\Xk$ can be estimated by
		\begin{equation*}
			\norm{\Zk - \Xk}\ff \leq \norm{\Xk(\Xk\tp \Xk)^{-1}}_2 \norm{\Phi((\Yk-\Xk)\tp \Xk)}\ff \leq \frac{6c\eta_k}{5}\norm{\Xk\tp \Xk - I_p}\ff,
		\end{equation*}
		where the last inequality results from the fact that $\norm{\Xk(\Xk\tp \Xk)^{-1}}_2\leq \frac{2\sqrt{3}}{3}<6/5$ which is implied by the inclusion 
		$\Xk\in\SR{\tau}$.
	    %

		Then we have		
		\begin{equation}
			\label{Eq_Le_subprob_descrease}
			\begin{aligned}
				&|p_k(\Zk) - p_k(\Xk)| \leq \left| \inner{\Zk-\Xk, \nabla f(\Xk)} \right| + \left| r(\Zk) - r(\Xk) \right| + \frac{1}{2\eta_k} \norm{\Zk -  \Xk}\fs\\
				\leq{}& (L_f + L_r) \norm{\Zk - \Xk}\ff + \frac{1}{2\eta_k} \norm{\Zk -  \Xk}\fs\\
				\leq{}&  \frac{6c(L_f+L_r)\eta_k}{5}\norm{\Xk\tp \Xk - I_p}\ff + \frac{18c^2\eta_k}{25}\norm{\Xk\tp \Xk - I_p}\fs. 
			\end{aligned}
		\end{equation}
	
		On the other hand, we consider the following optimization problem,
		\begin{equation}
			\label{Eq_Le_subprob_allu_model}
			\begin{aligned}
				\min_{D \in \bb{R}^{n\times p} }~ &\inner{D, \nabla f(\Xk)} + r(D) + \frac{1}{2\eta_k} \norm{D-\Xk}\fs \\
				\text{s.t.}~&  \Phi(\Xk\tp (D-\Xk)) =  \Phi(\Xk\tp (\Yk-\Xk)).
			\end{aligned}
		\end{equation}
	As described in Algorithm \ref{Alg:FPI}, there exists a symmetric matrix  $\Lambda_k$ 
	such that 
		$\Yk = \prox_{\eta_k}(\nabla f(\Xk) - \Xk\Lambda_k; \Xk)$.
	Then by \citet{rockafellar2009variational}, we have
	\begin{equation}
		0 \in \nabla f(\Xk) - \Xk \Lambda_k + \frac{1}{\eta_k}(\Yk - \Xk) + \partial r(\Yk) = \partial p_k(\Yk) - \Xk \Lambda_k.
	\end{equation}
	Therefore, there exists $W_k \in \partial p_k(\Yk)$ such that $W - \Xk \Lambda_k = 0$. Therefore, for any feasible $D$ in \eqref{Eq_Le_subprob_allu_model}, we have
	\begin{equation}
		\begin{aligned}
			&\inner{W, D - \Yk} = \inner{\Xk\Lambda_k, D - \Yk} = \inner{\Lambda_k, \Xk\tp (D - \Yk)} = \inner{\Lambda_k, \Phi(\Xk\tp (D - \Yk))}\\
			={}& \inner{\Lambda_k, \Phi(\Xk\tp (D - \Xk))} - \inner{\Lambda_k, \Phi(\Xk\tp (\Yk - \Xk))} = 0- 0 = 0. 
		\end{aligned}
	\end{equation}
	Then together with  \cite[Theorem 1.1.1]{hiriart2013convex}, 
	we can conclude that $\Yk$ is the global minimizer of 
	\eqref{Eq_Le_subprob_allu_model}, and hence $p_k(\Yk) \leq p_k(Z_k)$. 
		Recall the inequality \eqref{Eq_Le_subprob_descrease}, we arrive at 
		\begin{eqnarray*}
			&&	p_k(\Yk) - p_k(\Xk) \leq p_k(\Zk) - p_k(\Xk) \leq |p_k(\Zk) - p_k(\Xk)|\\
			&\leq& \frac{6\eta_kc(L_f+L_r)}{5}\norm{\Xk\tp \Xk - I_p}\ff + \frac{18c^2\eta_k}{25}\norm{\Xk\tp \Xk - I_p}\fs \\
			&\leq& \frac{6c(5L_f+5L_r + c)\eta_k}{25}\norm{\Xk\tp \Xk - I_p}\ff. 
		\end{eqnarray*}
		Here the last inequality follows the the inclusion 
		$\Xk\in\SR{\tau}\subset\SR{1/3}$. Then we complete the proof. 
	\end{proof}

	\begin{lem}
		\label{Le:upper_bound_dist}
			Suppose that Assumption \ref{Assumption_3} holds and the iterate sequences $\{\Xk\}$ and $\{\Yk\}$ are generated by Algorithm \ref{Alg_MOrthPen}
			initiated from $X_0$ satisfying $X_0\in\SR{\tau}$ with $\tau$
			defined in \eqref{eq:add1}. 
			Then for any $k=0,1, ...$, it holds
			\begin{eqnarray}\label{eq:feas}
				&\max\left\{ \norm{\Yk - \Xk}\ff,\,
				\norm{\Yk\tp \Yk - I_p}\ff\right\} \leq \frac{1}{2(1+c)},
				\quad 
				\norm{\Xk\tp\Xk - I_p}\ff \leq \frac{1}{4(1+c)^2},&\\ 
				\label{eq:feasred}
				&\norm{\Xkp\tp \Xkp - I_p}\ff  \leq   \frac{13}{32}\norm{\Xk\tp \Xk - I_p}\ff + \frac{13}{32} \norm{\Yk-\Xk}\ff^2.&
			\end{eqnarray}
	\end{lem}
	\begin{proof}
	We use mathematical induction. 
	Clearly $\norm{\Xk\tp\Xk - I_p}\ff \leq \frac{1}{4(1+c)^2}$ holds for 
	$k=0$. 
	From Lemma \ref{Le_subprob_descrease}, we have
	\begin{equation}\label{eq:add2}
			\frac{1}{2\eta_k} \norm{\Yk-\Xk}\fs 
			\leq (L_f + L_r) \norm{\Yk - \Xk}\ff + \frac{6\eta_kc(5L_f+5L_r + c)}{25}\norm{\Xk\tp \Xk - I_p}\ff. 
	\end{equation}
	Suppose that $\norm{\Yk - \Xk}\ff > 3\eta_k (L_f + L_r + c + 1)$, we have
	\begin{eqnarray*}
			&&\frac{1}{2\eta_k}
			 \norm{\Yk-\Xk}\fs =\frac{1}{3\eta_k}
			 \norm{\Yk-\Xk}\fs  + \frac{1}{6\eta_k}
			 \norm{\Yk-\Xk}\fs \\
			&>& (L_f + L_r) \norm{\Yk - \Xk}\ff + \frac{3}{2} (L_f + L_r+ c + 1)^2\eta_k \\
			&>& (L_f + L_r) \norm{\Yk - \Xk}\ff + \frac{6c(5L_f+5L_r + c)\eta_k}{25}\norm{\Xk\tp \Xk - I_p}\ff,
	\end{eqnarray*}
	where the last inequality results from 
	the inclusion $\Xk\in\SR{\tau}$. Clearly, this statement contradicts the
	inequality \eqref{eq:add2}.
	Therefore, we have
	\begin{equation*}
		\norm{\Yk - \Xk}\ff \leq 3\eta_k (L_f + L_r+ c + 1) \leq \frac{1}{2(1+c)},
	\end{equation*}
	where the last inequality follows from Assumption \ref{Assumption_3}.

	On the other hand, by recalling Lemmas \ref{Le_bound_X}, \ref{Le_bound_Y} and the Cauchy-Schwartz inequality, 
	we obtain
	\begin{eqnarray*}
		&&\norm{\Xkp\tp\Xkp - I_p}\ff \leq \frac{13}{16}\norm{\Yk\tp \Yk - I_p}\fs
		\leq \frac{13}{16} \left(\left(1+2\eta_kc \right)\norm{\Xk\tp\Xk - I_p}\ff + \norm{\Yk - \Xk}\fs\right)^2\\
		&\leq&
		\frac{13}{8}
		\left(\left(1+2\eta_kc \right)^2\norm{\Xk\tp\Xk - I_p}\fs + \norm{\Yk - \Xk}^4\ff\right)
		\leq \frac{13}{32}\norm{\Xk\tp \Xk - I_p}\ff + \frac{13}{32} \norm{\Yk-\Xk}\ff^2.
	\end{eqnarray*}
		
	Finally, by Lemma \ref{Le_bound_X} we have
		\begin{equation}\label{eq:XXI-1}
			\norm{\Xkp\tp\Xkp - I_p}\ff \leq \frac{13}{16}\norm{\Yk\tp \Yk - I_p}\fs 
			< \frac{1}{4(1+c)^2}. 
		\end{equation}
	Thus, we can conclude the proof by using the mathematical induction.

	\end{proof}

	\subsection{Global convergence}
	Before presenting the main convergence theorem of \SLPG, we
		first estimate certain sufficient function value  reduction.
	\begin{lem}
		\label{Le_Decrease_tangential}
		Suppose that Assumption \ref{Assumption_3} holds and the iterate sequences $\{\Xk\}$ and $\{\Yk\}$ are generated by Algorithm \ref{Alg_MOrthPen}
			initiated from $X_0$ satisfying $X_0\in\SR{\tau}$ with $\tau$
			defined in \eqref{eq:add1}. 
			Then for any $k=0,1, ...$, it holds
		\begin{equation}\label{eq:funred}
			\begin{aligned}
				f(\Xkp) + r(\Xkp) \leq{}& f(\Xk) + r(\Xk) + \left( -\frac{1}{2\eta_k} + \frac{L_{f'} }{2} + L_f+L_r \right)  \norm{\Yk-\Xk}\fs \\
				&+ \left(L_f+L_r+\frac{3c}{4+4c}\right)\norm{\Xk\tp\Xk - I_p}\ff.  
			\end{aligned}
		\end{equation}
	\end{lem}
	\begin{proof}
		Recalling the inequality \eqref{eq:fun-red}
			and the Taylor expansion of the objective function 
			of \eqref{Prob_Ori}, we can obtain
		\begin{equation}
			\label{Eq_Le_Decrease_tangential_2}
			\begin{aligned}
				& \left(f(\Yk) + r(\Yk)\right) -  \left(f(\Xk) + r(\Xk) \right)\\
				{\,\leq\,}& \inner{\Yk-\Xk, \nabla f(\Xk)} 
				+ \frac{L_{f'} }{2}\norm{\Yk-\Xk}\fs 
				+ r(\Yk) - r(\Xk) \\
				&+ \frac{6\eta_kc(5L_f+5L_r + 3c)}{25}\norm{\Xk\tp \Xk - I_p}\ff.\\
				{\,\leq\,}& \left( -\frac{1}{2\eta_k} + \frac{L_{f'} }{2} \right)\norm{\Yk-\Xk}\fs  + \frac{c}{4+4c} \norm{\Xk\tp \Xk - I_p}\ff. 
			\end{aligned}
		\end{equation}
		Here the last inequality follows the upper-bound for $\eta_k$ described in Assumption \ref{Assumption_3}.

		On the other hand, the assertion
			\eqref{eq:feas} of Lemma \ref{Le:upper_bound_dist} directly
			implies that $\norm{\Yk}_2 \leq 2$, which together with
			the normal step \eqref{Subproblem_normal} lead to the fact that
			\begin{equation*}
				\norm{\Xkp - \Yk}\ff = \norm{\Yk\Apen{\Yk} - \Yk}\ff = \frac{1}{2}\norm{\Yk(\Yk\tp\Yk - I_p)}\ff 
				\leq  \norm{\Yk\tp \Yk - I_p}\ff.
			\end{equation*} 
		
		Then by the Lipschitz continuity of $f$ and $r$,
		Lemma \ref{Le_bound_Y} and Assumption \ref{Assumption_3}, 
		we arrive at
		\begin{equation}
			\label{Eq_Le_Decrease_tangential_3}
			\begin{aligned}
				&\left(f(\Xkp) + r(\Xkp)\right) - \left( f(\Yk) + r(\Yk) \right) \\
				\leq & (L_f+L_r)\norm{\Xkp-\Yk}\ff \leq  (L_f+L_r)\norm{\Yk\tp \Yk - I_p}\ff \\
				\leq & \left(1+2\eta_kc\right)(L_f+L_r)\norm{\Xk\tp\Xk - I_p}\ff + (L_f+L_r) \norm{\Yk-\Xk}\fs \\
				\leq & 
				\left(L_f+L_r+\frac{c}{2+2c}\right)\norm{\Xk\tp\Xk - I_p}\ff + (L_f+L_r) \norm{\Yk-\Xk}\fs.
			\end{aligned}
		\end{equation}
		
		After summing up the inequalities \eqref{Eq_Le_Decrease_tangential_2} and \eqref{Eq_Le_Decrease_tangential_3} together, we complete the proof. 
	\end{proof}
	
	In the next step, we need to evaluate the sufficient reduction
		of the following merit function.
	\begin{equation}
		\label{Penalty_function}
		h(X) := f(X) + r(X) + \left(2L_f+2L_r+\frac{3}{2}\right) \norm{X\tp X - I_p}\ff. 
	\end{equation}
	
	\begin{lem}
		\label{Le_Decrease_all}
		Suppose that Assumption \ref{Assumption_3} holds and the iterate sequences $\{\Xk\}$ and $\{\Yk\}$ are generated by Algorithm \ref{Alg_MOrthPen}
			initiated from $X_0$ satisfying $X_0\in\SR{\tau}$ with $\tau$
			defined in \eqref{eq:add1}. 
			Then for any $k=0,1, ...$, it holds
		\begin{equation}\label{eq:hred}
			h(\Xkp) - h(\Xk) \leq - \frac{1}{4\eta_k} \norm{\Yk-\Xk}\fs - \frac{3}{16}\left(L_f+L_r+1\right)\norm{\Xk\tp\Xk - I_p}\ff. 
		\end{equation}
	\end{lem}
	\begin{proof}
			This is a direct corollary of inequalities \eqref{eq:feasred},
			\eqref{eq:funred} and \eqref{Penalty_function}.
	\end{proof}

	\begin{theo}
		\label{The_convergence}
		Suppose that Assumption \ref{Assumption_3} holds and the iterate sequences $\{\Xk\}$ and $\{\Yk\}$ are generated by Algorithm \ref{Alg_MOrthPen}
			initiated from $X_0$ satisfying 
			satisfying $X_0\in\SR{\tau}$ with $\tau$
			defined in \eqref{eq:add1}. 
			Then the sequence
			$\{\Xk\}$ exists at least one accumulation point
			which must be a first-order stationary point of \eqref{Prob_Ori}. 
		
		Moreover, 
		\begin{equation}\label{l1}
			\min_{0\leq i\leq k} \frac{1}{\eta_i}\norm{{Y_{i}} - X_i}\ff \leq  \sqrt{\frac{26L_f+26L_r+6}{(k+1)\tilde{\eta}}}. 
		\end{equation}
		and 
		\begin{equation}\label{l2}
			\min_{0\leq i\leq k} \norm{{X_i}\tp X_i - I_p}\ff \leq \frac{52}{3(k+1)}. 
		\end{equation}
	\end{theo}
	\begin{proof}
		Summing up the inequality \eqref{eq:hred} from $k=0$ to $+\infty$, 
		we obtain
		\begin{equation}
					\label{Eq_sufficient_decrease}
					\begin{aligned}
						&\sum_{k = 0}^{+\infty} \left[ \frac{1}{4{\eta_k}} \norm{\Yk-\Xk}\fs + \frac{3}{16}\left(L_f+L_r+1\right)\norm{\Xk\tp\Xk - I_p}\ff\right]
						\leq  h(X_0) - \lim\limits_{N\to +\infty} h(X_N)\\
						\leq{} & \sup_{N\to +\infty} (L_f+L_r)\norm{X_N - X_0}\ff + \left(2L_f+2L_r+\frac{3}{2}\right)\cdot  2\cdot
						\frac{1}{4(1+c)^2}\\ 
						\leq{} & \left(\frac{13}{4}L_f+\frac{13}{4}L_r+\frac{3}{4}\right),
					\end{aligned}
		\end{equation}
			where the last inequality uses the fact $||X_k||\ff \leq \frac{\sqrt{5}}{2}$
			which is implied by the second inequality of \eqref{eq:feas}.
			Thus, it holds that
			$$\mathop{\lim}\limits_{k\to +\infty} \norm{\Yk - \Xk}\ff =0,\quad
			\mbox{and\quad} \mathop{\lim}\limits_{k\to +\infty} \norm{\Xk\tp \Xk - I_p}\ff =0.$$
			
		On the other hand, by the boundedness of $\{\Xk\}$,
			we know that this sequence exists accumulation point,
			and denote it by $\bar{X}$.
			Recalling the boundedness of $\{\Xk\}$,
			$\{\Yk\}$ and $\{\eta_k\}$, 
			without loss of generality, we can
			assume that  
		there exists a subsequence $\{k_j\}_{j=1,2,...}$  such that $X_{k_j} \to \bar{X}$ and meanwhile
			it holds that $\eta_{k_j}  \to \bar{\eta}$.
			It can be easily verified that 
			\begin{eqnarray}\label{eq:lim}
				\norm{{Y_{k_j}} - X_{k_j}}\ff \to 0
				\quad\mbox{and}\quad
				\norm{{X_{k_j}}\tp X_{k_j} - I_p}\ff \to 0,
			\end{eqnarray}
			which imply $\bar{X}\tp \bar{X} - I_p = 0.$
		
		For convenience, we invoke the Maximum Theorem stated in \cite[p.116]{berge1963topological} without proof. 
			The Maximum Theorem tells us 
			that the Lipschitz continuities of 
			$\nabla f$ and $r$ lead to the fact that the global minimizer of \eqref{Subproblem_tangential}
			is continuous with respect to $\Xk$. 
			We define $\bar{Y}$ as  
			\begin{equation*}
				\bar{Y} = \mathop{\arg\min}_{\Phi(D\tp \bar{X}) = \Phi(\bar{X}\tp \bar{X})} \inner{\nabla f(\bar{X}), D} + r(D ) + \frac{1}{2\bar{\eta}}\norm{D - \bar{X}}\fs.
			\end{equation*}
			Combining 
			the definition of the tangential step,
			the orthonormalization of $\bar{X}$, 
			the relation \eqref{eq:lim}, we have $\bar{Y} = \bar{X}$. By simple calculation, we can conclude that
			$\bar{X}$ satisfies \eqref{new1}
			and hence is a first-order stationary point of \eqref{Prob_Ori}.

			By summing up the inequality \eqref{eq:hred} from $i=0$ to $k$,
			and using the same deduction in \eqref{Eq_sufficient_decrease},
			we obtain 
			\begin{eqnarray*}
				&\sum\limits_{i = 0}^k \frac{1}{4{\eta_i}^2}\norm{{Y_{i}} - X_i}\fs \leq \sum\limits_{0 = 1}^k \frac{1}{2 \bar{\eta}\eta_i}\norm{{Y_{i}} - X_i}\fs \leq 
				\frac{13L_f+13L_r+3}{2\bar{\eta}},&\\
				&\mbox{and}\quad \sum\limits_{i = 0}^k \frac{3}{16}\left(L_f+L_r+1\right)\norm{{X_i}\tp X_i - I_p}\ff \leq 
				\frac{13L_f+13L_r+3}{4}, &
			\end{eqnarray*}
			which imply the inequalities \eqref{l1} and \eqref{l2} immediately.
		
	\end{proof}

	\subsection{Orthonormalization as post-process}
	\label{section_sub_post_process}
	In the last subsection, we present a result on how the post-process 
		affects the value of the merit function.
	\begin{prop}
		\label{Prop_Orthonormalization}
		Suppose $X$ satisfying $\norm{X\tp X - I_p}\ff \leq \frac{1}{4}$. 
			Let $X = U\Sigma V\tp$ be the SVD of $X$ in economic size for $X$ 
			and we set $\Xo := UV\tp$, then it holds that
		\begin{equation*} 
			h(\Xo) \leq h(X) - \left(L_f+L_r+\frac{3}{2}\right)  \norm{X\tp X - I_p}\ff. 
		\end{equation*}
	\end{prop}
	\begin{proof}
		Firstly, by simple calculation, we have 
		\begin{equation}
			\label{Eq_Prop_Orthonormalization_1}
			\norm{\Xo - X}\ff = \norm{\Sigma - I_p}\ff \leq \norm{(\Sigma + I_p)(\Sigma - I_p)}\ff = \norm{X\tp X - I_p}\ff.
		\end{equation}
		
		Then by the Lipschitz continuity of $f$ and $r$, we obtain
		\begin{equation*}
			\begin{aligned}
				\left(f(\Xo) + r(\Xo)\right) - \left(f(X) + r(X)\right) 
				\leq   (L_f+L_r)\norm{\Xo - X}\ff \leq  (L_f+L_r)\norm{X\tp X - I_p}\ff,
			\end{aligned}
		\end{equation*}
		which implies
		\begin{equation*}
				\begin{aligned}
					&h(\Xo) - h(X) \\
					\leq{}&  \left(f(\Xo) + r(\Xo)\right) - \left(f(X) + r(X)\right)
					-\left(2L_f+2L_r+\frac{3}{2}\right) \norm{X\tp X - I_p}\ff\\
					\leq{}&  -\left(L_f+L_r+\frac{3}{2}\right)  \norm{X\tp X - I_p}\ff. 
				\end{aligned}
		\end{equation*}
	\end{proof}

	\section{Numerical Experiments}
	In this section, we perform preliminary numerical experiments to illustrate the efficiency and the robustness of \SLPGs.
		We first present the test settings including how to choose the parameters
		in SLPG, introduce the test problems and then illustrate some observations in 
		the numerical tests.
		Then we compare SLPG with some of the state-of-the-art algorithms on these test problems. 
	
	All the numerical experiments in this section are run in serial in a platform with Intel(R) Xeon(R) Silver 4110 CPU @ 2.10GHz and 
	384GB RAM running MATLAB R2018a under Ubuntu 18.10. 
	
	\subsection{Test settings}
	
	Theorem \ref{The_convergence} has provided a range for choosing the stepsize parameter
		$\eta_k$ with guaranteed convergence. However, such choice is too restrictive to
		be practically useful. In this section, we suggest to adopt the following  
		extended version,
		which was first proposed in \cite{wen2010fast}, of Barzilar-Borwein (BB) stepsize \cite{barzilai1988two} in \SLPG. 
	\begin{equation}
		\label{Eq_Stepsize}
		\eta_{k} =	{\inner{S^k,V^k}} \left/ {\inner{V^k,V^k}}\right.,
	\end{equation}
	where $S^k = X_{k} - X_{k-1}$ and $$V^k = \left[\nabla f(\Xk) - \Xk\Phi(\Xk\tp \nabla f(\Xk)) \right] - \left[\nabla f(X_{k-1}) - X_{k-1}\Phi({X_{k-1}}\tp \nabla f({X_{k-1}})) \right].$$ 
	
	In Algorithm \ref{Alg:FPI}, we set the maximum iterations as $10$ and choose the stepsize $t$ as $1/\eta_{k}$. In \SLPG, we also adopt the warm-start technique 
		in selecting the initial guess of Algorithm \ref{Alg:FPI}. Namely,
		$\Lambda_0$ in the $k+1$-th iteration can be set as the last $\Lambda_j$ in the $k$-th iteration. Besides, we set the constant $c$ as $1000$ in Condition \ref{Assumption_2}.
	
	In this paper, the substationarity, the feasibility violation (``feasibility" for short) and that of the tangential subproblem (``TS feasibility" for short) at the $k$-th iterate are estimated by
	\begin{eqnarray*}
		\norm{\Yk-\Xk}\ff/\eta_k,\quad
		\norm{\Xk\tp\Xk - I_p}\ff,\quad
		\mbox{and\,\,}
		\norm{\Phi(\Xk\tp ( \Yk - \Xk))}\ff,
	\end{eqnarray*}
	respectively. 
	Unless otherwise stated, 
	\SLPGs terminates if either the 
	stopping criteria $ \norm{\Yk - \Xk}\ff/\eta_k \leq 10^{-4}$
	is satisfied or the maximum number of iterations $10000$ is reached.

	\subsection{Test Problems}
	We adopt Problems  \ref{Example_SPCA}-\ref{Example_KS} as the test problems.
		Unless otherwise stated,
		for Problems  \ref{Example_SPCA} and \ref{Example_l21PCA}, 
		we set the covariance matrix $L\in\bb{R}^{1000\times 1000}$
		of $200$ randomly generated samples $S = \mathrm{randn}(1000,200)$
		with unified normalization as the following
		\begin{eqnarray}\label{eq:co}
			L = SS\tp/\norm{S}^2_2.
		\end{eqnarray}
	For Problem  \ref{Example_KS}, we uses the test instances as ``h2o'' molecular from KSSOLV toolbox  \cite{yang2009kssolv}. Additionally, the initial points are chosen as the leading $p$ eigenvectors of $L$ for Problems  \ref{Example_SPCA} and \ref{Example_l21PCA}, or 
		generated by the build-in function ``getX0'' in KSSOLV toolbox \cite{yang2009kssolv} for Problem  \ref{Example_KS}. 
	
	\subsection{Observations in testing \SLPG}
	
	We first investigate how the substationarity, 
	feasibility and TS feasibility vary
		in the running of \SLPGs without post-process in solving Problems  \ref{Example_SPCA} and  \ref{Example_l21PCA}
		with randomly generated data. 
	We put the numerical results in
		Figure \ref{Fig_iter_feas}. The blue, red and yellow lines
		represent the substationarity, feasibility and TS feasibility, respectively.
		The problem parameters are listed below the subfigures.
		We can learn from
		Figure \ref{Fig:l21pca}-\ref{Fig:ks} that the feasibility violation of \SLPGs is actually a high order 
		infinitesimal of the substationarity, which coincides with the 
		theoretical results in Lemma \ref{Le_bound_X}. 
	It is worthy of mentioning that \SLPGs decreases the  feasibility violation much faster than
		those existing infeasible first-order approaches, such as PLAM, PCAL in \cite{gao2019parallelizable}, and PenCF from \cite{xiao2020class}, in solving \eqref{Prob_Ori}.
	%
	Although we can only theoretically establish 
	the global sublinear convergence rate for \SLPG, 
	in Figure \ref{Fig_iter_feas}, 
	we have observed its local linear convergence rate in solving Problems \ref{Example_SPCA}-\ref{Example_KS}.
	
	\begin{figure}[!htbp]
		\centering
		\subfigure[Problem \ref{Example_SPCA}, $(p,\gamma) = (5,0.007)$]{
			\begin{minipage}[t]{0.33\linewidth}
				\centering
				\includegraphics[width=\linewidth]{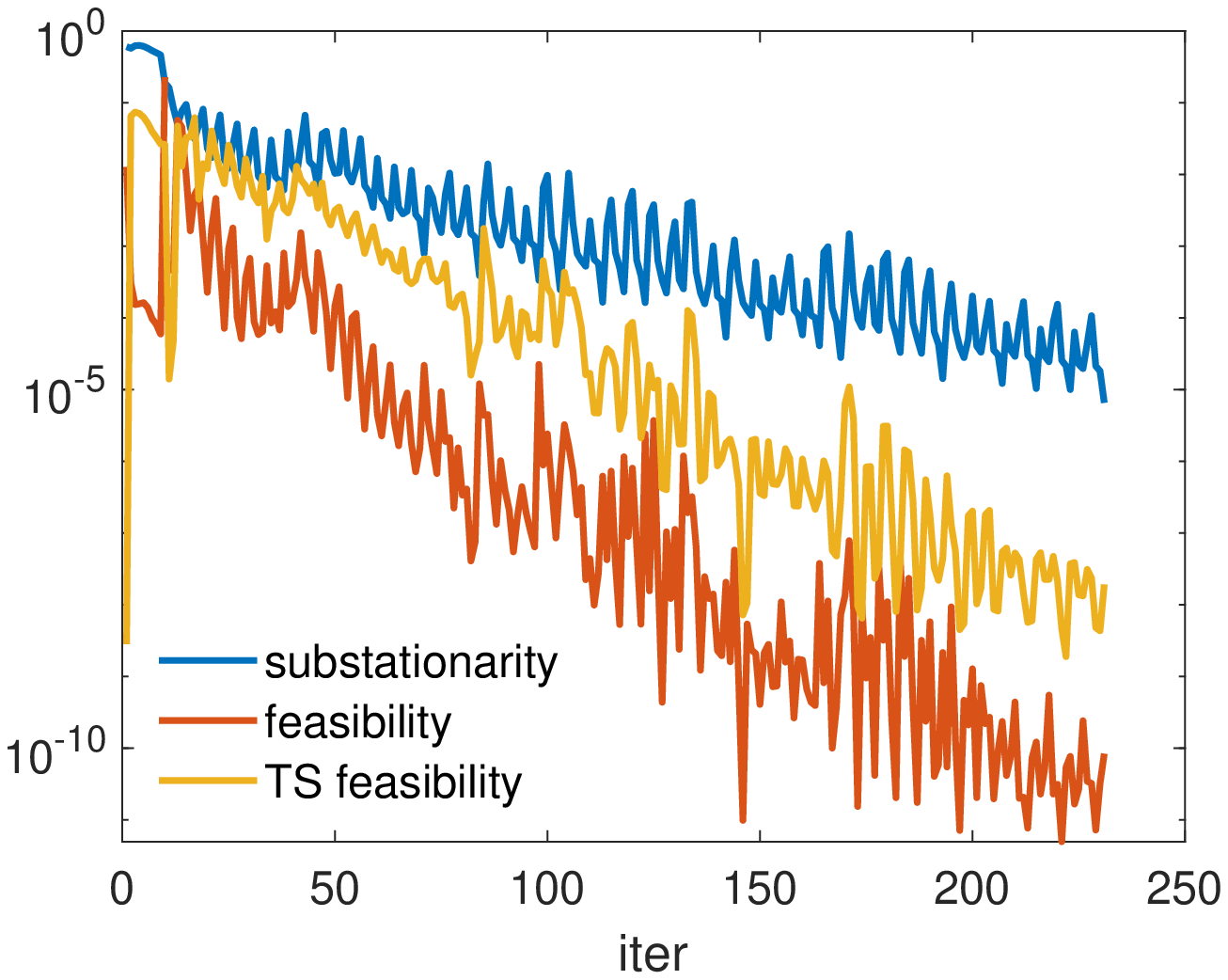}
				\label{Fig:l21pca}
			\end{minipage}%
		}%
		\subfigure[Problem \ref{Example_l21PCA}, $(p,\gamma) = (5,0.001)$]{
			\begin{minipage}[t]{0.33\linewidth}
				\centering
				\includegraphics[width=\linewidth]{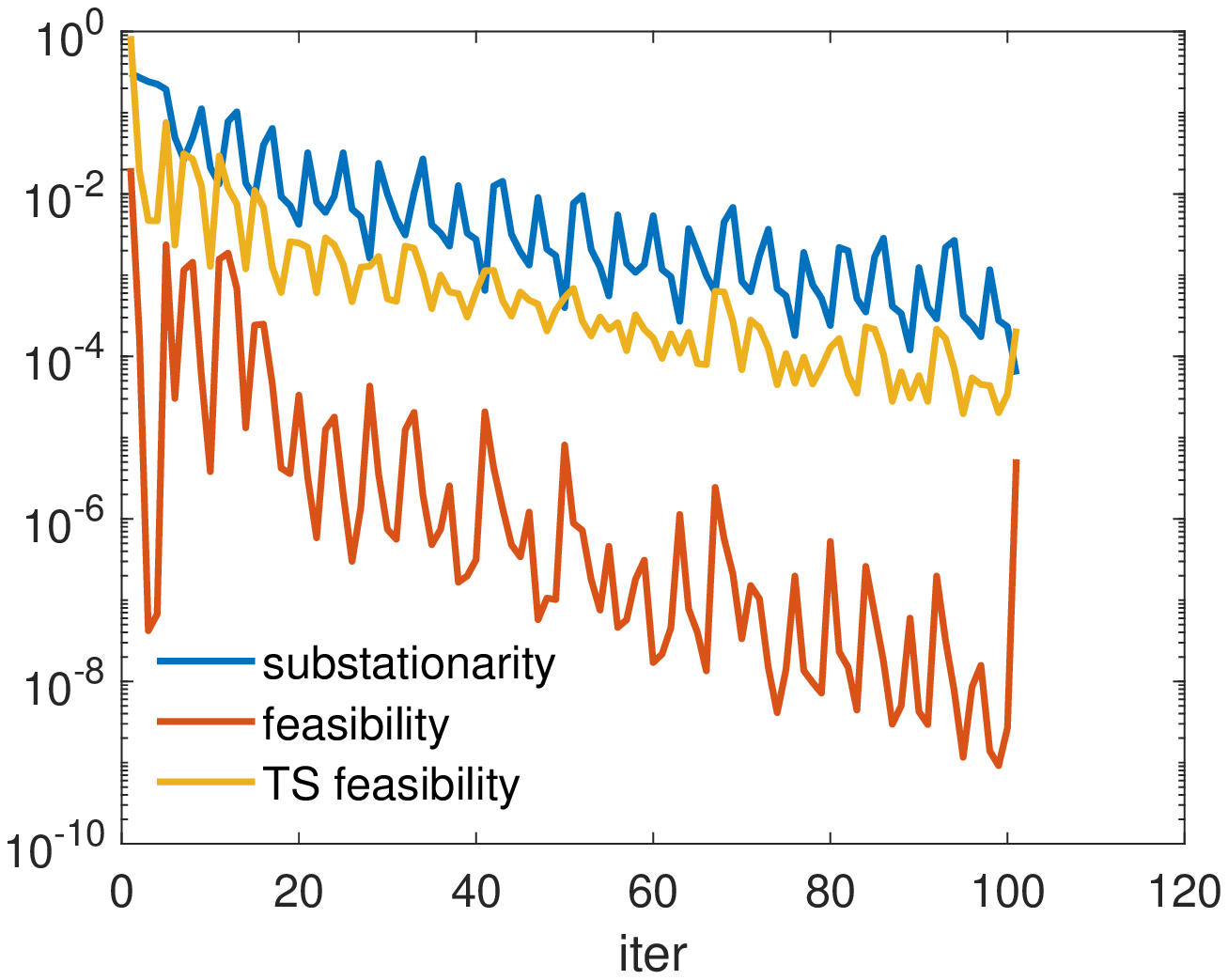}
				\label{Fig:l1pca}
			\end{minipage}%
		}%
		\subfigure[Problem \ref{Example_KS}, $(n,p) = (2013,7)$]{
			\begin{minipage}[t]{0.33\linewidth}
				\centering
				\includegraphics[width=\linewidth]{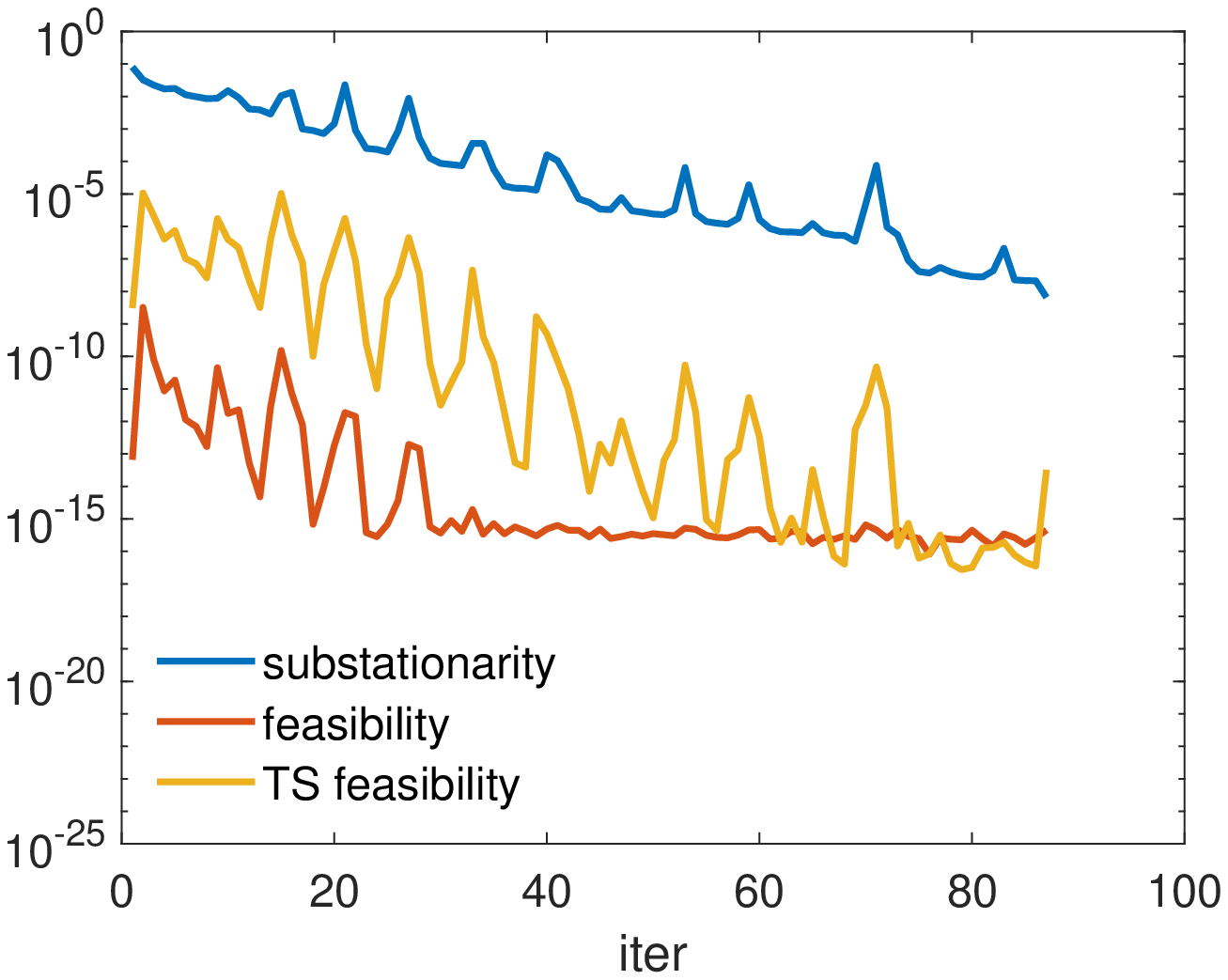}
				\label{Fig:ks}
			\end{minipage}%
		}
		\caption{The substationarity, feasibility and TS feasibility of \SLPG.}
		\label{Fig_iter_feas}
	\end{figure}

	Next we investigate how the post-process of \SLPGs affects the substationarity
		by testing \SLPGs in solving Problems \ref{Example_SPCA} and \ref{Example_l21PCA} with randomly generated
		data. 
		We display the substationarity and the feasibility of \SLPGs without the post-process, and
		the difference on the substationarity of \SLPGs after imposing the post-process as the 
		blue, red and yellow lines, respectively, in Figure \ref{Fig_iter_kkts}. 
		The problem parameters are listed below the subfigures.
		We can learn from Subfigures \ref{Fig:l21pca_kkt} and \ref{Fig:l1pca_kkt} 
		that the post-process only affects the substationarity a little. In fact, 
		the difference is a high order infinitesimal
		of the substationarity, which can partly be explained as the feasibility violation
		itself is a high order infinitesimal of the substationarity.
	
	\begin{figure}[!htbp]
		\centering
		\subfigure[Problem \ref{Example_SPCA}, $(p, \gamma) = (5, 0.007)$]{
			\begin{minipage}[t]{0.45\linewidth}
				\centering
				\includegraphics[width=\linewidth, height=0.20\textheight]{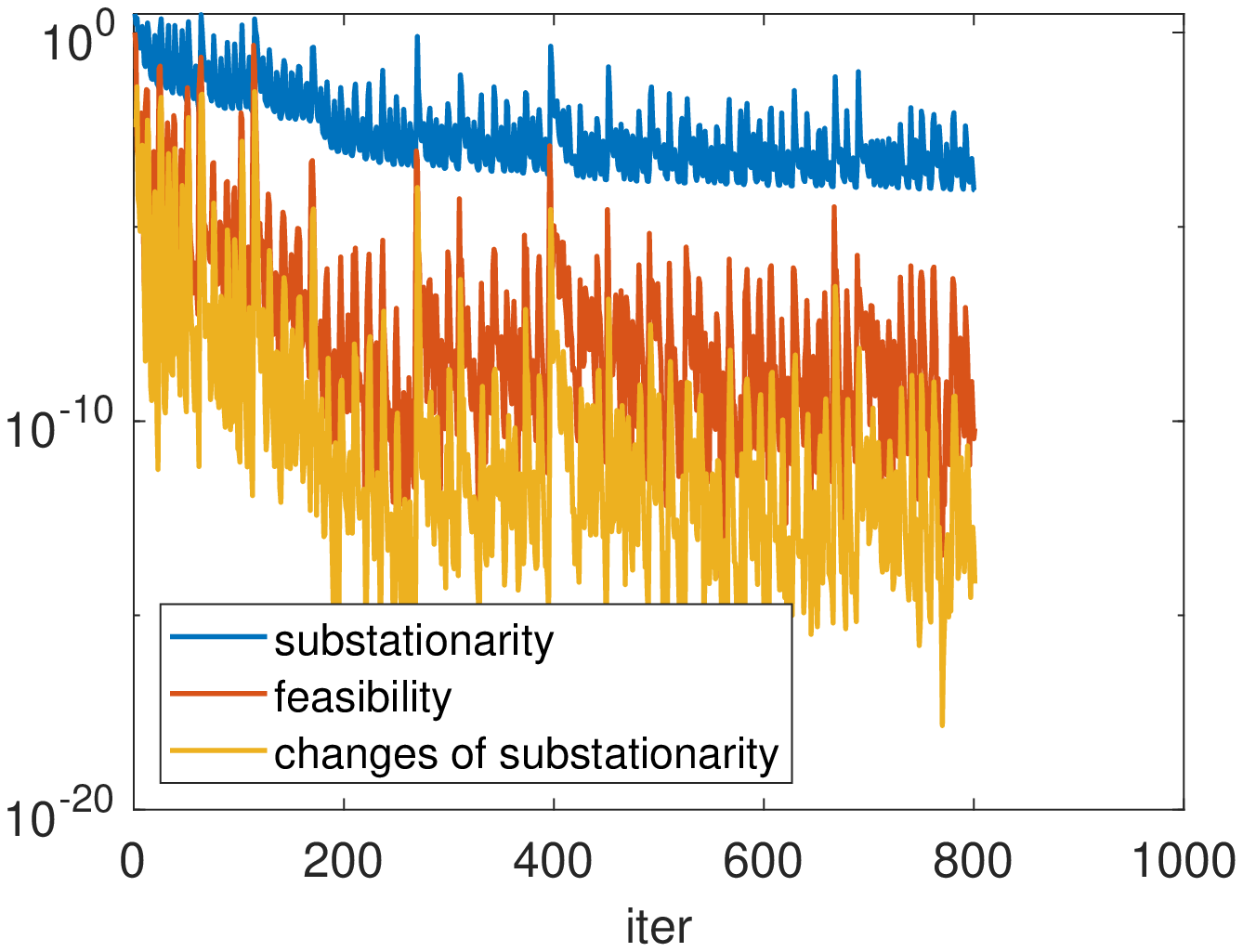}
				\label{Fig:l1pca_kkt}
			\end{minipage}%
		}%
		\subfigure[Problem \ref{Example_l21PCA}, $(p, \gamma) = (5,0.001)$]{
			\begin{minipage}[t]{0.45\linewidth}
				\centering
				\includegraphics[width=\linewidth, height=0.20\textheight]{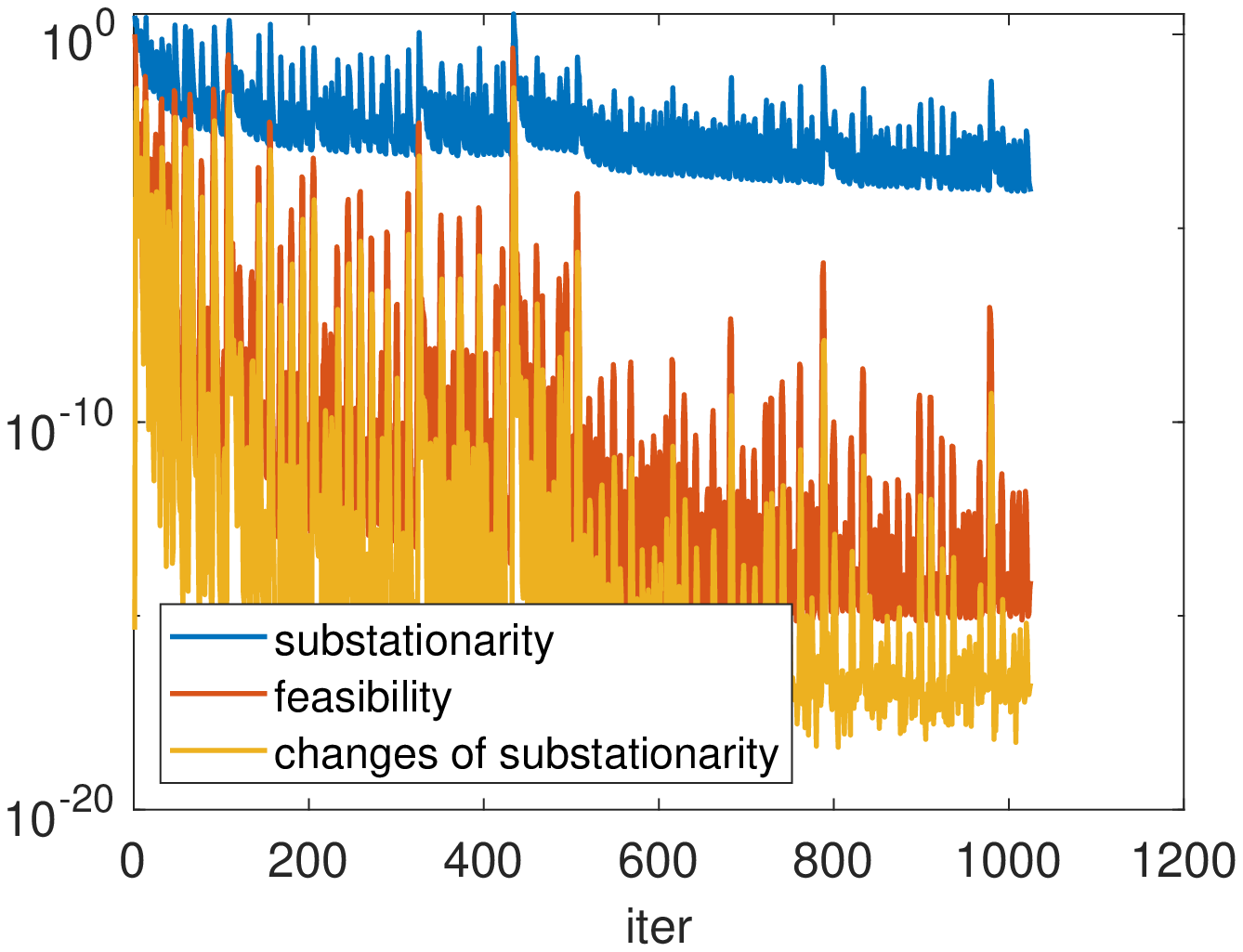}
				\label{Fig:l21pca_kkt}
			\end{minipage}%
		}%
		\caption{\revise{The effect of the post-process.}}
		\label{Fig_iter_kkts}
	\end{figure}
	
	We note that our original problem \eqref{Prob_Ori} 
		is nonconvex, hence it is expected to have multi-stationary points.
		Therefore, it is meaningful to check how the initial guesses affect the performance of
		\SLPG. 
		We generate two data sets by \eqref{eq:co} 
		for Problems \ref{Example_SPCA} and \ref{Example_l21PCA}, 
		respectively. 
		%
		Then we fix these two data sets 
		and run \SLPGs for $1000$ times with different randomly 
		generated initial points $X_0 = \mathrm{qr}(\mathrm{randn}(n,p))$
		for each problem. 
		To achieve high precision in function value, we set the stopping criteria as
		$\norm{\Yk - \Xk}\ff / \eta_k \leq 10^{-10}$ here. 
		We regard
		the function values varying in a range less than $10^{-7}$
		as one value due to the possible numerical error.
		We study the function value distribution in the $1000$ runs for each problem and put the results into 
		Figure \ref{Fig_multistart}. The problem parameters are listed below the subfigures. 
		From both Subfigures \ref{Fig:l1pca_fval} and \ref{Fig:l21pca_fval}, we can conclude that 
		\SLPGs has high probability to reach the lowest function values,
		which could be regarded as good estimates of the global minimizers of Problems \ref{Example_SPCA} and \ref{Example_l21PCA}, respectively, 
		with certain probability.
	\begin{figure}[!htbp]
		\centering
		\subfigure[Problem \ref{Example_SPCA}, $(p,\gamma) = (5,0.05)$]{
			\begin{minipage}[t]{0.45\linewidth}
				\centering
				\includegraphics[width=\linewidth, height=0.20\textheight]{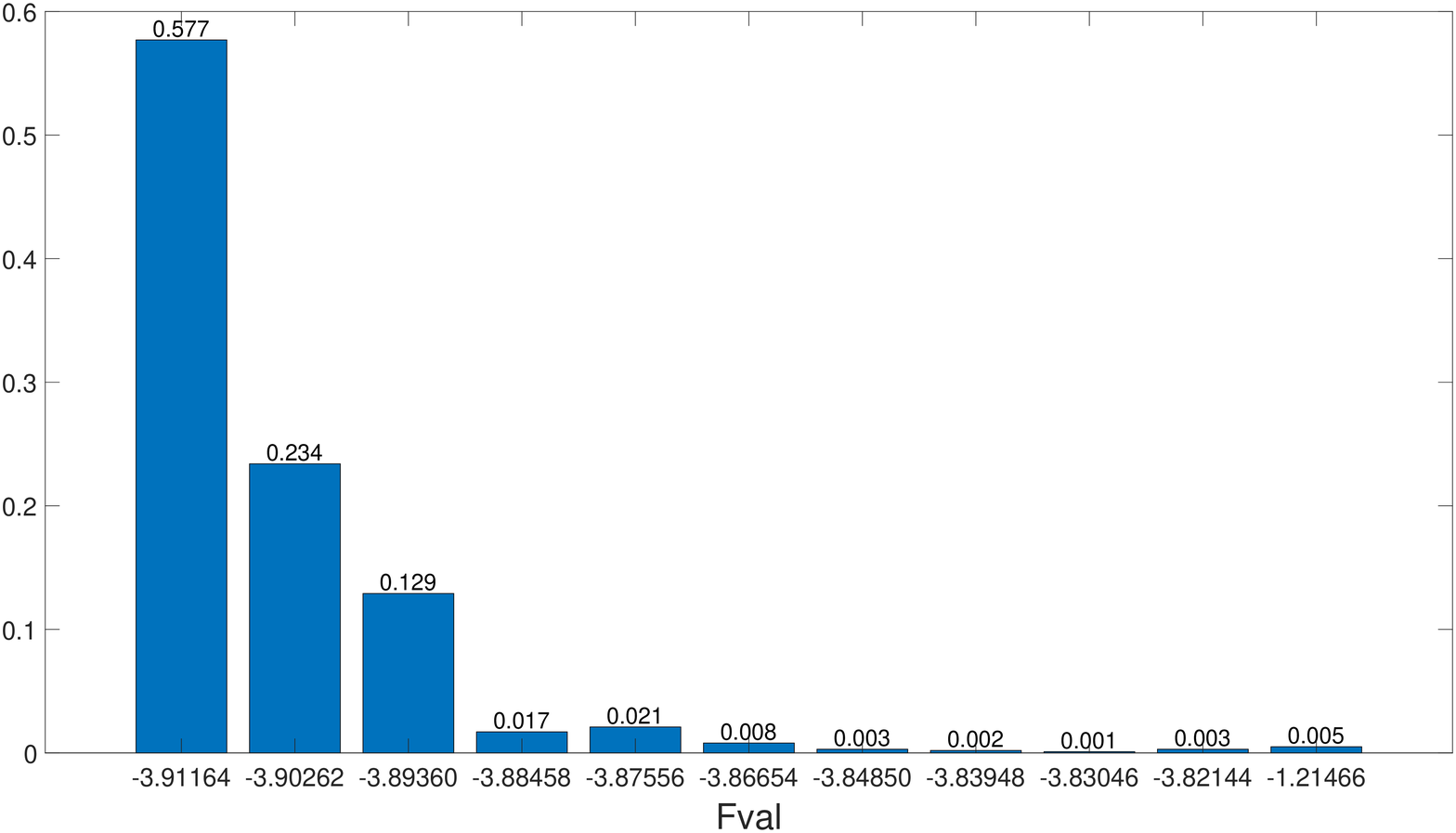}
				\label{Fig:l1pca_fval}
			\end{minipage}%
		}%
		\subfigure[Problem \ref{Example_l21PCA}, $(p,\gamma) = (5,0.001)$]{
			\begin{minipage}[t]{0.45\linewidth}
				\centering
				\includegraphics[width=\linewidth, height=0.20\textheight]{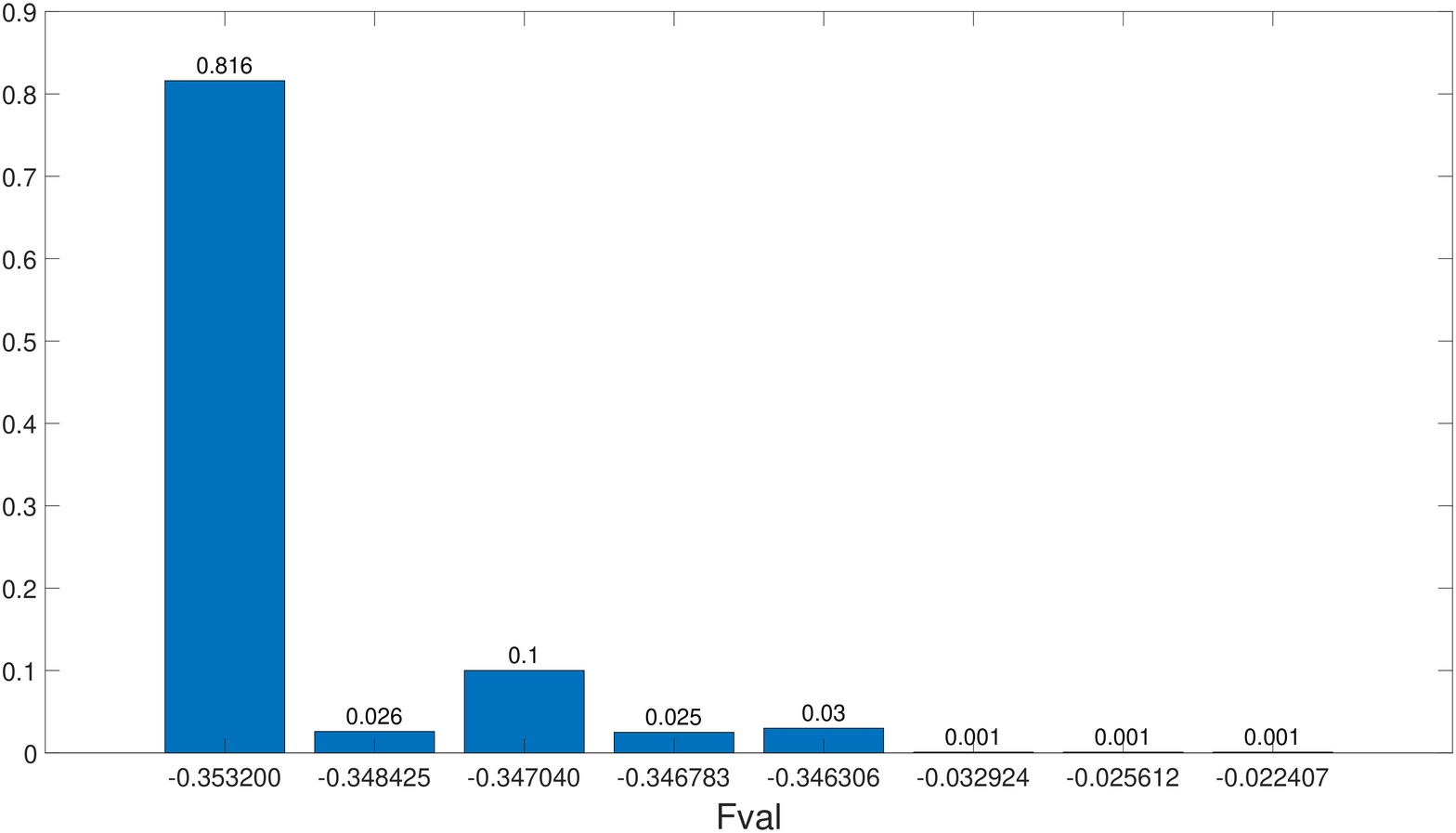}
				\label{Fig:l21pca_fval}
			\end{minipage}%
		}%
		\caption{Function value distributions of \SLPGs.}
		\label{Fig_multistart}
	\end{figure}

		Finally, we investigate how the inner solver affects the overall performance of \SLPG.
		We compare our fixed point iteration Algorithm \ref{Alg:FPI} with the frequently used 
		semi-smooth Newton methods in solving 
		Problem \ref{Example_SPCA}. 
		In our numerical examples, \SLPGs refers to \SLPGs where the subproblem is solved by Algorithm \ref{Alg:FPI} while SLPG+SSN refer to the algorithm where the the subproblems are solved by semi-smooth Newton methods. The parameters of the 
		semi-smooth Newton method 
		adopted in SLPG+SSN are fixed as its default setting as 
		stated in \cite{huang2019extending}. Figure \ref{Fig_subproblems} illustrates the performance of \SLPGs and SLPG+SSN under different column size with $(n,\gamma)$ fixed as $(4000, 0.07)$. From subfigures \ref{Fig:subamanpg2}-\ref{Fig:subamanpg3} we can learn that \SLPGs 
		reaches the same function value in same number of iterations, but requires
		slightly less CPU time than SLPG+SSN. That is the reason
		we use Algorithm \ref{Alg:FPI} as the default inner 
		solver in \SLPG.

	\begin{figure}[!htbp]
		\centering
		\subfigure[$(n,\gamma) = (4000,0.007)$]{
			\begin{minipage}[t]{0.33\linewidth}
				\centering
				\includegraphics[width=\linewidth, height=0.15\textheight]{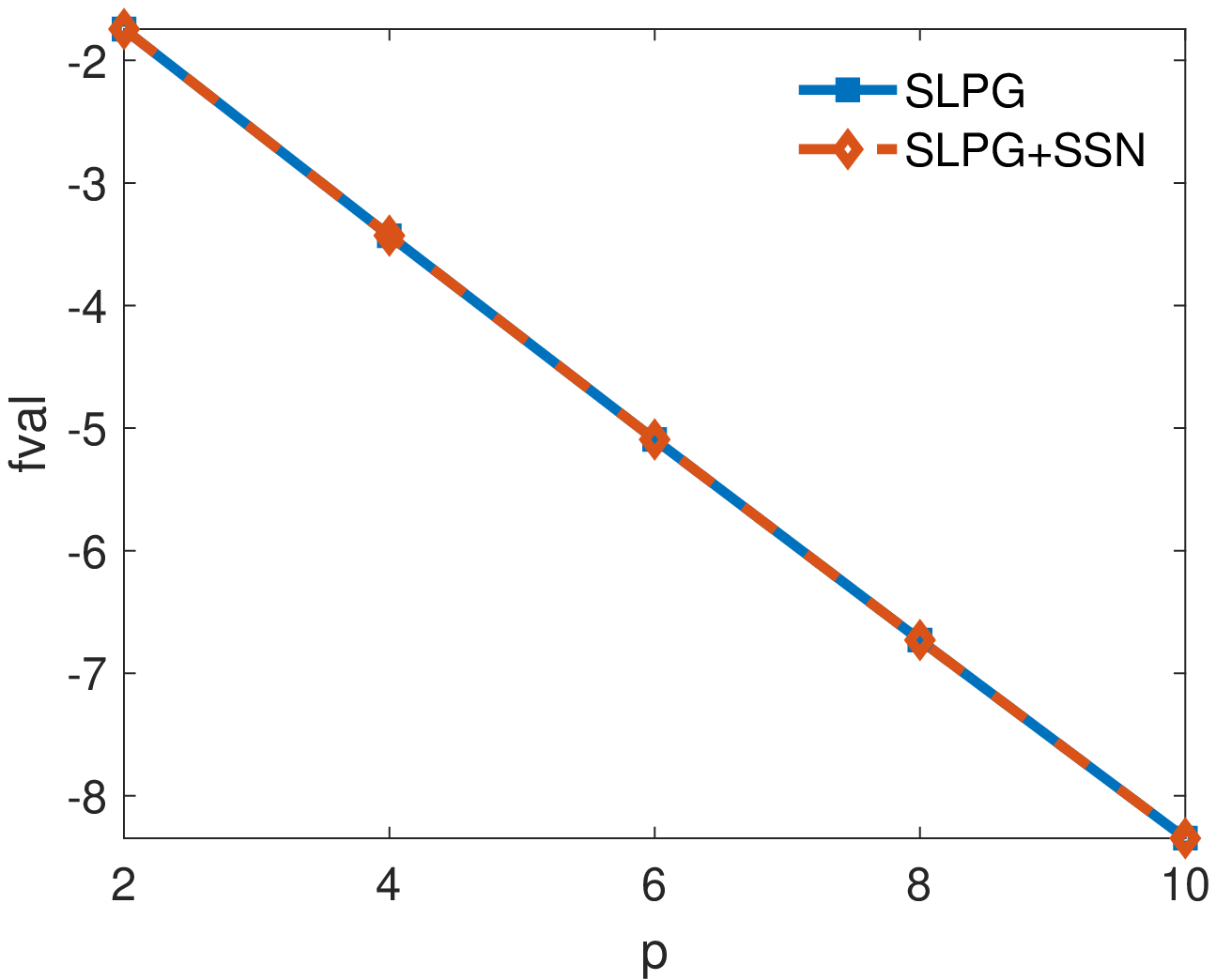}
				\label{Fig:subamanpg1}
			\end{minipage}%
		}%
		\subfigure[$(n,\gamma) = (4000,0.007)$]{
			\begin{minipage}[t]{0.33\linewidth}
				\centering
				\includegraphics[width=\linewidth, height=0.15\textheight]{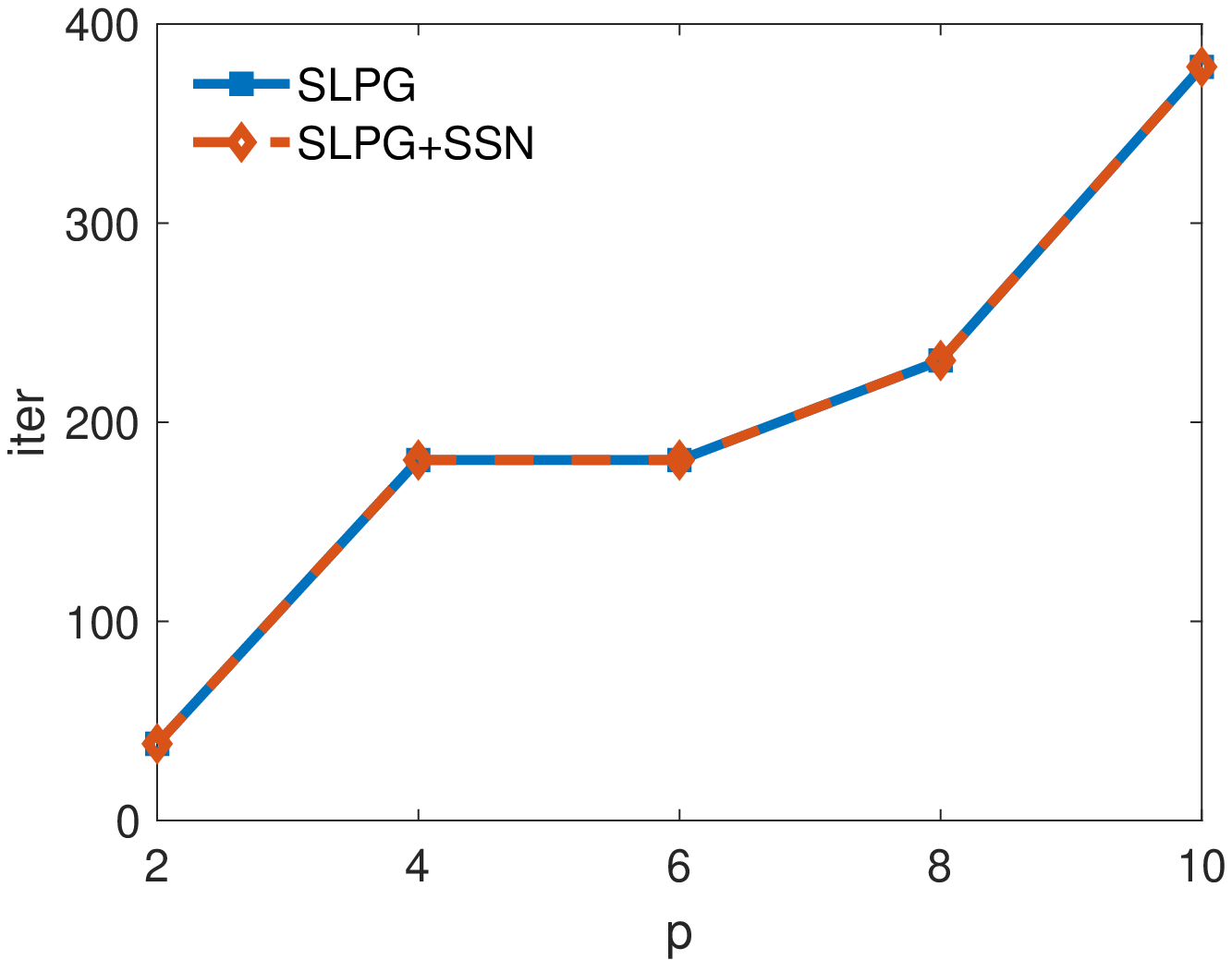}
				\label{Fig:subamanpg2}
			\end{minipage}%
		}%
		\subfigure[$(n,\gamma) = (4000,0.007)$]{
			\begin{minipage}[t]{0.33\linewidth}
				\centering
				\includegraphics[width=\linewidth, height=0.15\textheight]{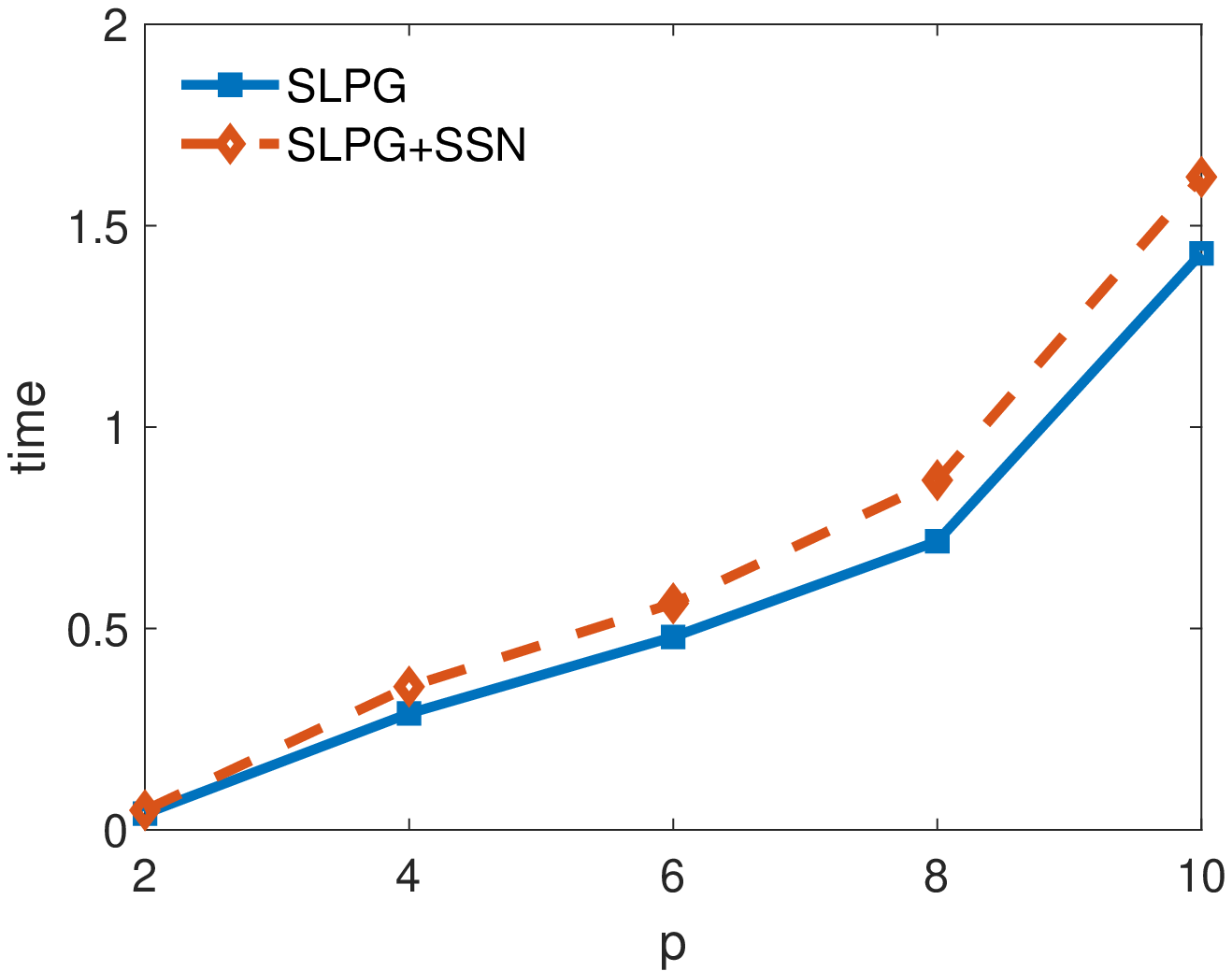}
				\label{Fig:subamanpg3}
			\end{minipage}%
		}%
		%
		\caption{The comparison of inner solver.}
		\label{Fig_subproblems}
	\end{figure}

	\subsection{$\ell_{2,1}$-norm regularized PCA}
	\label{subsection_numerical_l21}
	In this subsection, we first compare \SLPGs with some of the state-of-the-art algorithms including ManPG-Ada and PenCPG. Then, we further investigate the 
	robustness brought by our penalty-free scheme. 
	The first algorithm in comparison is ManPG-Ada, which is an accelerated version of ManPG \cite{chen2019manifold}. The second one is PenCPG, which is an infeasible proximal gradient method based on the closed-form expression of the multipliers. In our experiments, 
	all three algorithms are run in their default settings.
	As suggested in \cite{gao2017sparse}, 
	the penalty parameter of Problem \ref{Example_l21PCA}
	is set as $\gamma = b\sqrt{p+\log(n)}$, where parameter $b$ 
	is used to control the sparsity.
	
	Figure \ref{Fig_l21PCA_Coefficients} illustrates the performance of the  
	three algorithms in comparison in solving Problem \ref{Example_l21PCA}
	with different combinations of $n$, $p$, $b$. The detailed problem
	parameters are listed below the subfigures.
	As illustrated in Figure \ref{Fig_l21PCA_Coefficients}, 
	all of these three algorithm reach the same function values.
	\SLPGs takes fewer iterations than the other two, meanwhile
	it takes much less CPU time than ManPG-Ada.
	Since PenCPG does not have any subproblem to solve, it has the lowest computational
	cost in each iteration among the three. 
	Finally, it only takes slightly less CPU time than \SLPGs.
	We can conclude that \SLPGs is superior to the other two algorithms
	in the testing problems.

	\begin{figure}[!htbp]
		\centering
		\subfigure[$(p,b) = (4,0.1)$]{
			\begin{minipage}[t]{0.33\linewidth}
				\centering
				\includegraphics[width=\linewidth,height=0.18\textheight]{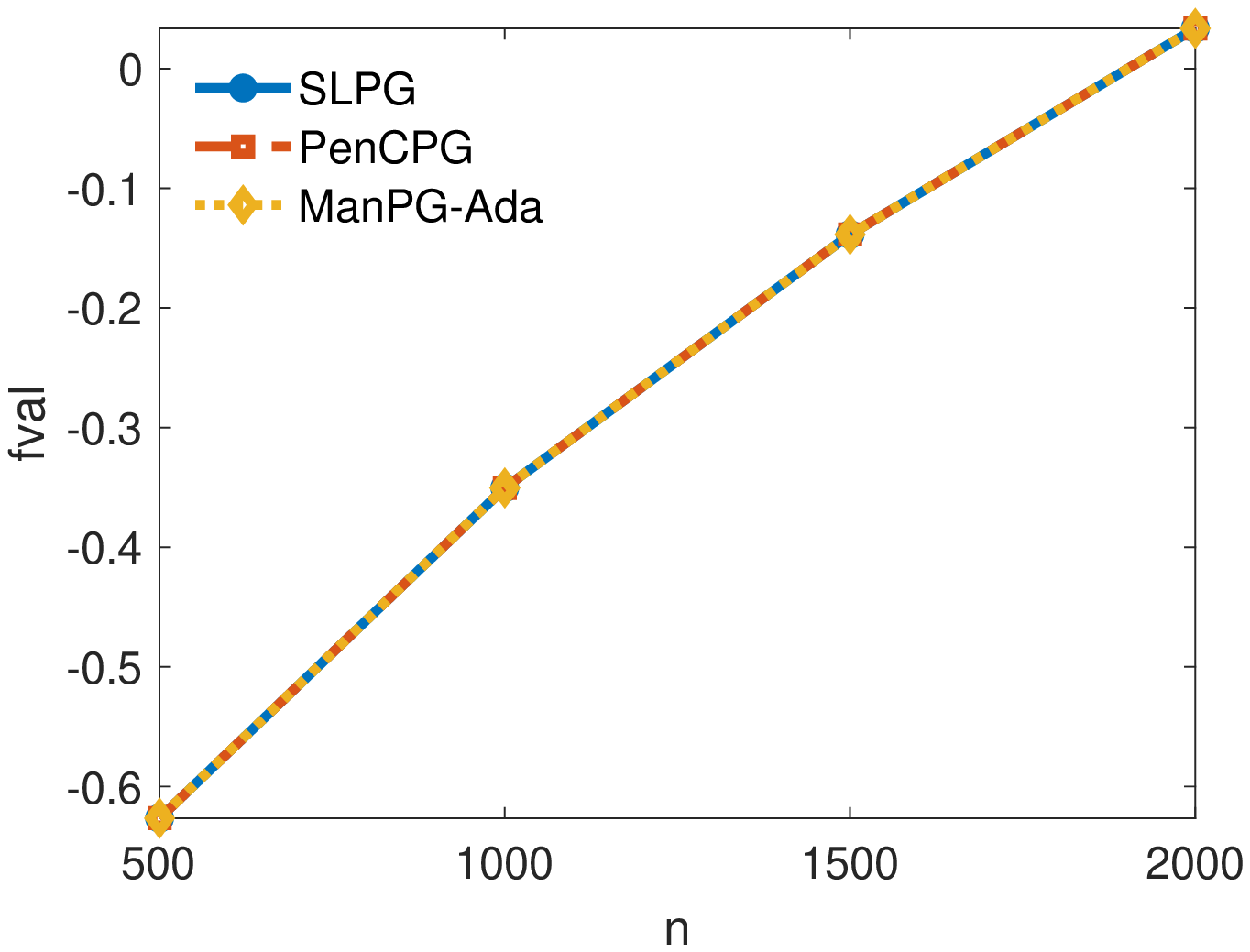}
				\label{Fig:l21PCA_Prob1_item1}
			\end{minipage}%
		}%
		\subfigure[$(p,b) = (4,0.1)$]{
			\begin{minipage}[t]{0.33\linewidth}
				\centering
				\includegraphics[width=\linewidth,height=0.18\textheight]{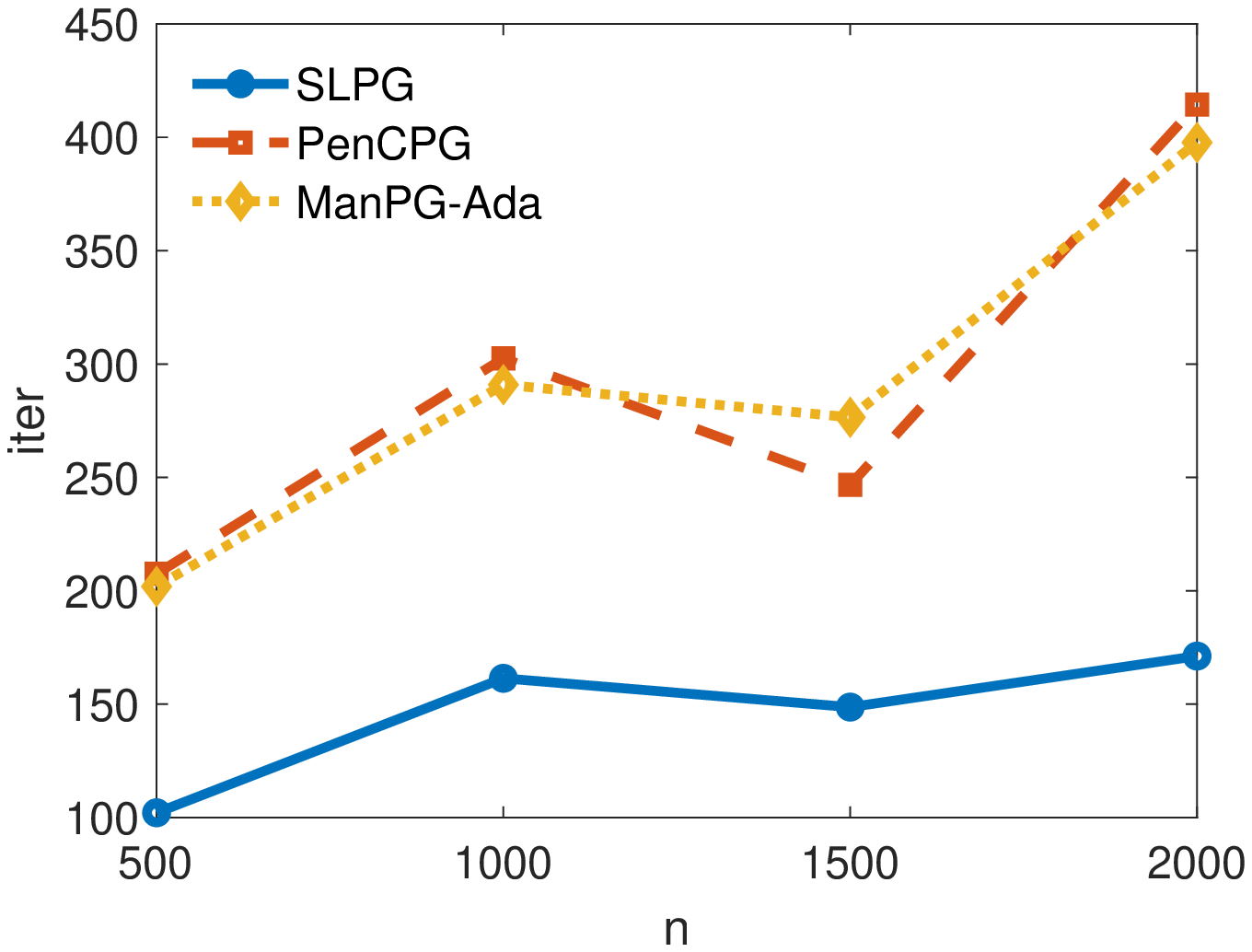}
				\label{Fig:l21PCA_Prob1_item2}
			\end{minipage}%
		}%
		\subfigure[$(p,b) = (4,0.1)$]{
			\begin{minipage}[t]{0.33\linewidth}
				\centering
				\includegraphics[width=\linewidth,height=0.18\textheight]{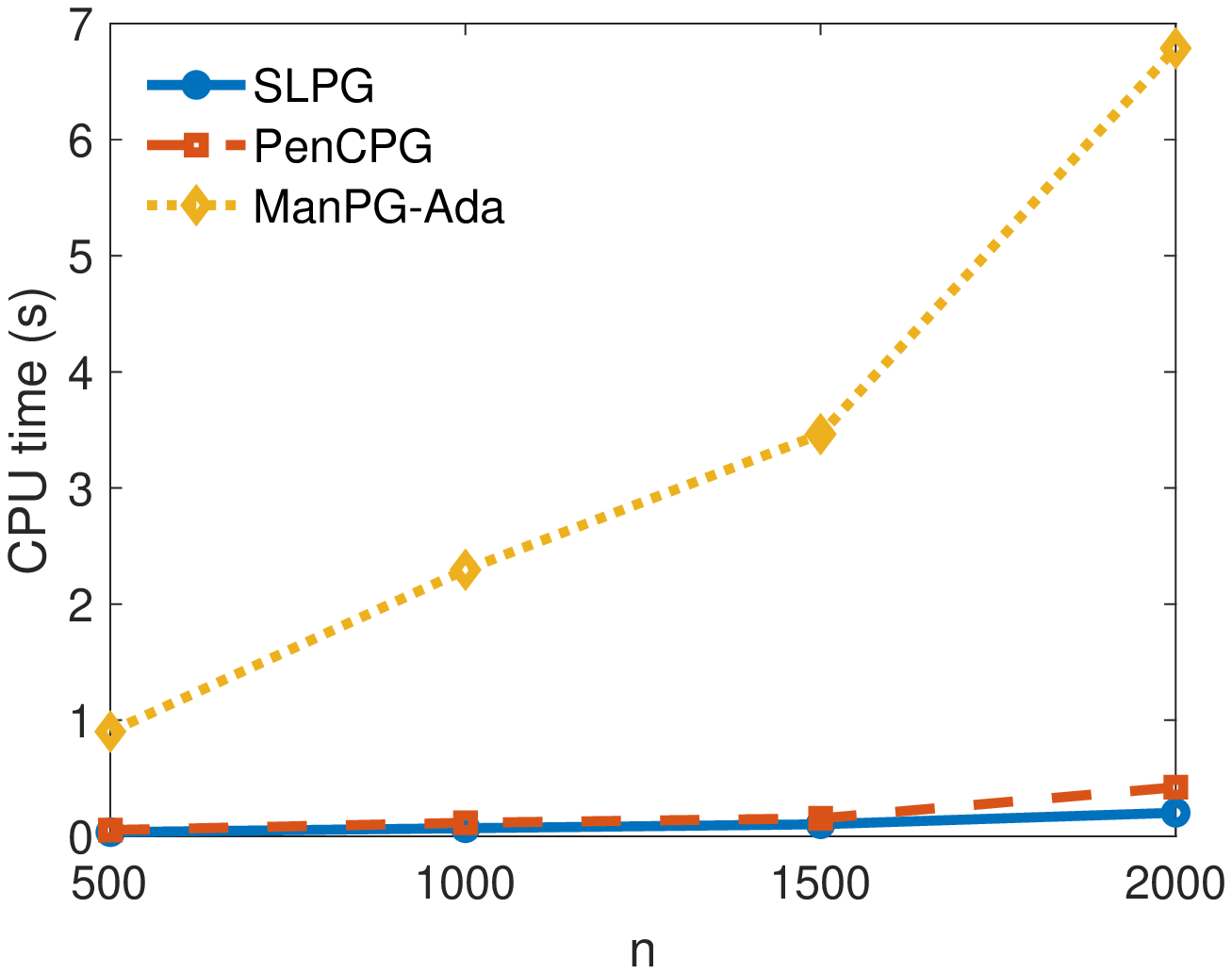}
				\label{Fig:l21PCA_Prob1_item5}
			\end{minipage}%
		}%
	
		\subfigure[$(n,b) = (1000,0.1)$]{
			\begin{minipage}[t]{0.33\linewidth}
				\centering
				\includegraphics[width=\linewidth,height=0.18\textheight]{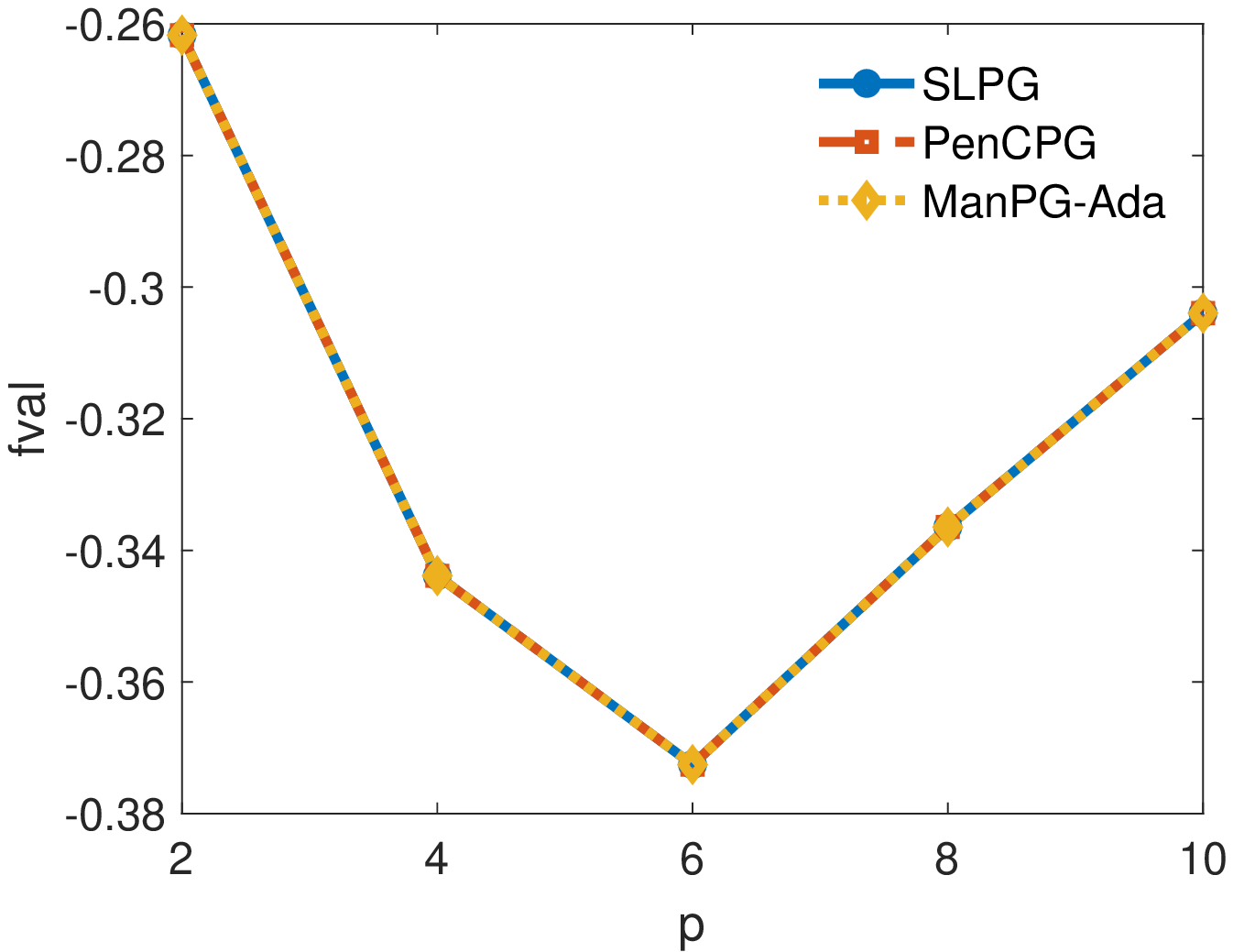}
				\label{Fig:l21PCA_Prob2_item1}
			\end{minipage}%
		}%
		\subfigure[$(n,b) = (1000,0.1)$]{
			\begin{minipage}[t]{0.33\linewidth}
				\centering
				\includegraphics[width=\linewidth,height=0.18\textheight]{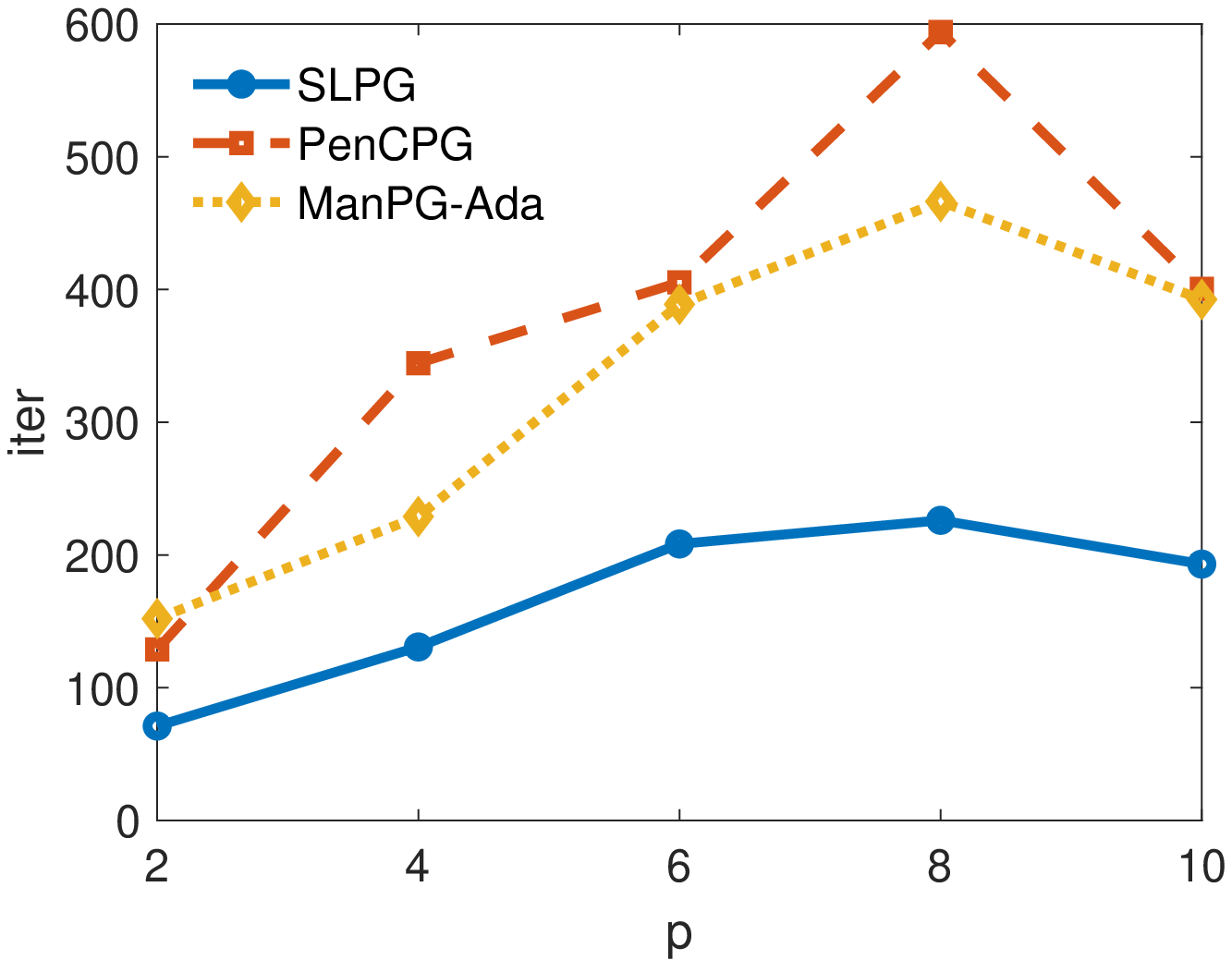}
				\label{Fig:l21PCA_Prob2_item2}
			\end{minipage}%
		}%
		\subfigure[$(n,b) = (1000,0.1)$]{
			\begin{minipage}[t]{0.33\linewidth}
				\centering
				\includegraphics[width=\linewidth,height=0.18\textheight]{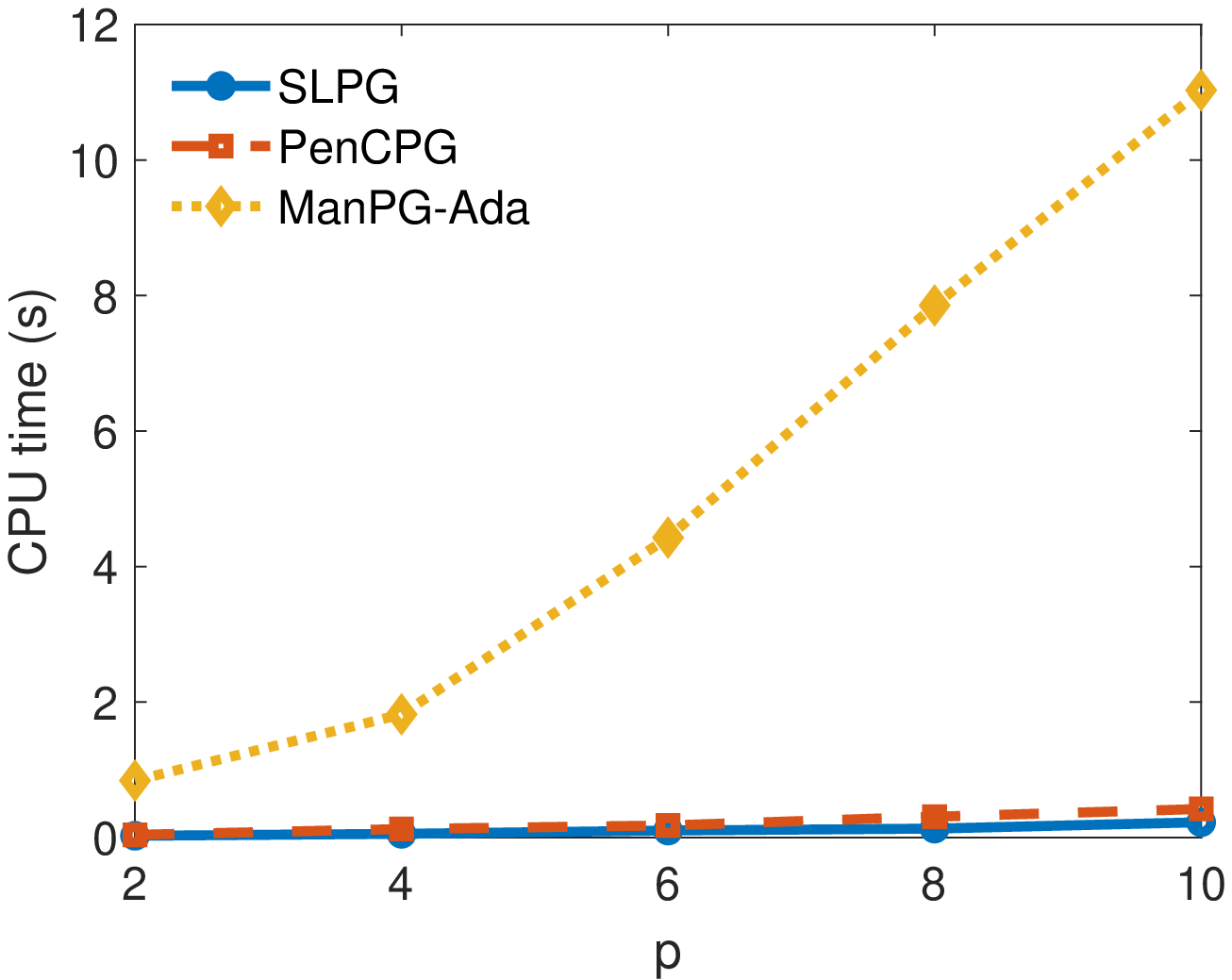}
				\label{Fig:l21PCA_Prob2_item5}
			\end{minipage}%
		}%
	
		\subfigure[$(n,p) = (1000,4)$]{
			\begin{minipage}[t]{0.33\linewidth}
				\centering
				\includegraphics[width=\linewidth,height=0.18\textheight]{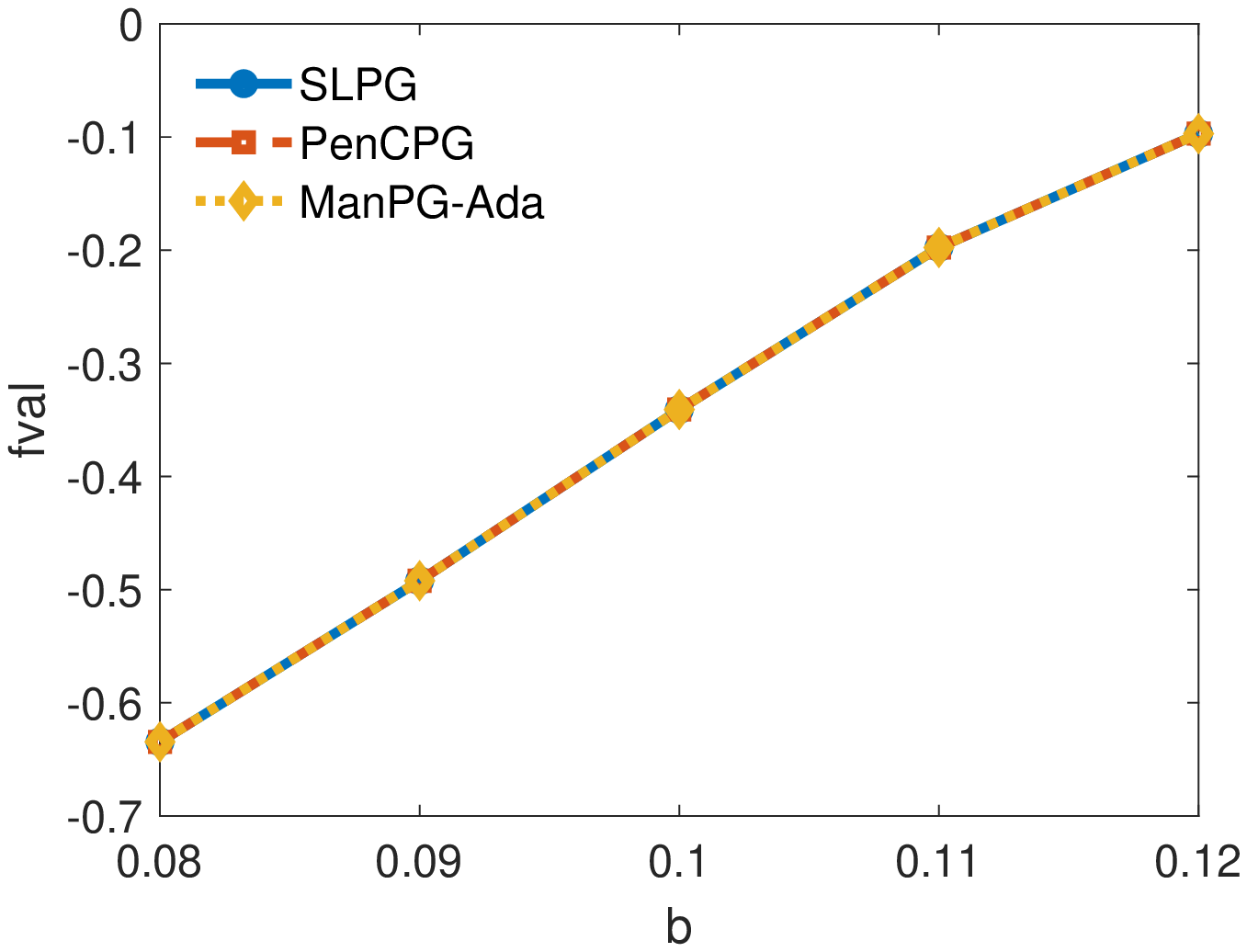}
				\label{Fig:l21PCA_Prob3_item1}
			\end{minipage}
		}%
		\subfigure[$(n,p) = (1000,4)$]{
			\begin{minipage}[t]{0.33\linewidth}
				\centering
				\includegraphics[width=\linewidth,height=0.18\textheight]{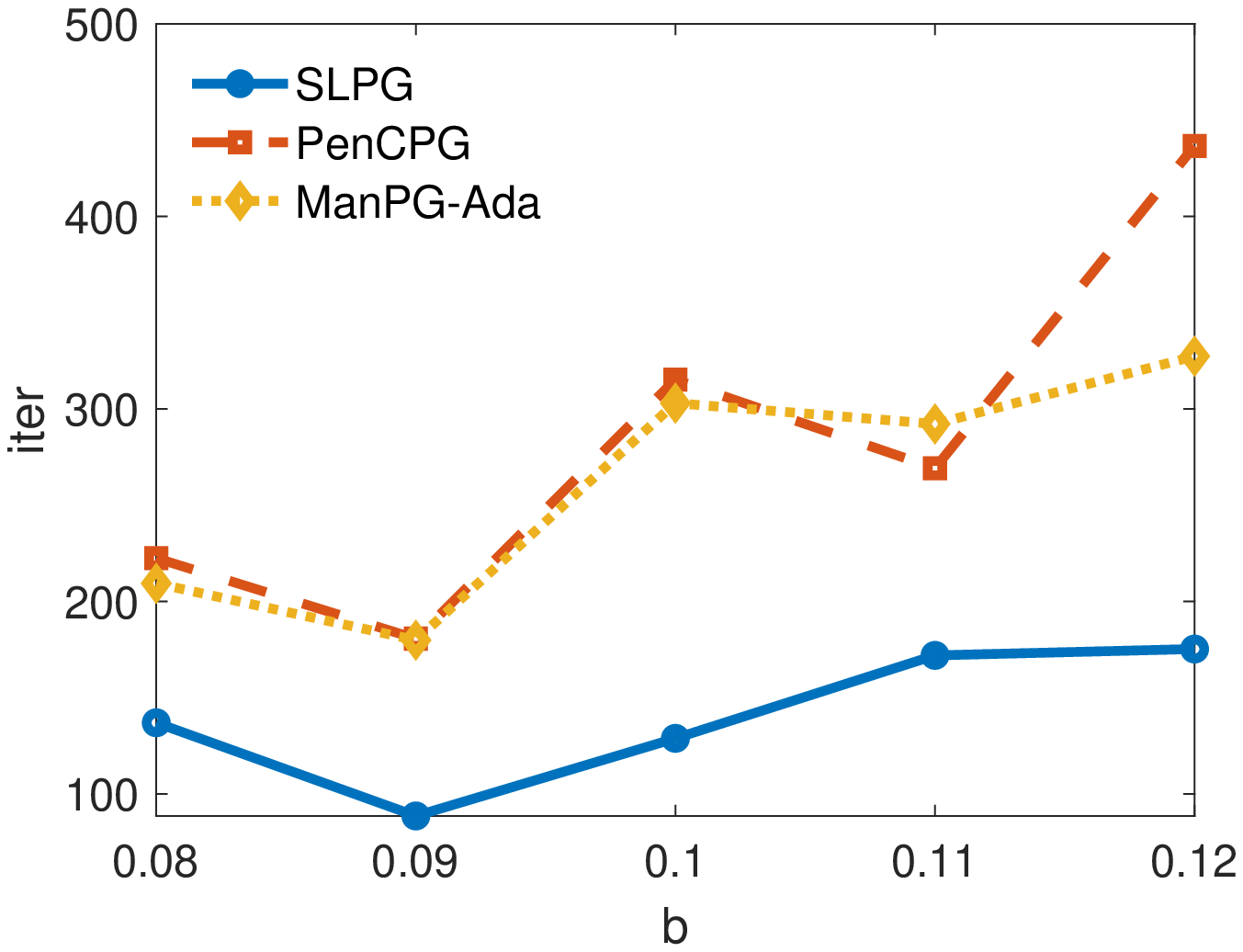}
				\label{Fig:l21PCA_Prob3_item2}
			\end{minipage}
		}%
		\subfigure[$(n,p) = (1000,4)$]{
			\begin{minipage}[t]{0.33\linewidth}
				\centering
				\includegraphics[width=\linewidth,height=0.18\textheight]{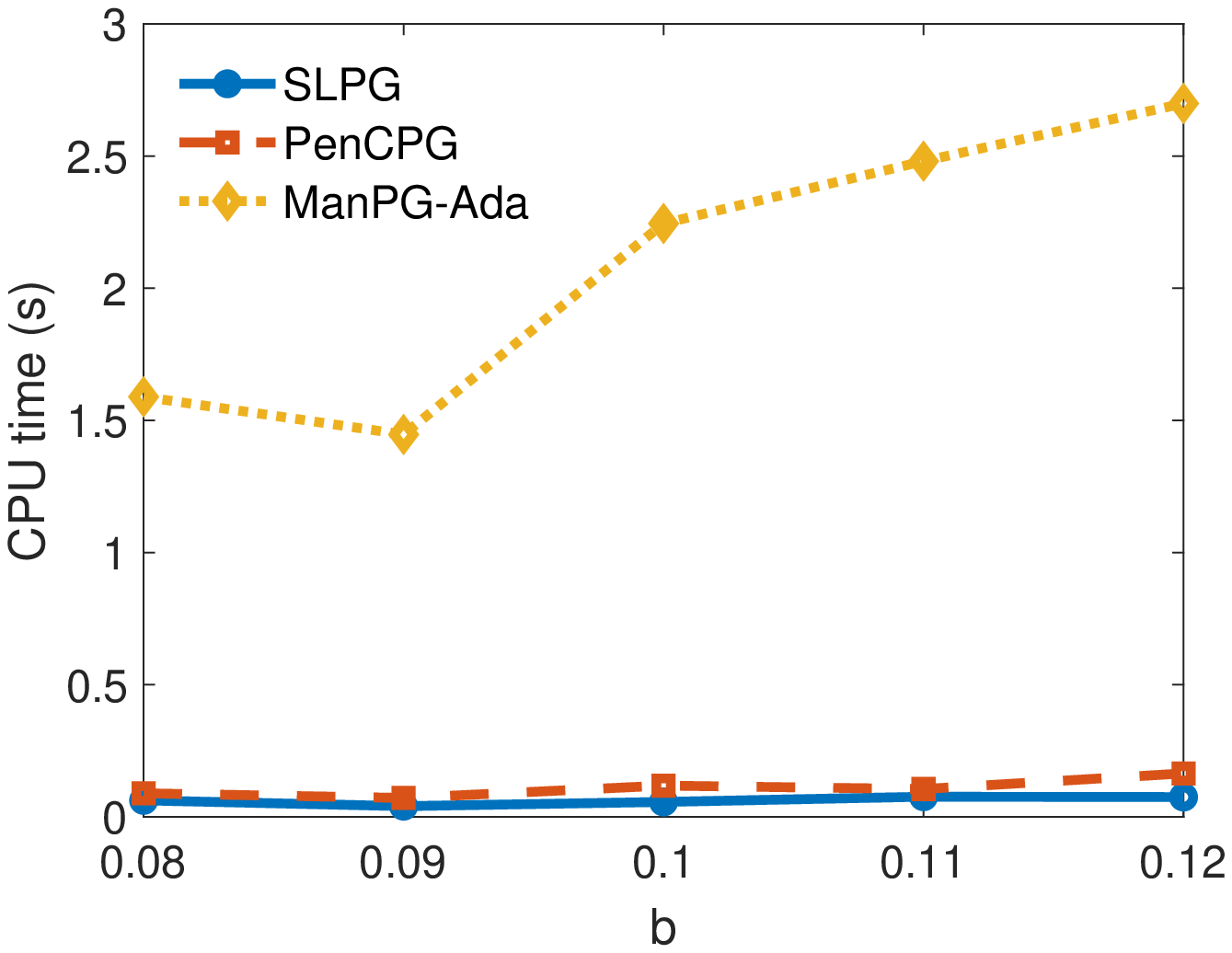}
				\label{Fig:l21PCA_Prob3_item5}
			\end{minipage}
		}%

		\caption{A comparison among \SLPG, PenCPG and ManPG-Ada in solving $\ell_{2,1}$-norm regularized PCA problems.}
		\label{Fig_l21PCA_Coefficients}
	\end{figure}

	From the above experiment, we notice that PenCPG is comparable with \SLPGs
	in the aspect of CPU time.
	However, we notice that PenCPG requires to tune a penalty parameter  $\beta$
	while \SLPGs does not have one.
	In the following experiment, we compare  \SLPGs with PenCPG equipped with 
	different choices of $\beta$. We still use Problem \ref{Example_l21PCA}
	with data set generated randomly as stated in \eqref{eq:co}.
	We present the results in Figure \ref{Fig_penalty_parameter}.
	The detailed problem settings are listed below the subfigures.
	We can learn for Figure \ref{Fig_penalty_parameter} that 
	the performance of PenCPG is sensitive to the penalty parameter,
	meanwhile, \SLPGs can always outperforms PenCPG with the best choice of 
	$\beta$.

	\begin{figure}[!htbp]
		\centering
		\subfigure[$(p, b) = (2, 0.1)$]{
			\begin{minipage}[t]{0.33\linewidth}
				\centering
				\includegraphics[width=\linewidth,height=0.18\textheight]{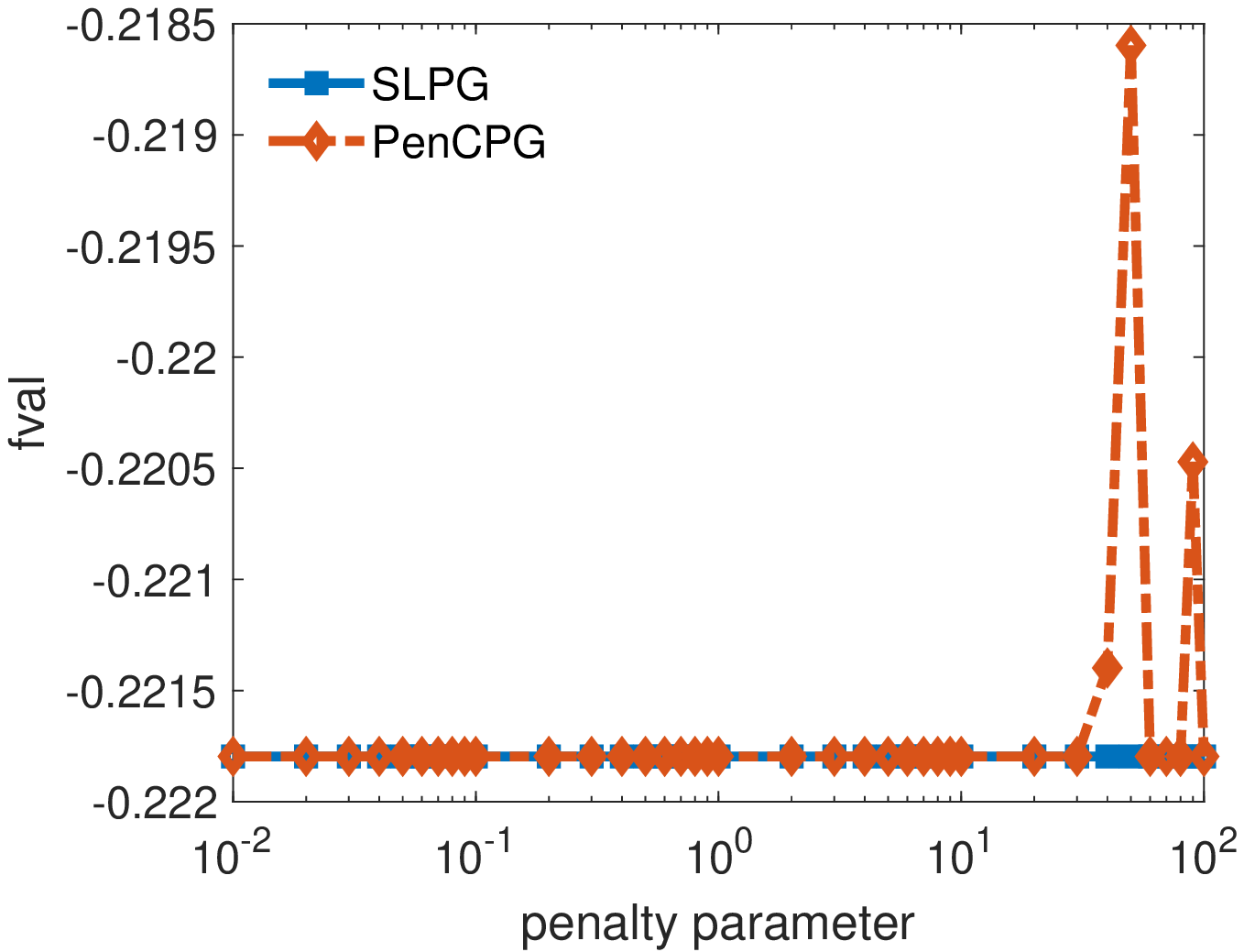}
				\label{Fig:21SPCA_Prob1_item1}
			\end{minipage}%
		}%
		\subfigure[$$(p, b) = (2, 0.1)$$]{
			\begin{minipage}[t]{0.33\linewidth}
				\centering
				\includegraphics[width=\linewidth,height=0.18\textheight]{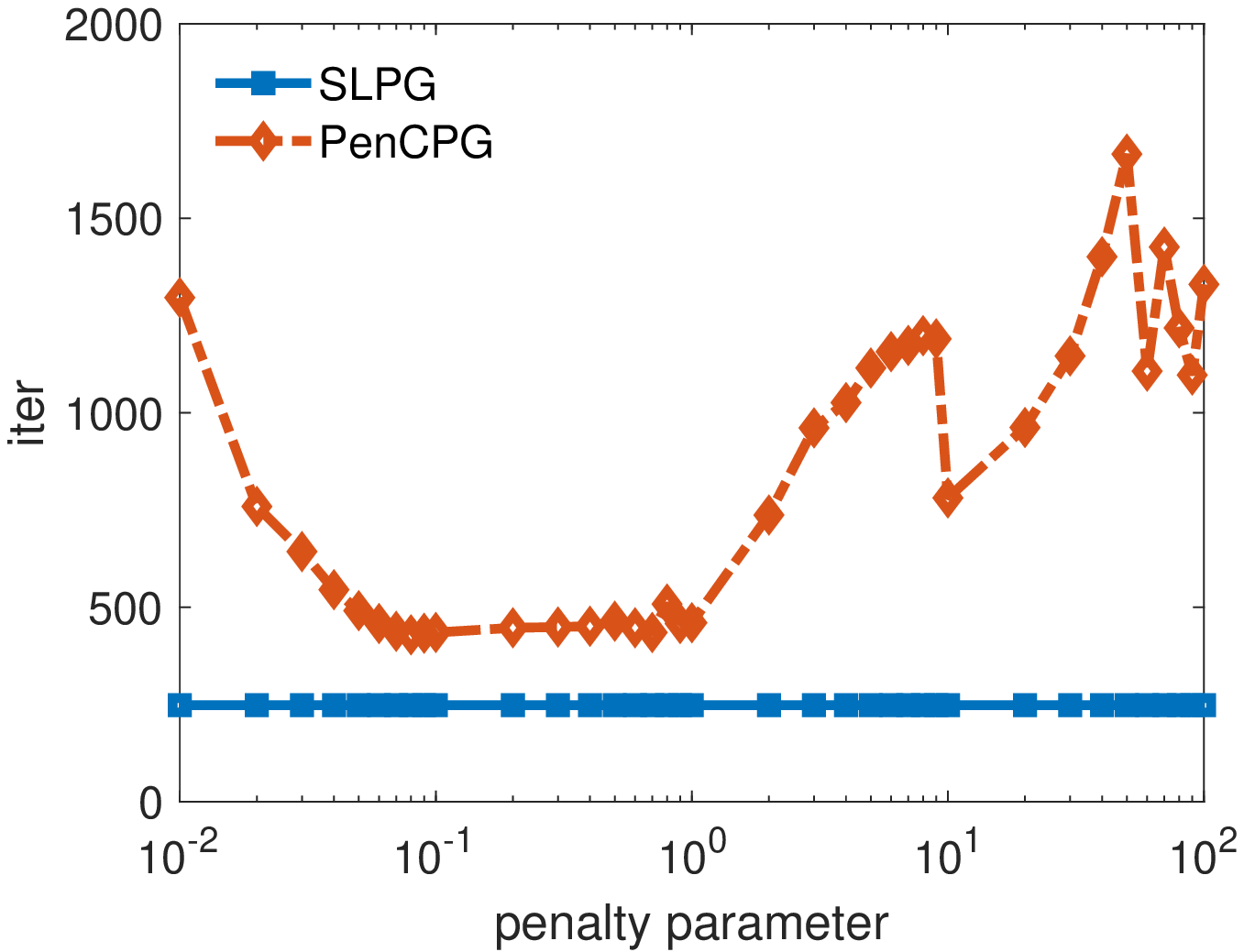}
				\label{Fig:21SPCA_Prob1_item2}
			\end{minipage}%
		}%
		\subfigure[$$(p, b) = (2, 0.1)$$]{
			\begin{minipage}[t]{0.33\linewidth}
				\centering
				\includegraphics[width=\linewidth,height=0.18\textheight]{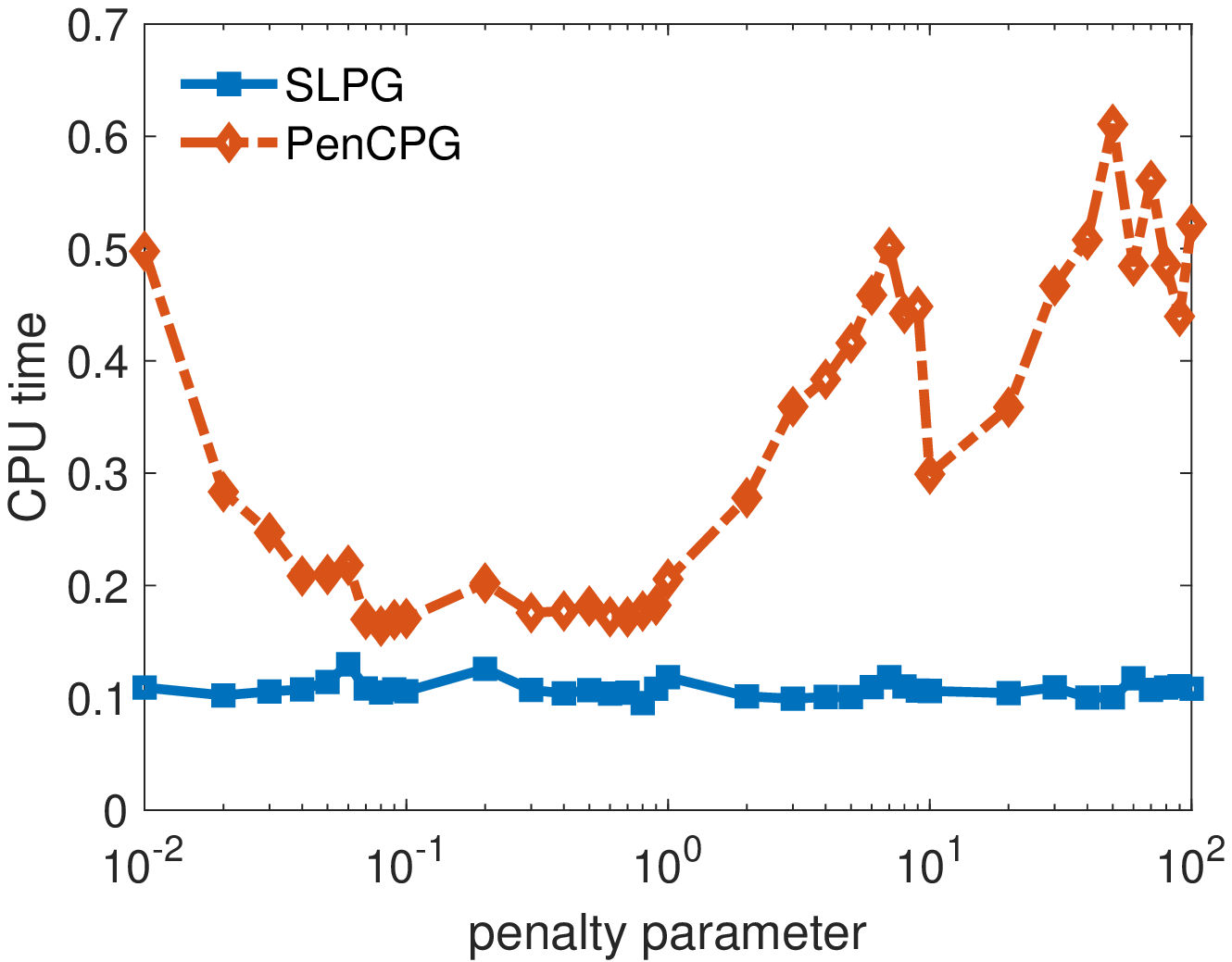}
				\label{Fig:21SPCA_Prob1_item5}
			\end{minipage}%
		}%
	
		\subfigure[$(p, b) = (5, 0.1)$]{
			\begin{minipage}[t]{0.33\linewidth}
				\centering
				\includegraphics[width=\linewidth,height=0.18\textheight]{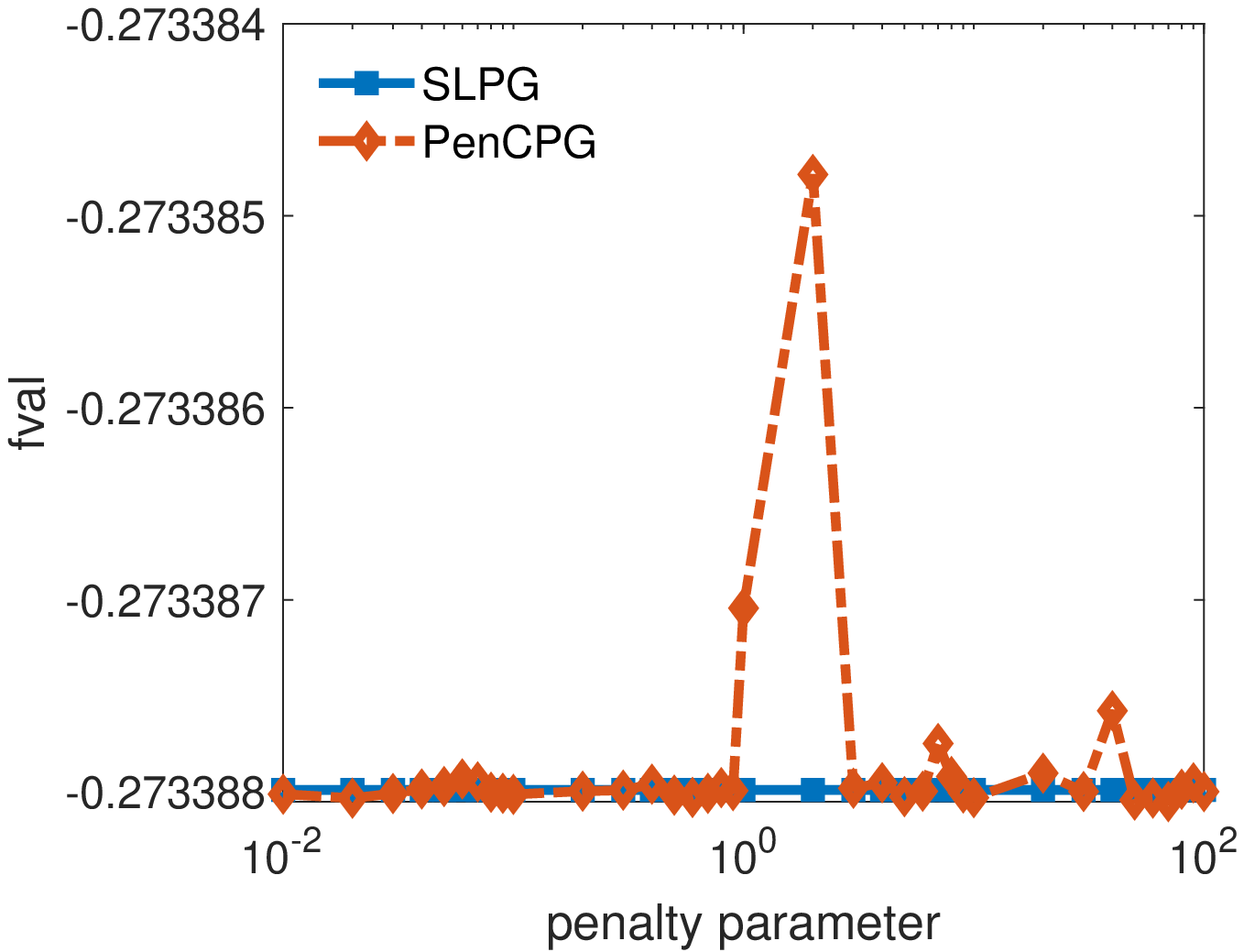}
				\label{Fig:21SPCA_Prob2_item1}
			\end{minipage}%
		}%
		\subfigure[$$(p, b) = (5, 0.1)$$]{
			\begin{minipage}[t]{0.33\linewidth}
				\centering
				\includegraphics[width=\linewidth,height=0.18\textheight]{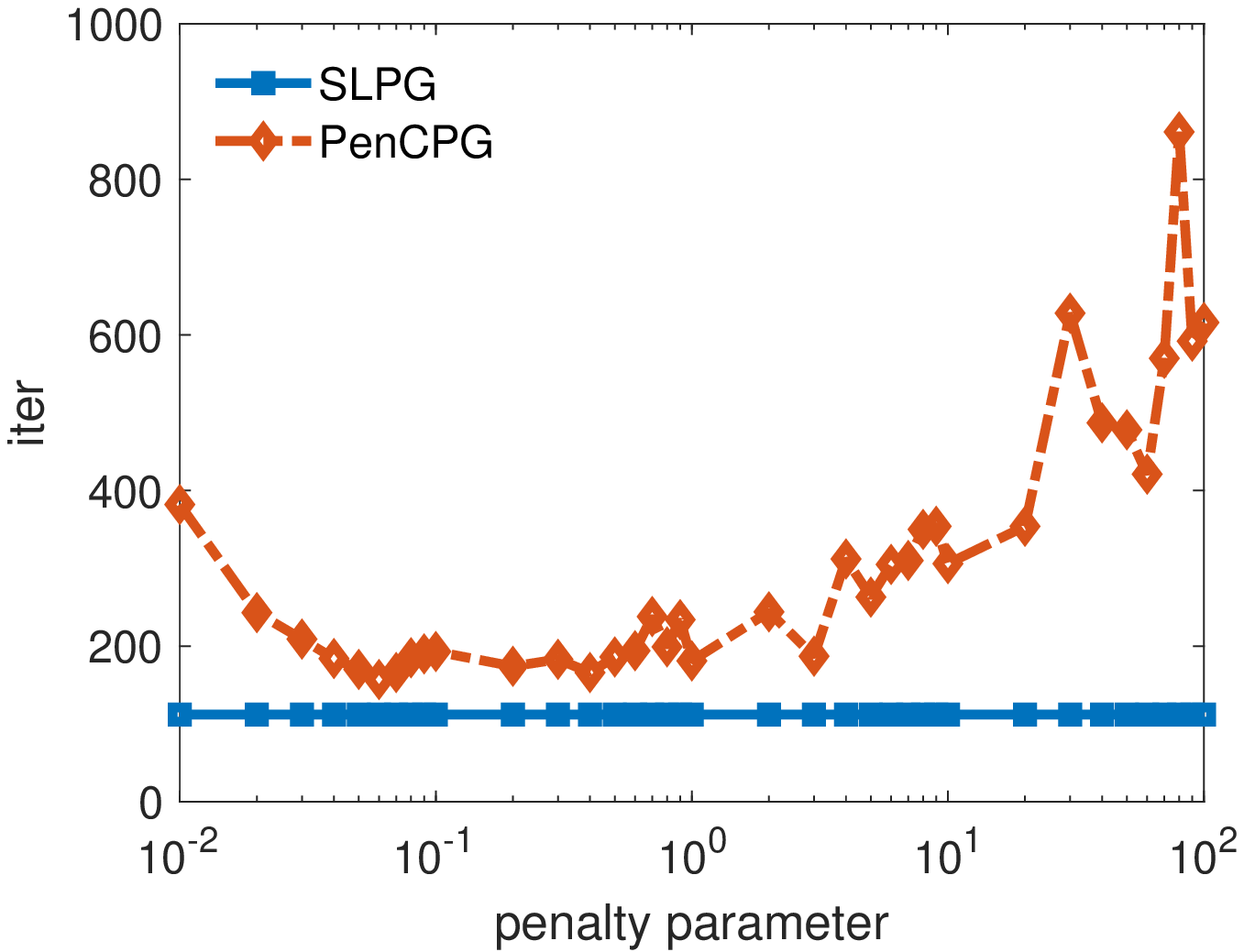}
				\label{Fig:21SPCA_Prob2_item2}
			\end{minipage}%
		}%
		\subfigure[$$(p, b) = (5, 0.1)$$]{
			\begin{minipage}[t]{0.33\linewidth}
				\centering
				\includegraphics[width=\linewidth,height=0.18\textheight]{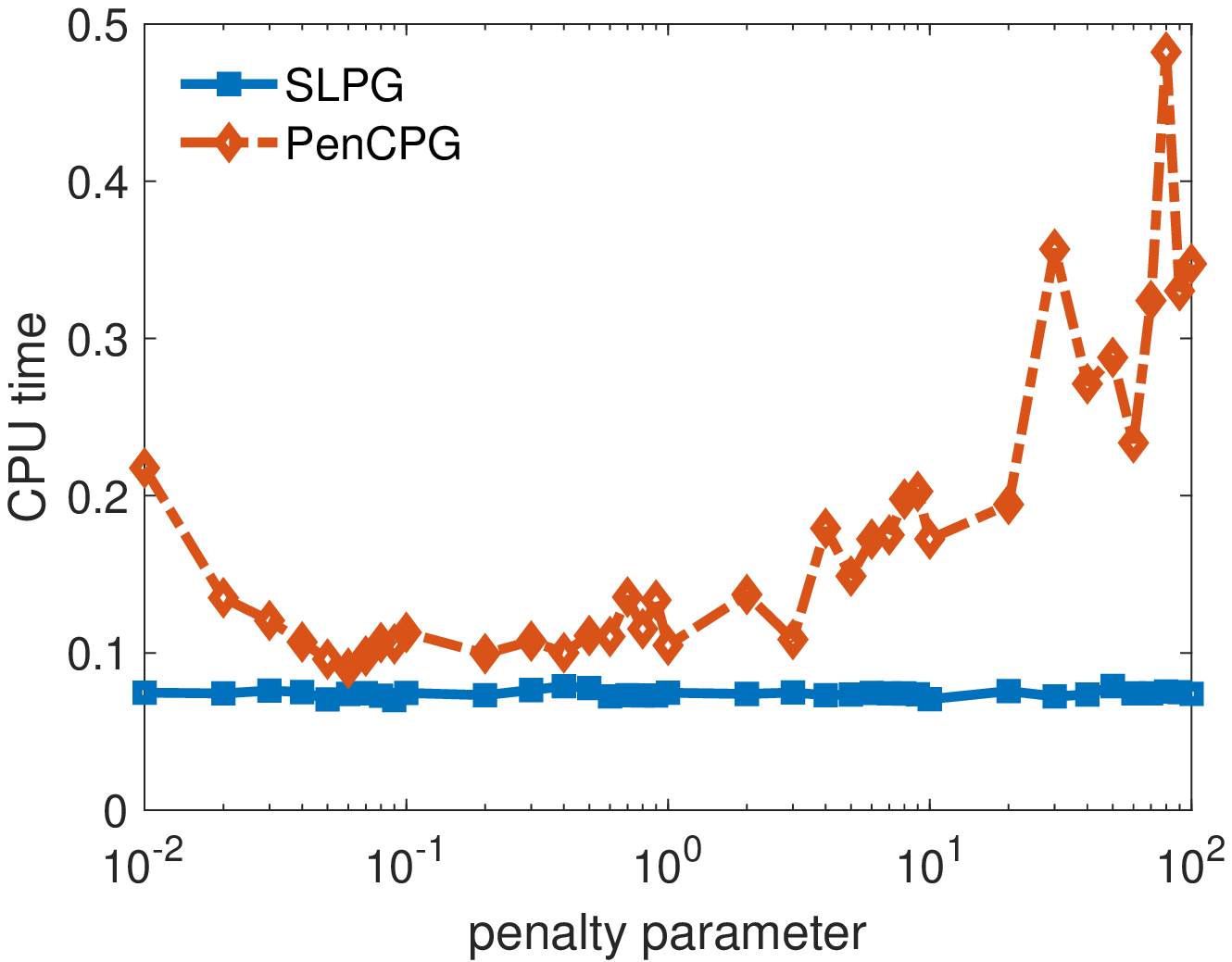}
				\label{Fig:21SPCA_Prob2_item5}
			\end{minipage}%
		}%
	
		\subfigure[$$(p, b) = (10, 0.1)$$]{
			\begin{minipage}[t]{0.33\linewidth}
				\centering
				\includegraphics[width=\linewidth,height=0.18\textheight]{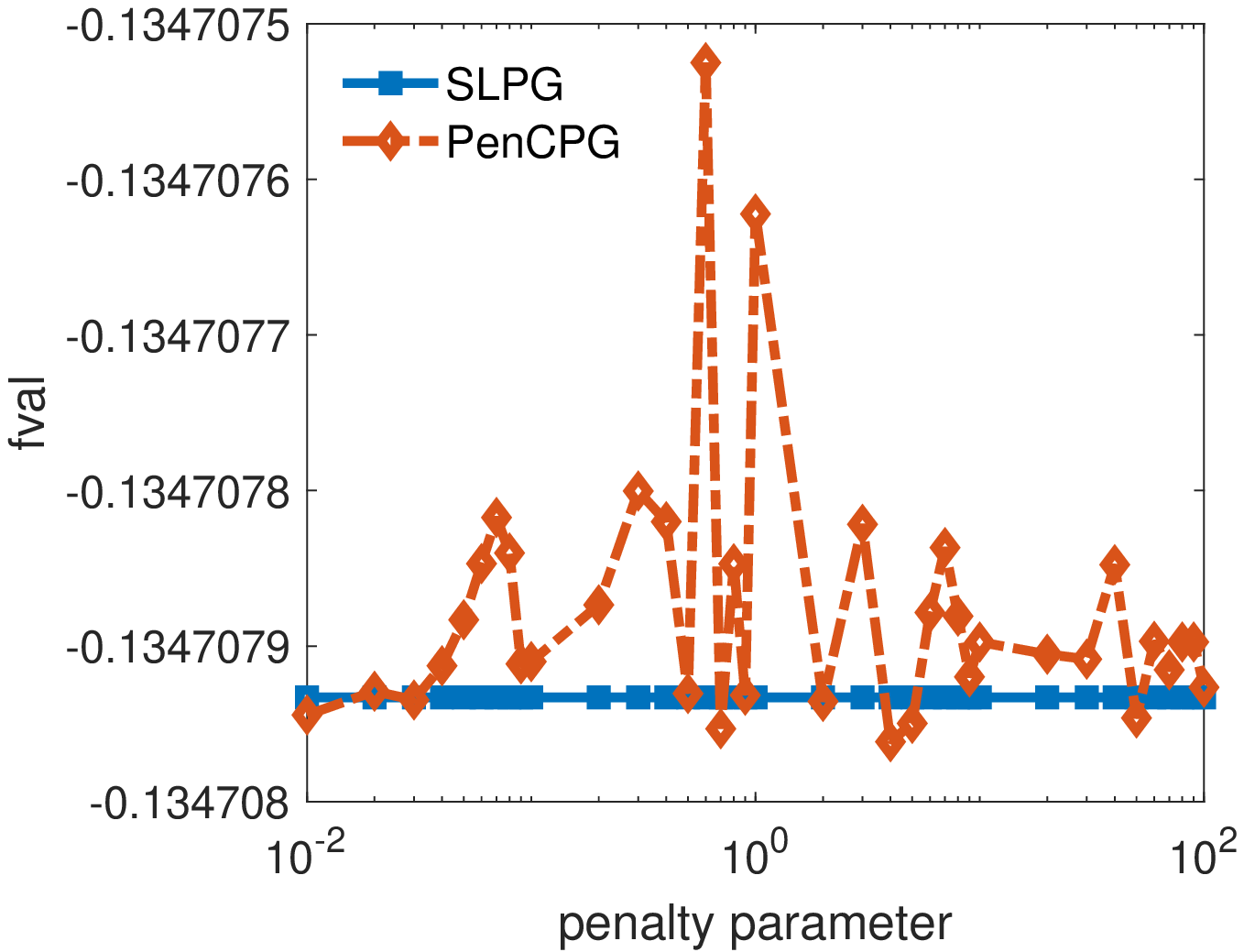}
				\label{Fig:21SPCA_Prob3_item1}
			\end{minipage}
		}%
		\subfigure[$$(p, b) = (10, 0.1)$$]{
			\begin{minipage}[t]{0.33\linewidth}
				\centering
				\includegraphics[width=\linewidth,height=0.18\textheight]{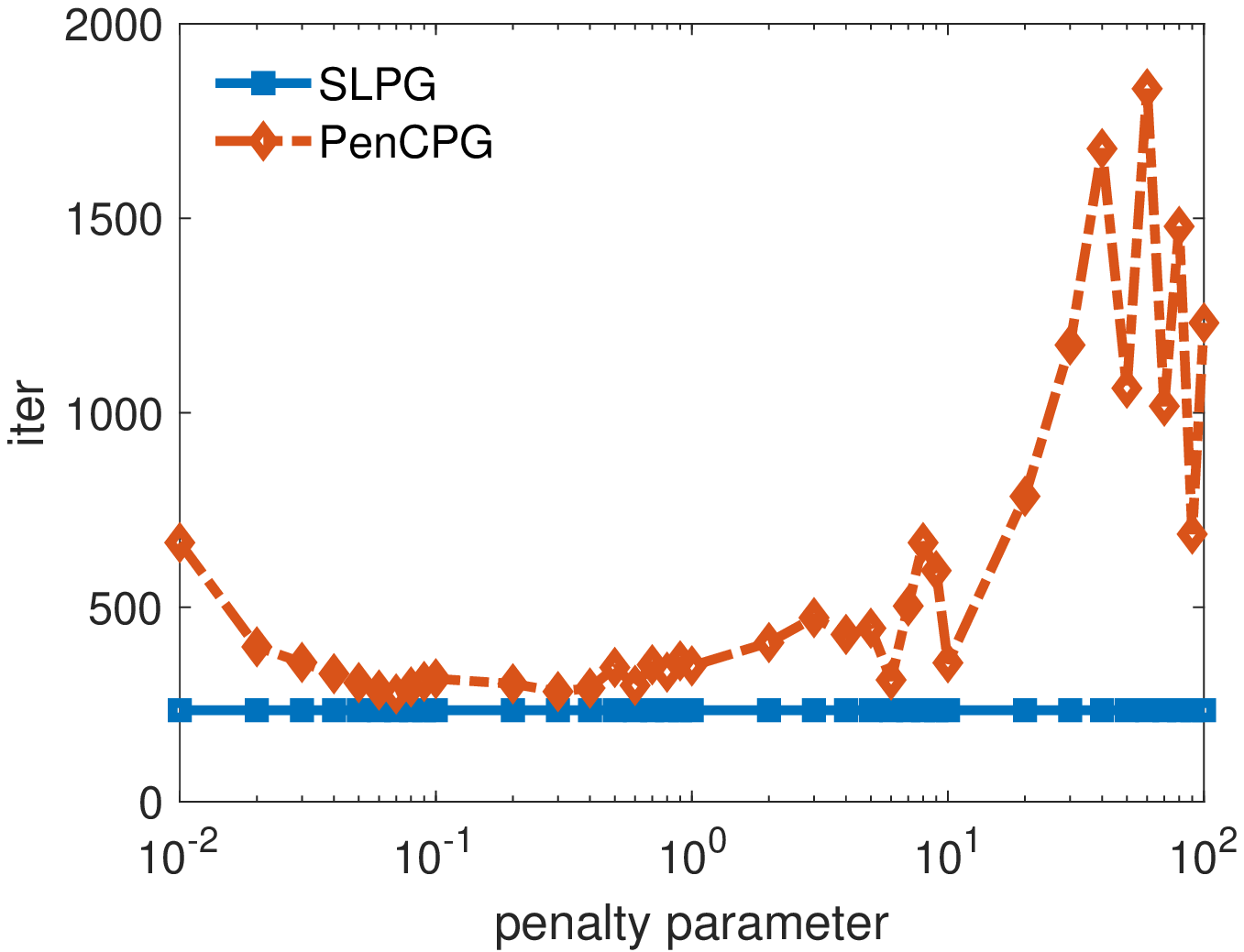}
				\label{Fig:21SPCA_Prob3_item2}
			\end{minipage}
		}%
		\subfigure[$$(p, b) = (10, 0.1)$$]{
			\begin{minipage}[t]{0.33\linewidth}
				\centering
				\includegraphics[width=\linewidth,height=0.18\textheight]{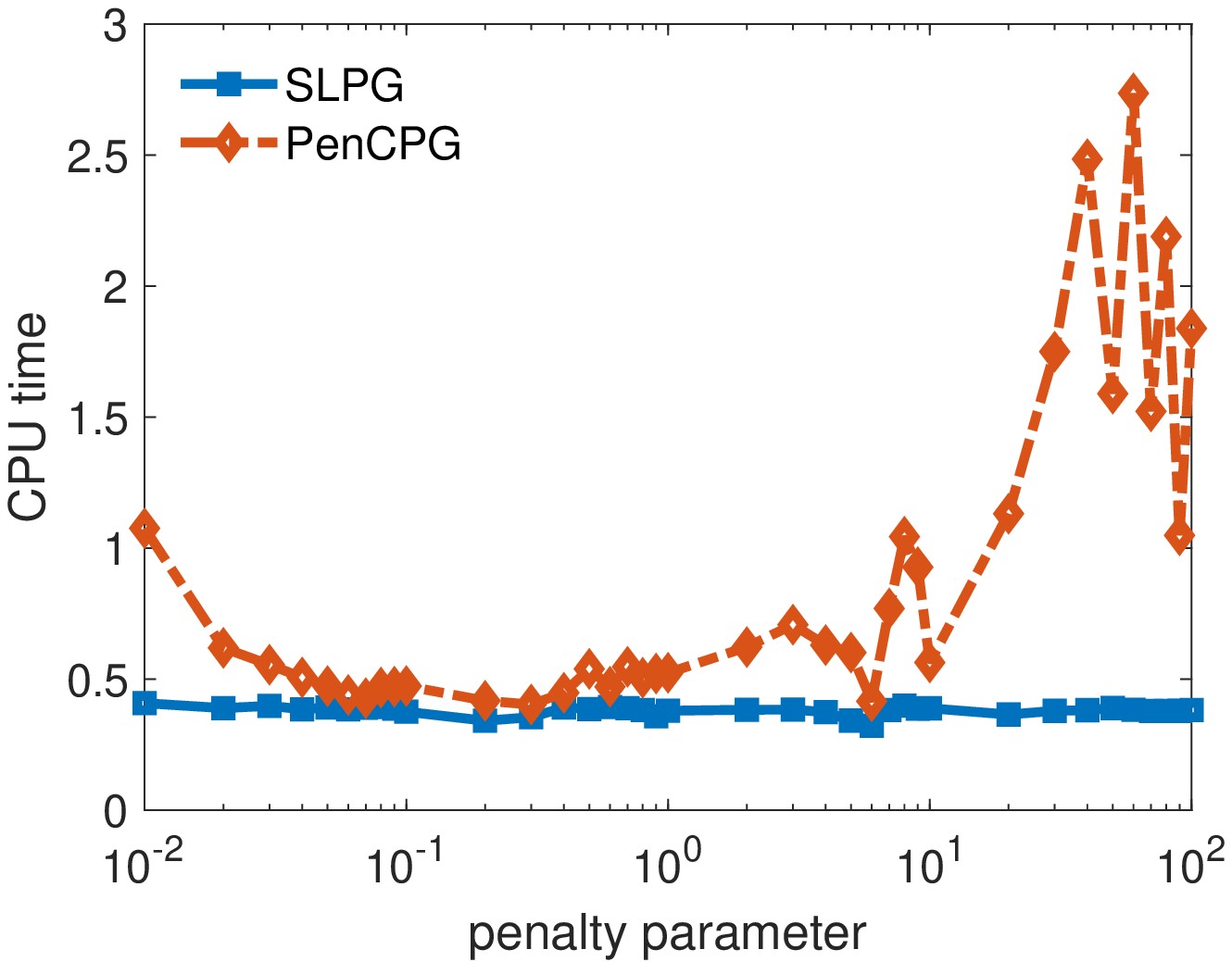}
				\label{Fig:21SPCA_Prob3_item5}
			\end{minipage}
		}%
		
		\caption{A comparison between SLPG and PenCPG with different penalty parameters.}
		\label{Fig_penalty_parameter}
	\end{figure}

	\subsection{Sparse PCA}
	
	In this subsection, we compare \SLPGs with two state-of-the-art algorithms  including ManPG-Ada \cite{chen2018proximal} and AManPG \cite{huang2019extending}
	in solving sparse PCA problem. 
	In our experiments, all the three algorithms are run in their default settings. 
	Figure \ref{Fig_l21PCA_Coefficients} illustrates the performance of the  
	three algorithms in comparison in solving Problem \ref{Example_SPCA}
	with different combinations of $n$, $p$, $\gamma$. The detailed problem
	parameters are listed below the subfigures.
	We can learn from Figure \ref{Fig_l21PCA_Coefficients}
	that
	all of these three algorithms reach the same function values.
	\SLPGs takes much fewer iterations than ManPG-Ada and
	slightly fewer iterations than AManPG. Meanwhile, it takes much less CPU
	time than the other two algorithms.

	\begin{figure}[!htbp]
		\centering
		\subfigure[$(p,\gamma) = (10,0.007)$]{
			\begin{minipage}[t]{0.33\linewidth}
				\centering
				\includegraphics[width=\linewidth,height=0.18\textheight]{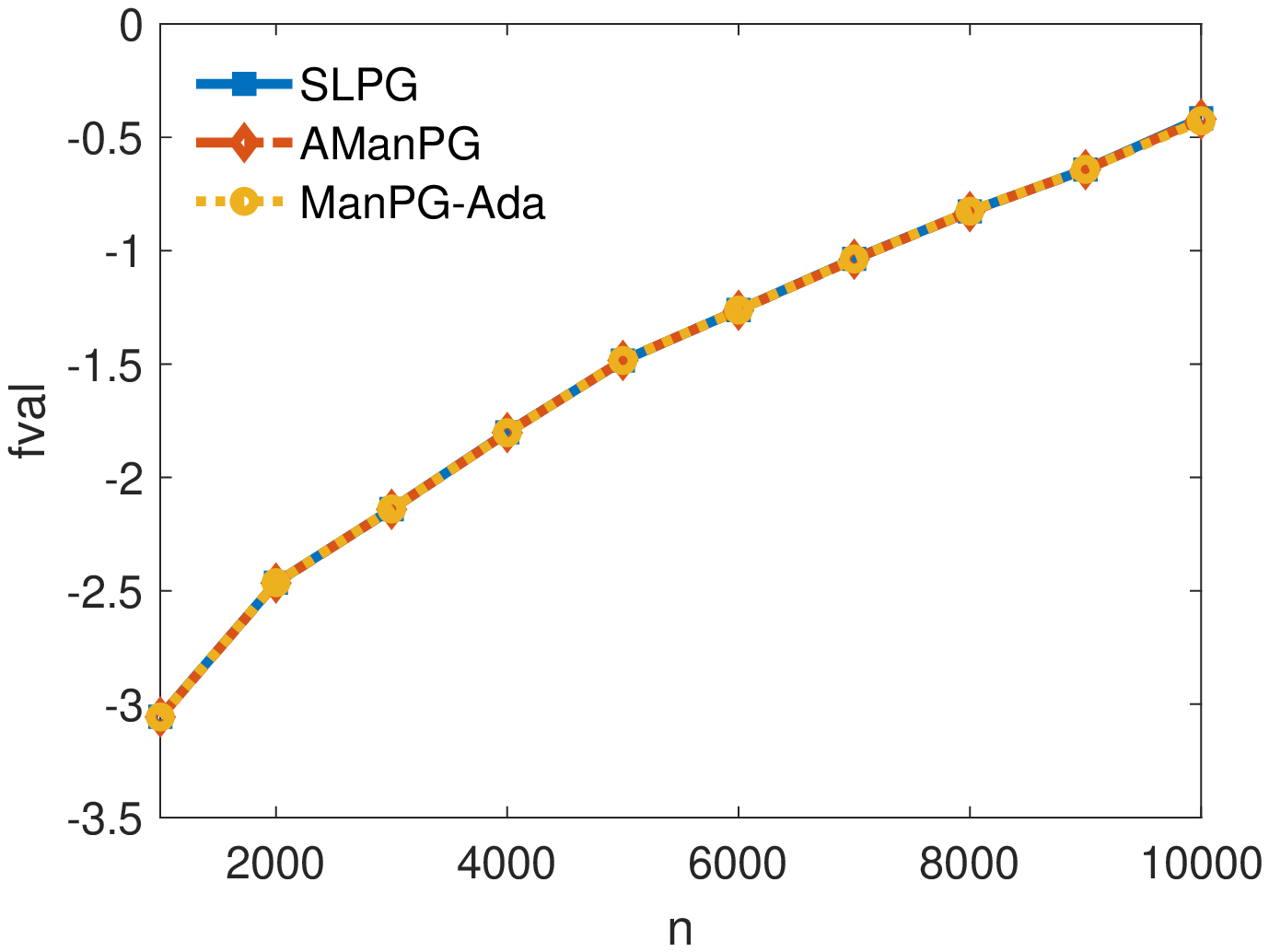}
				\label{Fig:SPCA_Prob1_item1}
			\end{minipage}%
		}%
		\subfigure[$(p,\gamma) = (10,0.007)$]{
			\begin{minipage}[t]{0.33\linewidth}
				\centering
				\includegraphics[width=\linewidth,height=0.18\textheight]{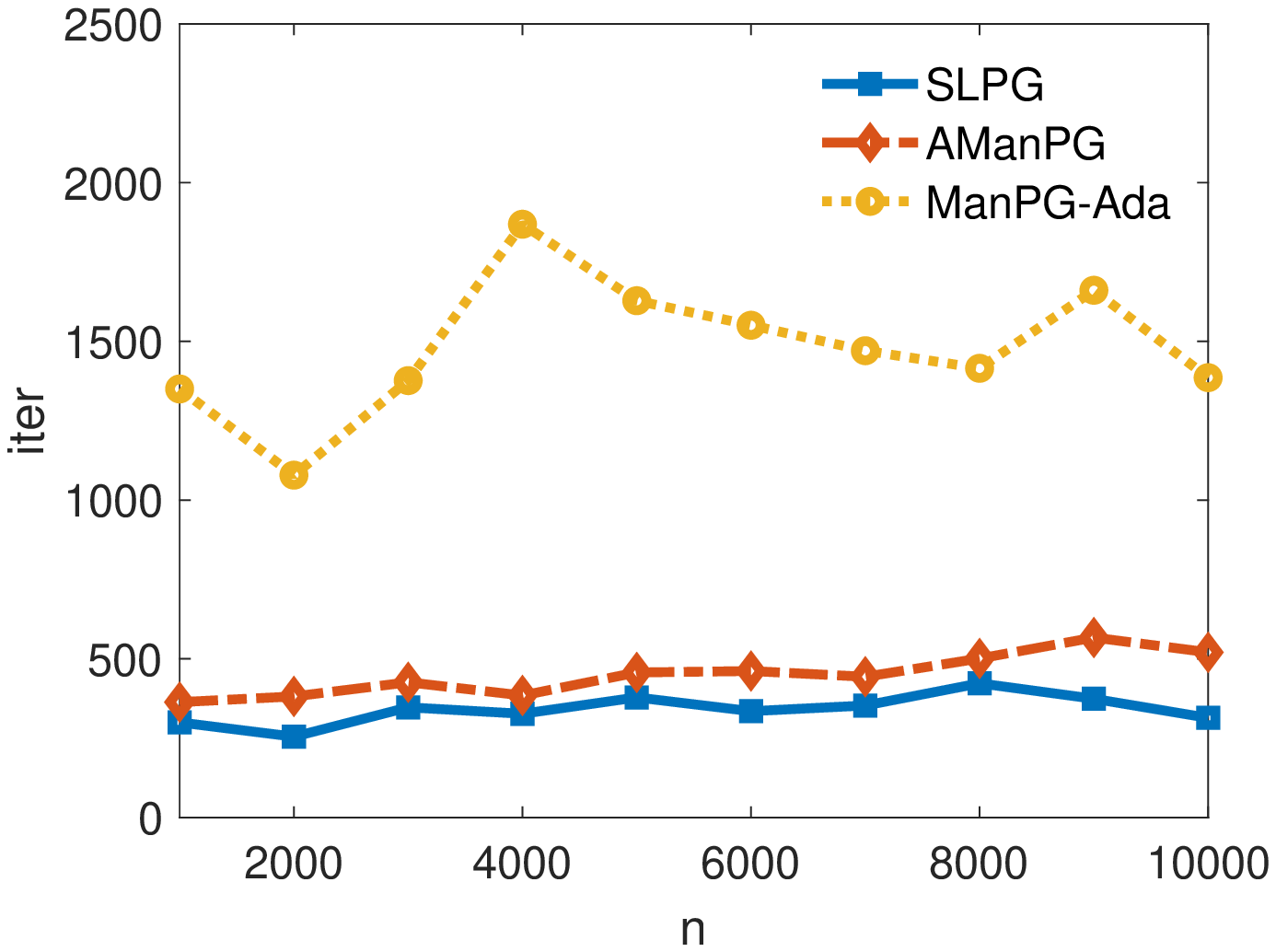}
				\label{Fig:SPCA_Prob1_item2}
			\end{minipage}%
		}%
		\subfigure[$(p,\gamma) = (10,0.007)$]{
			\begin{minipage}[t]{0.33\linewidth}
				\centering
				\includegraphics[width=\linewidth,height=0.18\textheight]{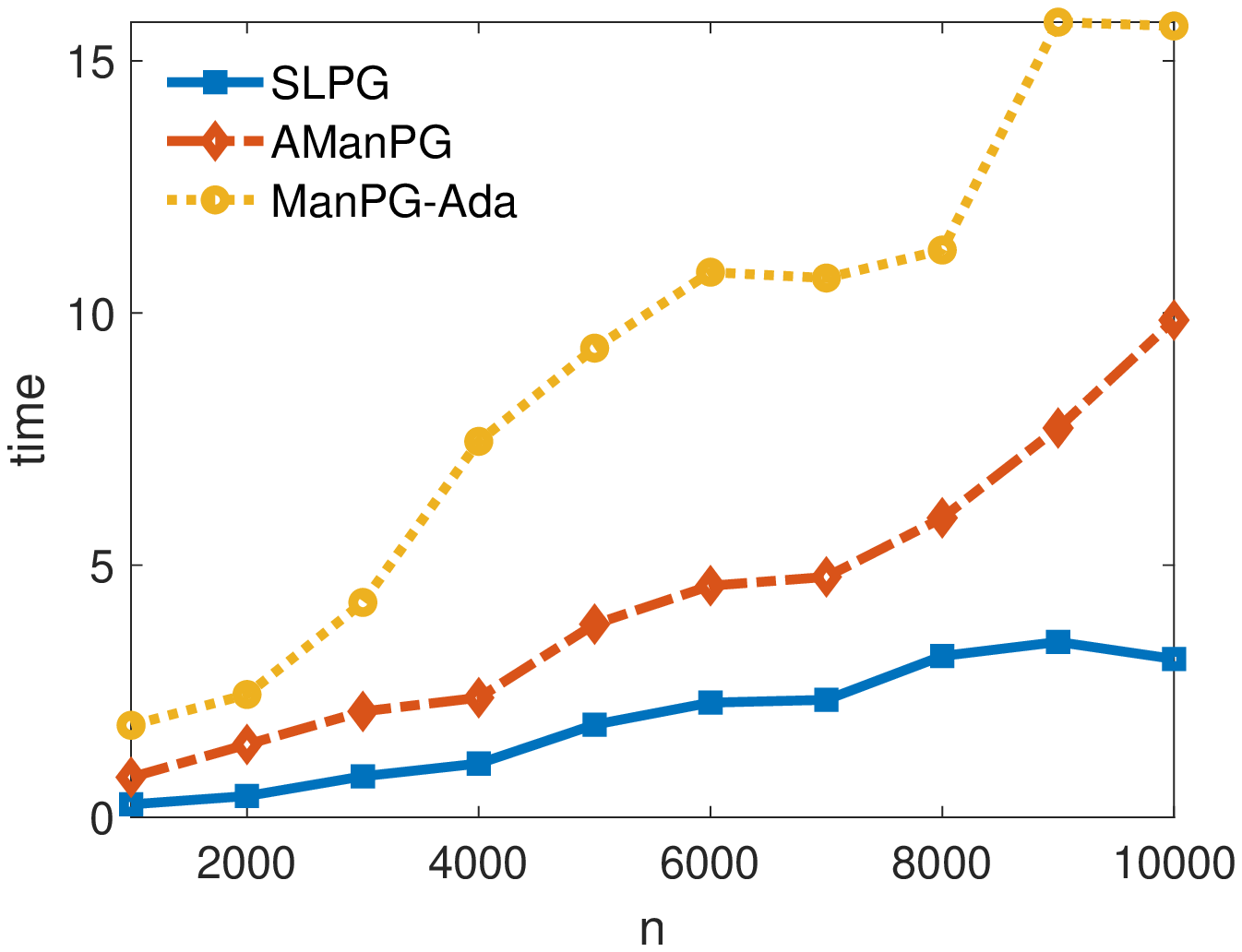}
				\label{Fig:SPCA_Prob1_item5}
			\end{minipage}%
		}%
	
		\subfigure[$(n,\gamma) = (5000,0.007)$]{
			\begin{minipage}[t]{0.33\linewidth}
				\centering
				\includegraphics[width=\linewidth,height=0.18\textheight]{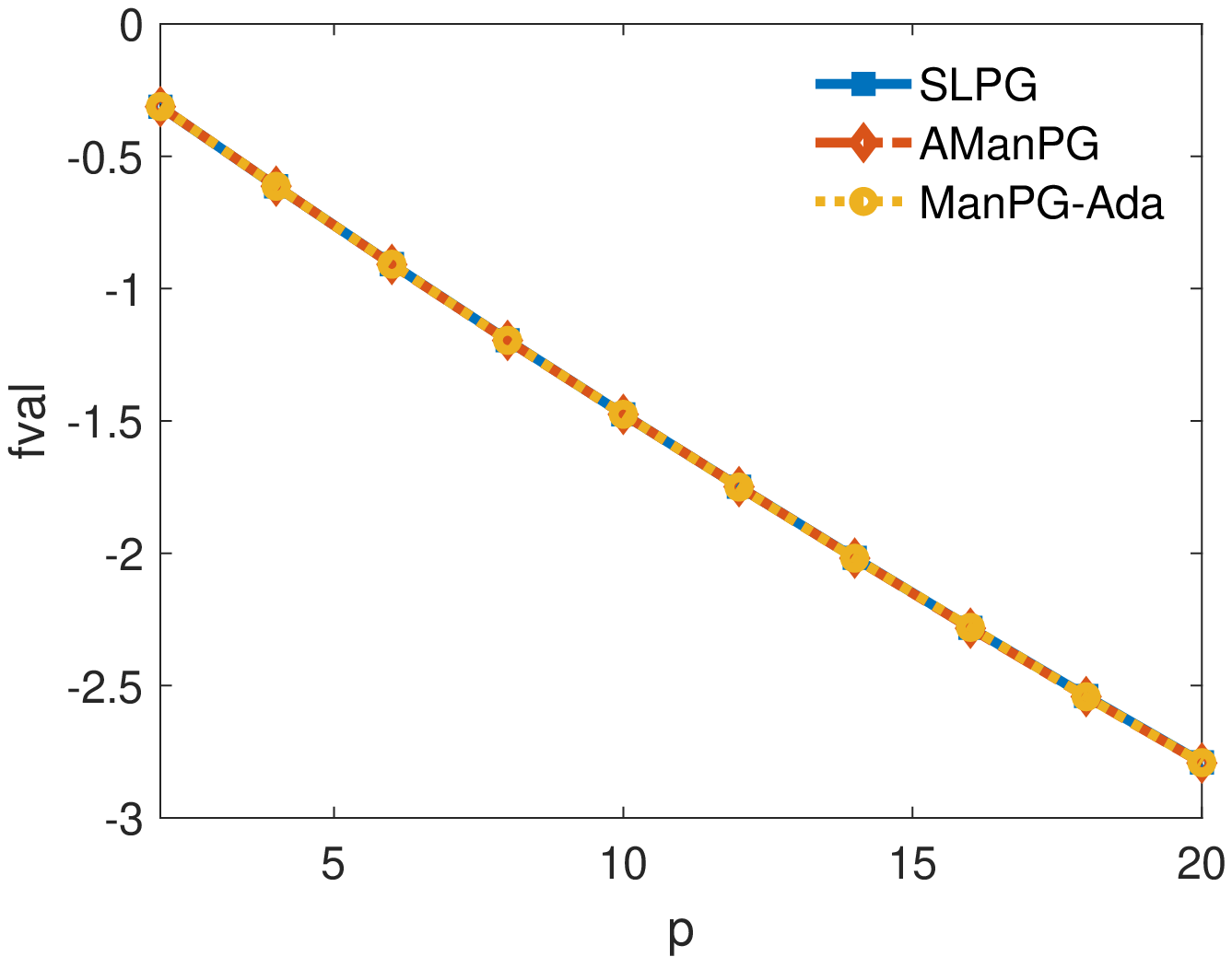}
				\label{Fig:SPCA_Prob2_item1}
			\end{minipage}%
		}%
		\subfigure[$(n,\gamma) = (5000,0.007)$]{
			\begin{minipage}[t]{0.33\linewidth}
				\centering
				\includegraphics[width=\linewidth,height=0.18\textheight]{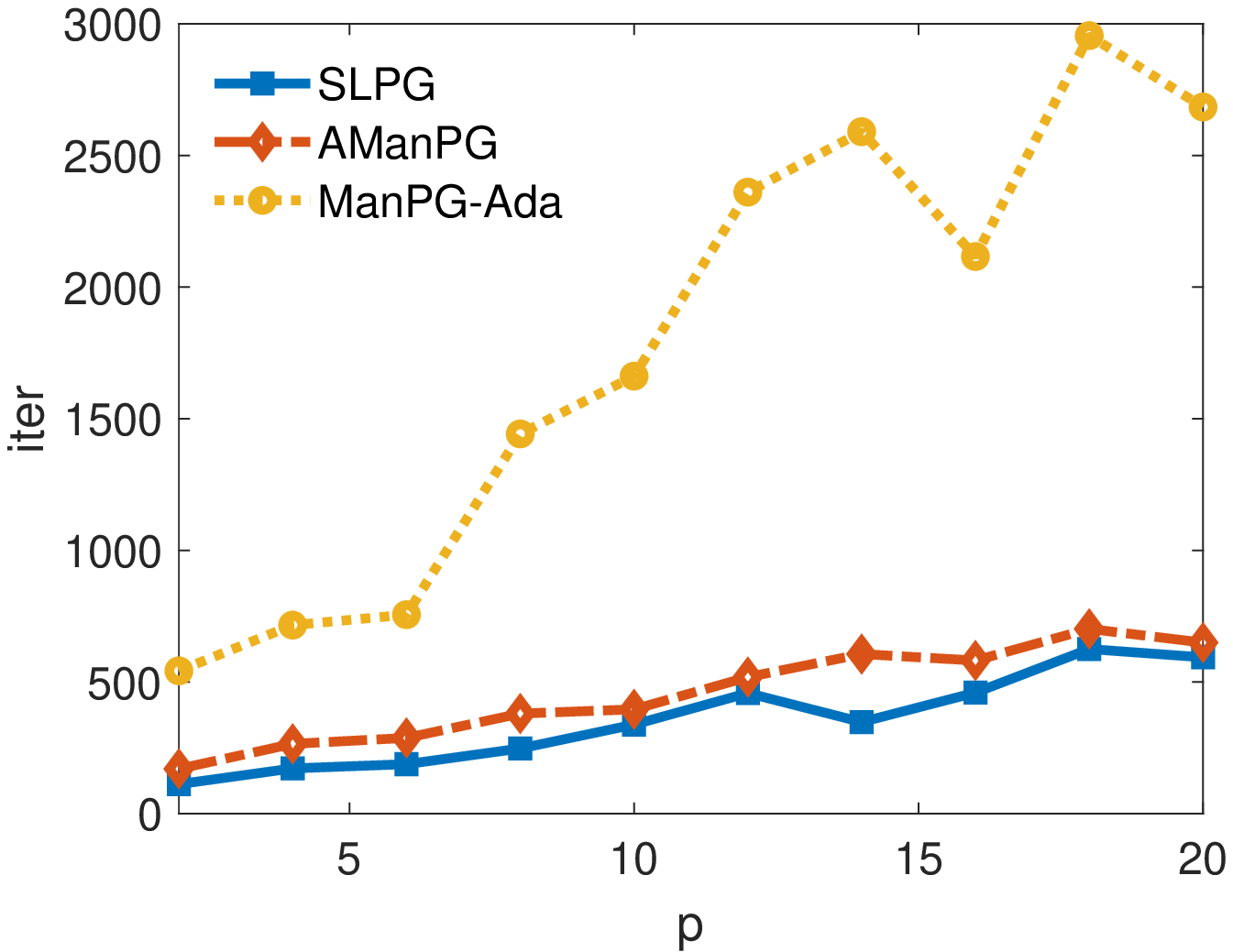}
				\label{Fig:SPCA_Prob2_item2}
			\end{minipage}%
		}%
		\subfigure[$(n,\gamma) = (5000,0.007)$]{
			\begin{minipage}[t]{0.33\linewidth}
				\centering
				\includegraphics[width=\linewidth,height=0.18\textheight]{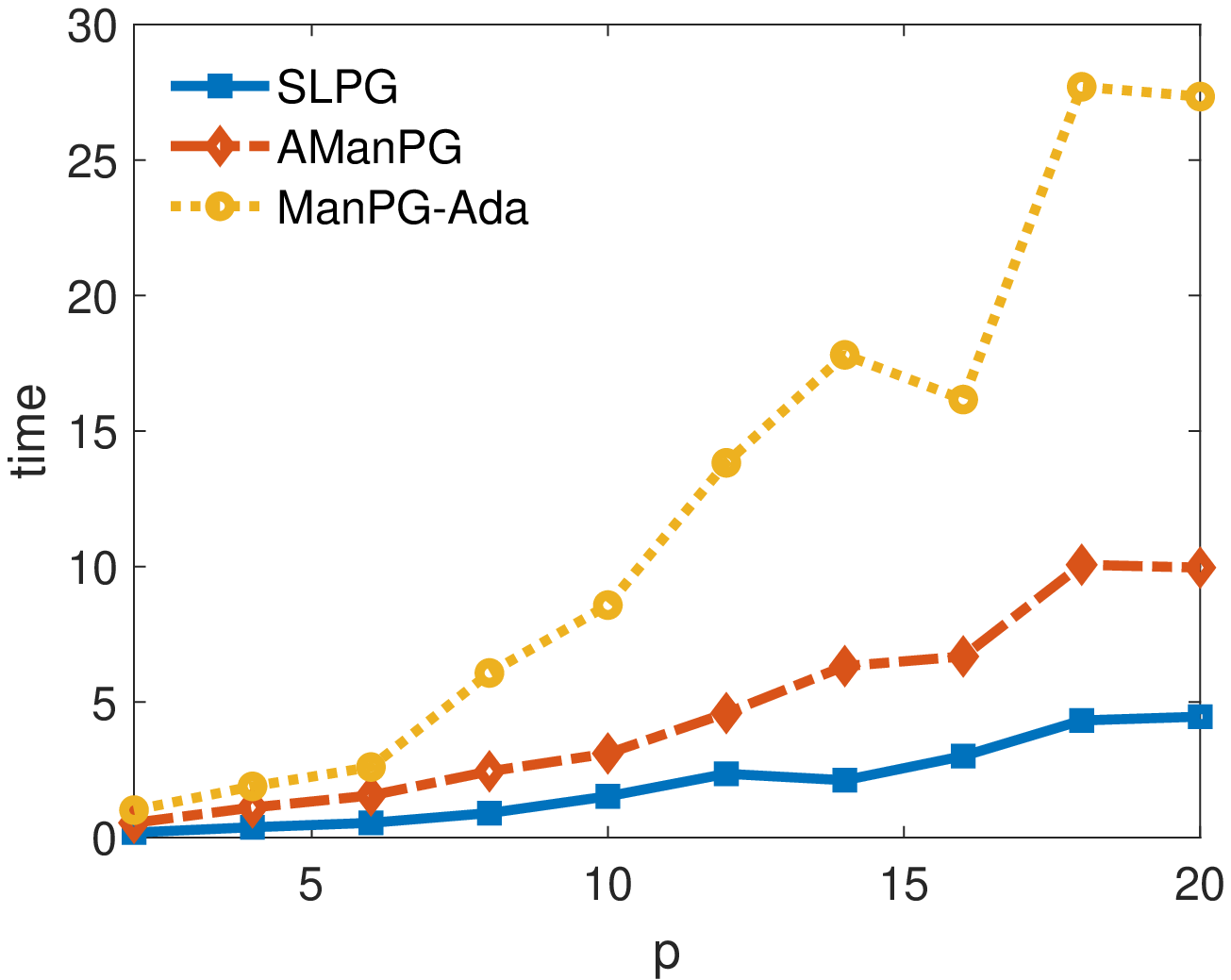}
				\label{Fig:SPCA_Prob2_item5}
			\end{minipage}%
		}%
	
		\subfigure[$(n,p) = (5000,10)$]{
			\begin{minipage}[t]{0.33\linewidth}
				\centering
				\includegraphics[width=\linewidth,height=0.18\textheight]{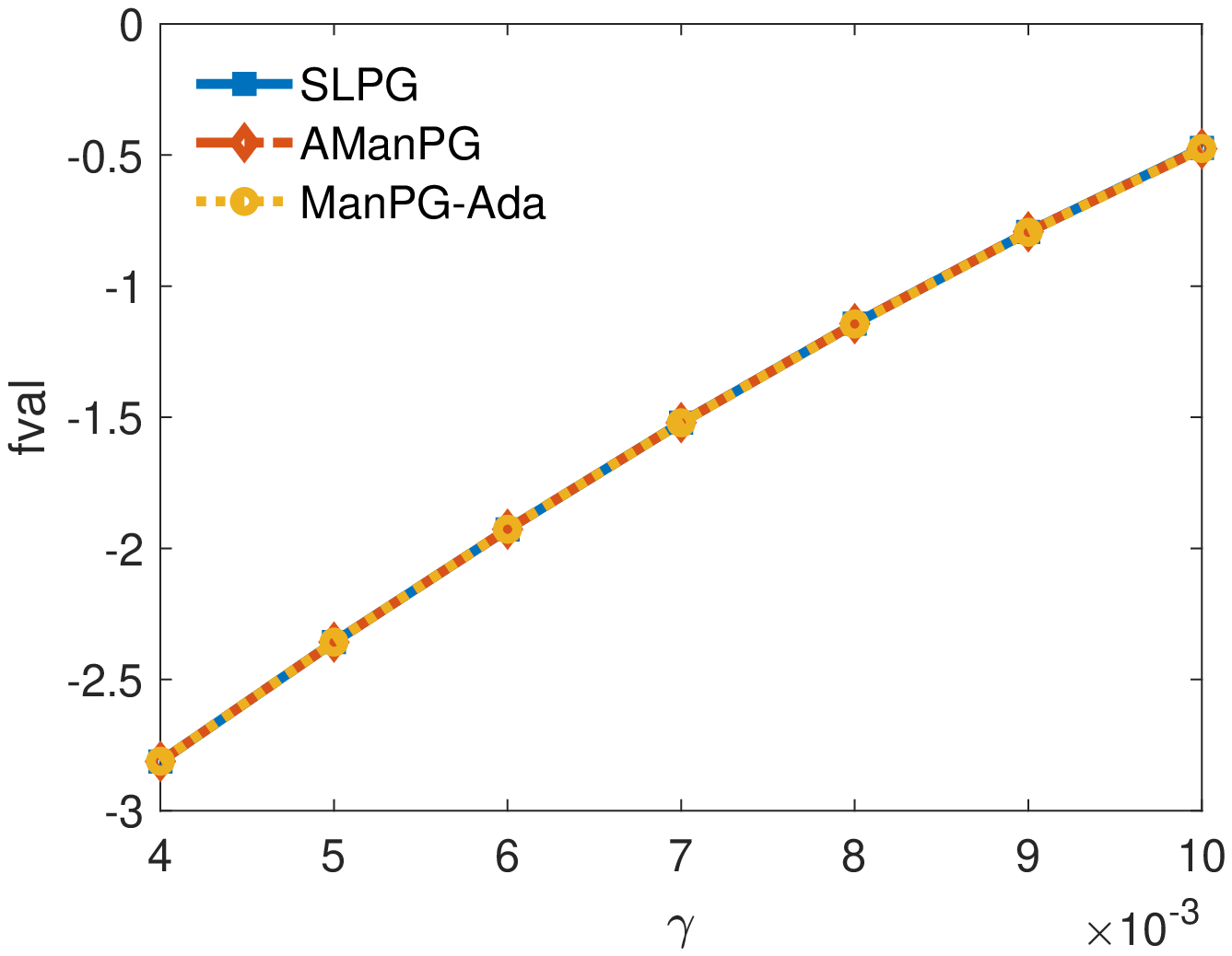}
				\label{Fig:SPCA_Prob3_item1}
			\end{minipage}
		}%
		\subfigure[$(n,p) = (5000,10)$]{
			\begin{minipage}[t]{0.33\linewidth}
				\centering
				\includegraphics[width=\linewidth,height=0.18\textheight]{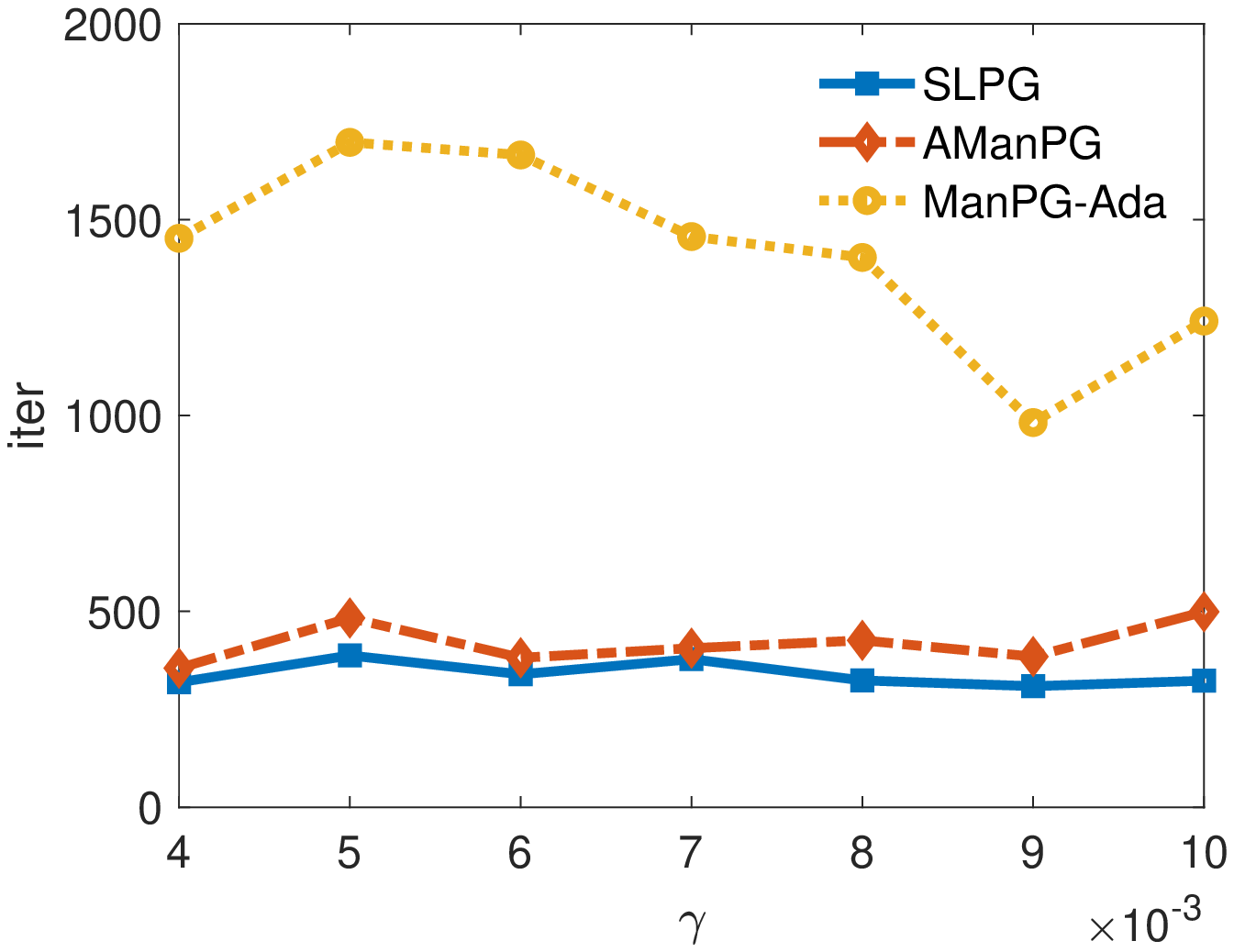}
				\label{Fig:SPCA_Prob3_item2}
			\end{minipage}
		}%
		\subfigure[$(n,p) = (5000,10)$]{
			\begin{minipage}[t]{0.33\linewidth}
				\centering
				\includegraphics[width=\linewidth,height=0.18\textheight]{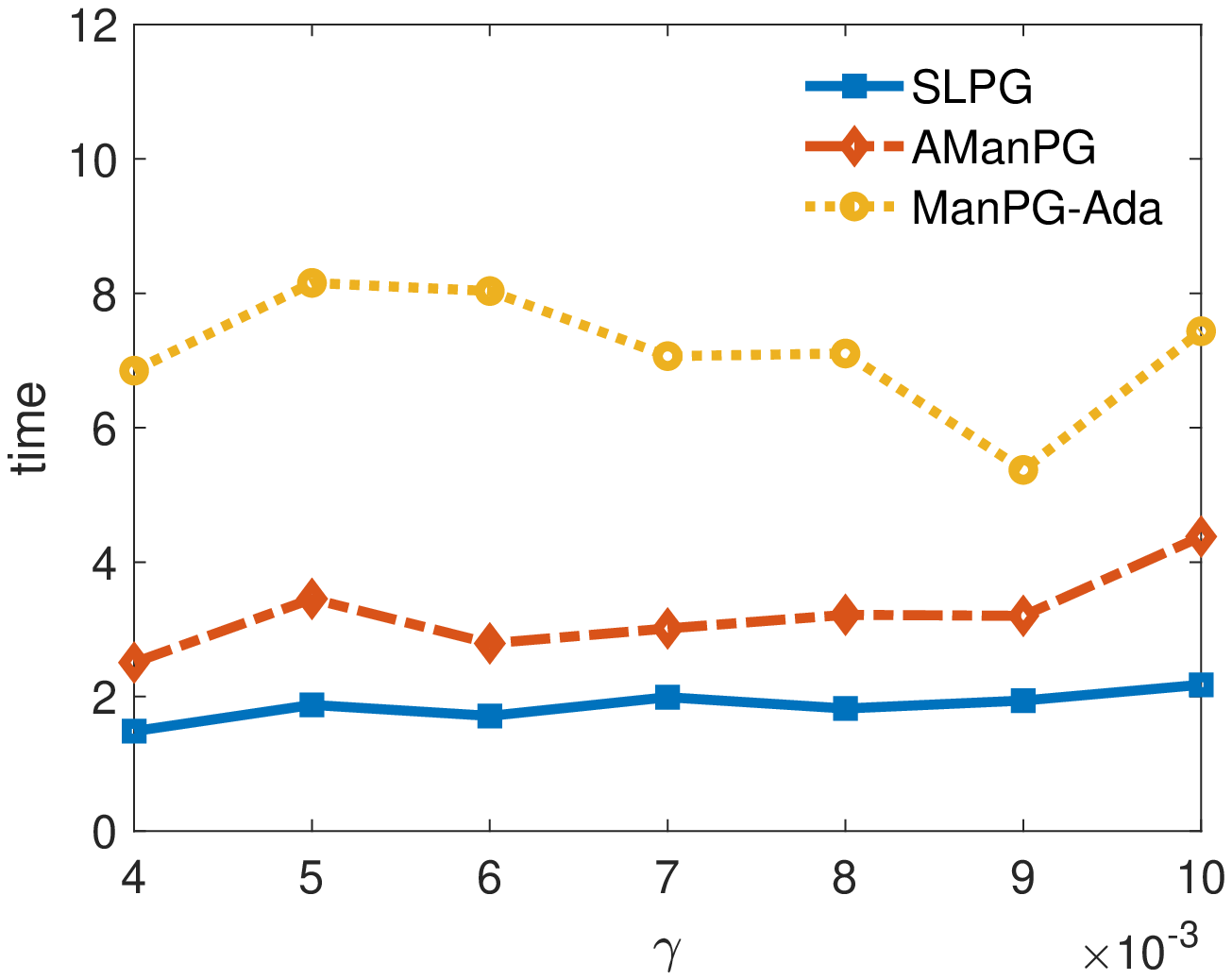}
				\label{Fig:SPCA_Prob3_item5}
			\end{minipage}
		}%
		
		\caption{A comparison among \SLPG, AManPG and ManPG-Ada in solving sparse PCA problems.}
		\label{Fig_Coefficients}
	\end{figure}

	\subsection{Kohn-Sham total energy minimization}
	In this subsection, we compare our algorithm with some state-of-the-art approaches in solving Problem \ref{Example_KS}. 
	The test problems are selected from the Kohn-Sham total energy minimization platform KSSOLV \cite{yang2009kssolv}, which is a MATLAB toolbox designed for electronic structure calculation. The algorithms in comparison include
	PCAL \cite{gao2019parallelizable} and PenCF \cite{xiao2020class}. 
	We compare all these algorithms in their default settings.
	We first study the numerical performance of \SLPGs and compare it with PCAL and PenCF 
	equipped with different penalty parameter $\beta$. 
	The performances of these algorithms are demonstrated in 
	Figure \ref{Fig_KS_iter}. We can learn that the performances
	of PCAL and PenCF  
	are sensitive to the penalty parameter $\beta$. Meanwhile,
	\SLPG is penalty-paramter-free and has comparable performs 
	with the
	other two algorithms equipped with fine-tuned penalty parameters.

	Finally, we comprehensively compare the performance of \SLPGs with more state-of-the-art algorithms, including 
	the projection-based feasible method with QR factorization as retraction
	(``ManOptQR" for short) from Manopt toolbox \citep{Absil2009optimization,boumal2014manopt}, OptM proposed by 
	Wen and Yin \cite{wen2013feasible}, PCAL and PenCF. 
	In this experiment, all the algorithms are run in their default settings. 
	We set the stopping criteria and the maximum number of iterations as $\norm{\nabla f(\Xk) - \Xk \Phi(\Xk\tp \nabla f(\Xk))}\ff \leq 10^{-7}$ and $1000$, respectively. Table \ref{Table_KS} illustrates 
	the performance of these algorithms on $8$ test problems 
	with respect to different molecules. The terms ``$E_{tot}$'',
	``Substationarity", ``Iteration", ``Feasibility violation" and ``CPU time" 
	stand for the function value, $\norm{\nabla f(X) - X\Lambda(X)}\ff$, the number of iterations, $\norm{X\tp X - I_p}\ff$, and the wall-clock running time, respectively. We can learn from Table \ref{Table_KS} that  \SLPGs is comparable with these state-of-the-art algorithms in the aspect of iterations and CPU time 
	in solving all the test problems.
	
	To sum up, from the above numerical experiments,  we can conclude that
	\SLPGs exhibits its robustness and efficiency comparing with the existing
	algorithms in solving both smooth and nonsmooth minimization over the 
	Stiefel manifold.

	\begin{figure}[!htbp]
		\centering
		\subfigure[benzene]{
			\begin{minipage}[t]{0.33\linewidth}
				\centering
				\includegraphics[width=\linewidth,height=0.18\textheight]{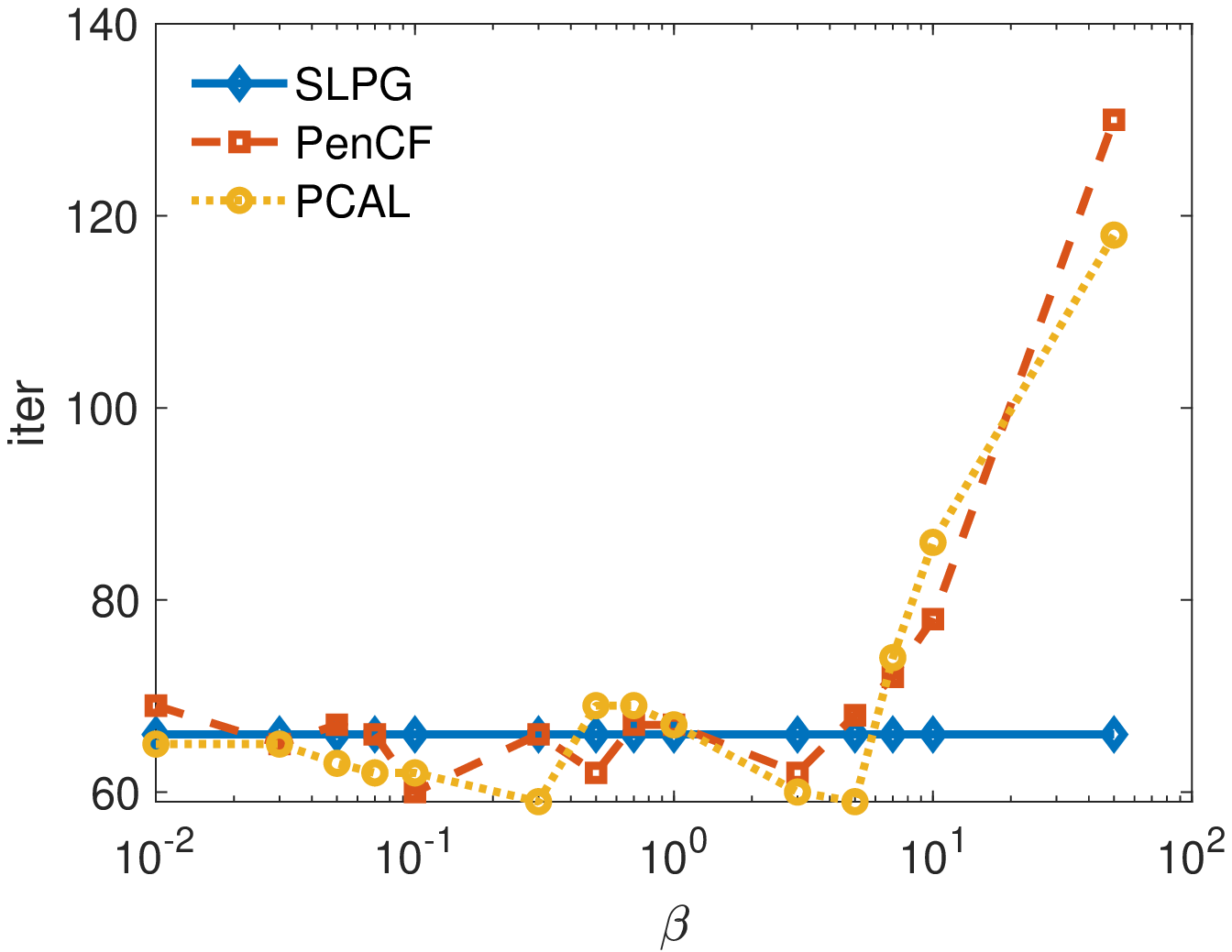}
				\label{Fig:KS_Prob1_item_1}
			\end{minipage}%
		}%
		\subfigure[ctube661]{
			\begin{minipage}[t]{0.33\linewidth}
				\centering
				\includegraphics[width=\linewidth,height=0.18\textheight]{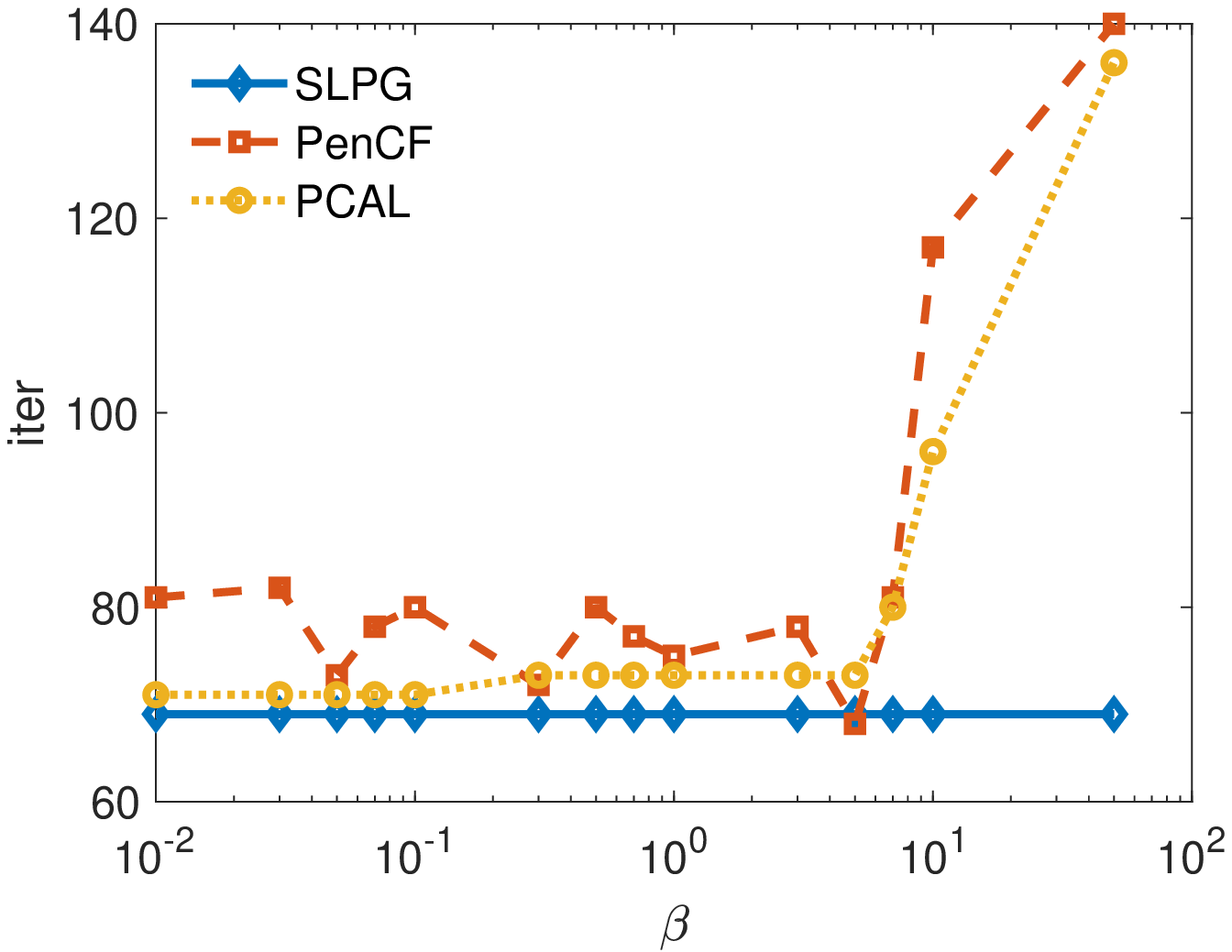}
				\label{Fig:KS_Prob2_item_1}
			\end{minipage}%
		}%
		\subfigure[glutamine]{
			\begin{minipage}[t]{0.33\linewidth}
				\centering
				\includegraphics[width=\linewidth,height=0.18\textheight]{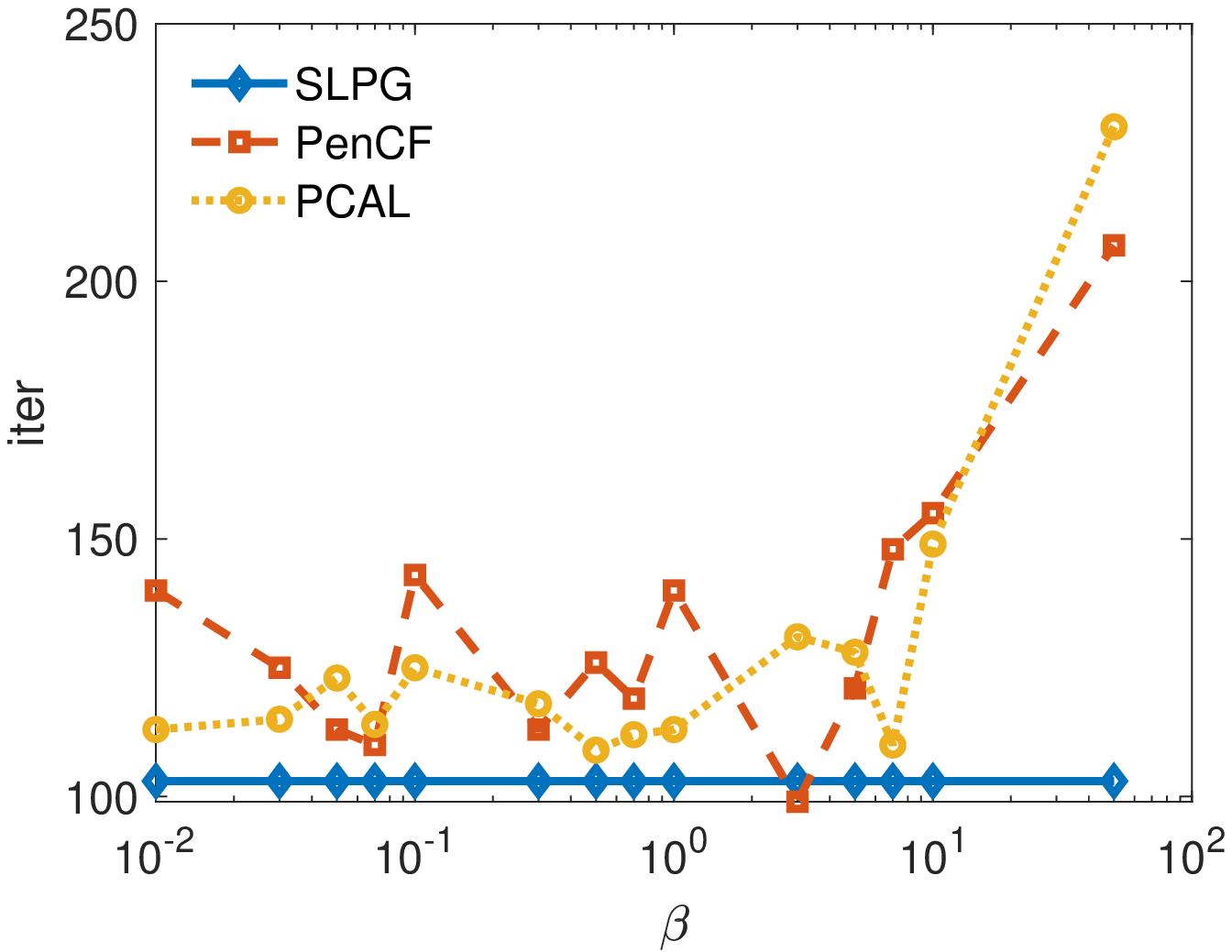}
				\label{Fig:KS_Prob3_item_1}
			\end{minipage}
		}%
		
		\subfigure[benzene]{
			\begin{minipage}[t]{0.33\linewidth}
				\centering
				\includegraphics[width=\linewidth,height=0.18\textheight]{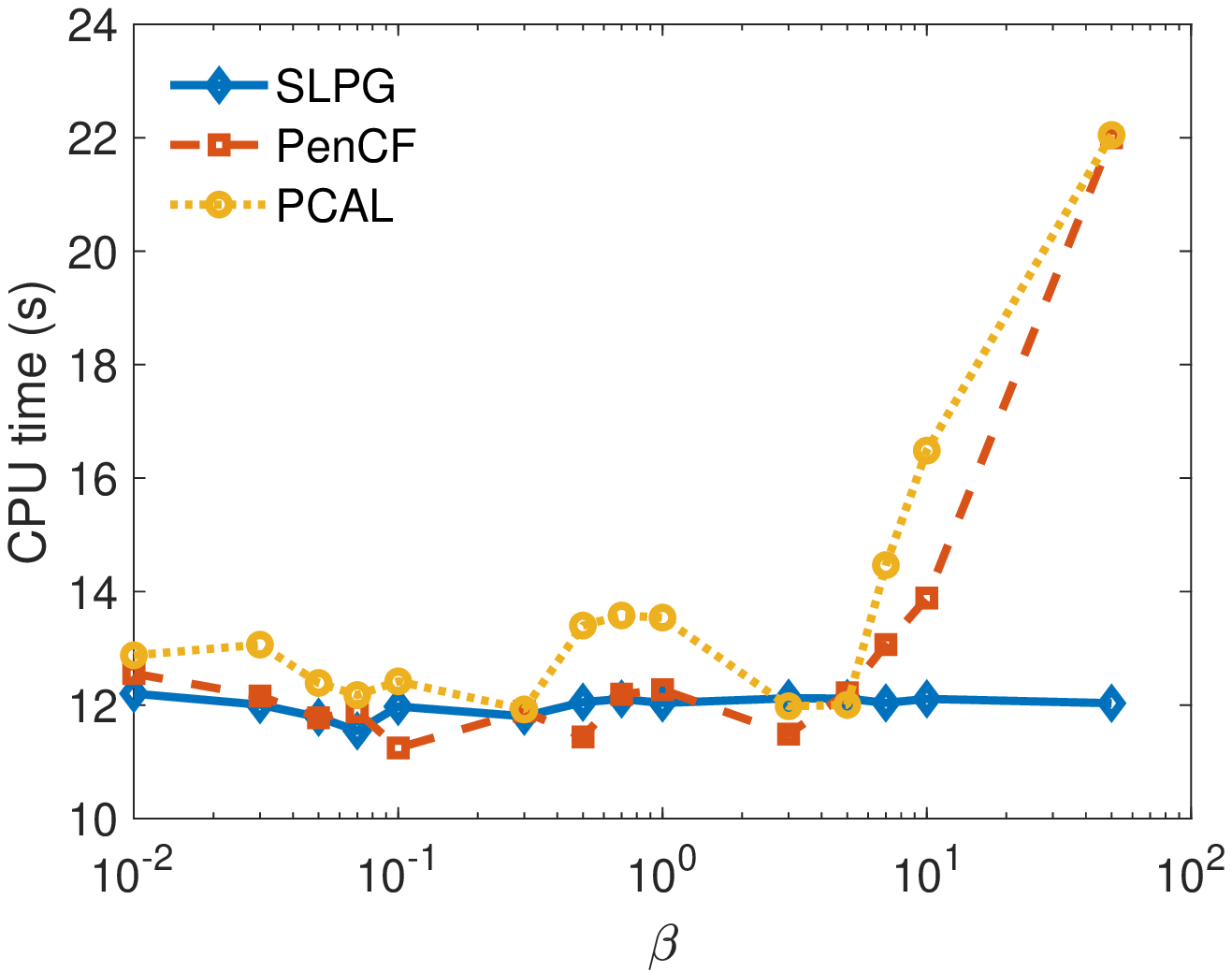}
				\label{Fig:KS_Prob1_item_2}
			\end{minipage}%
		}%
		\subfigure[ctube661]{
			\begin{minipage}[t]{0.33\linewidth}
				\centering
				\includegraphics[width=\linewidth,height=0.18\textheight]{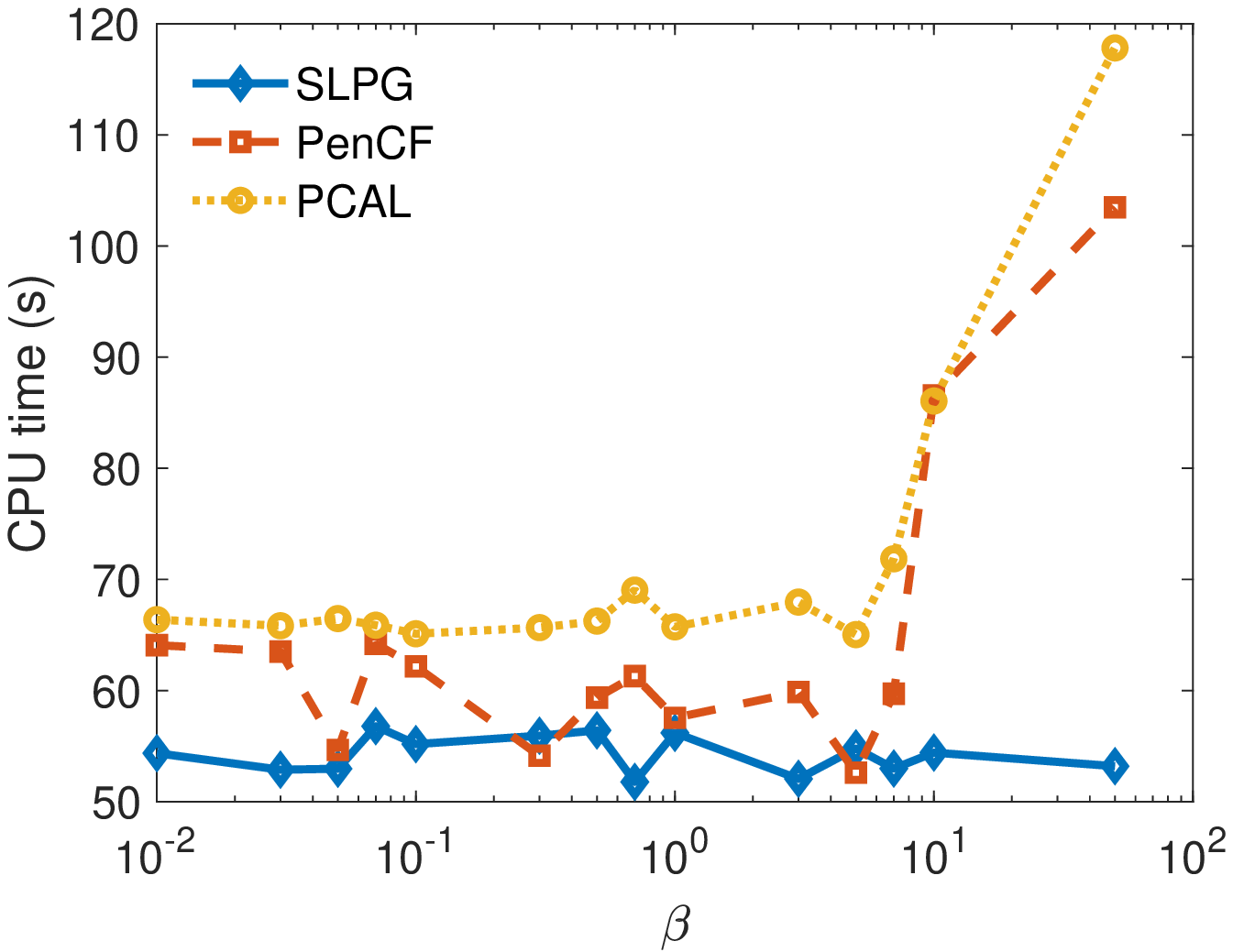}
				\label{Fig:KS_Prob2_item_2}
			\end{minipage}%
		}%
		\subfigure[glutamine]{
			\begin{minipage}[t]{0.33\linewidth}
				\centering
				\includegraphics[width=\linewidth,height=0.18\textheight]{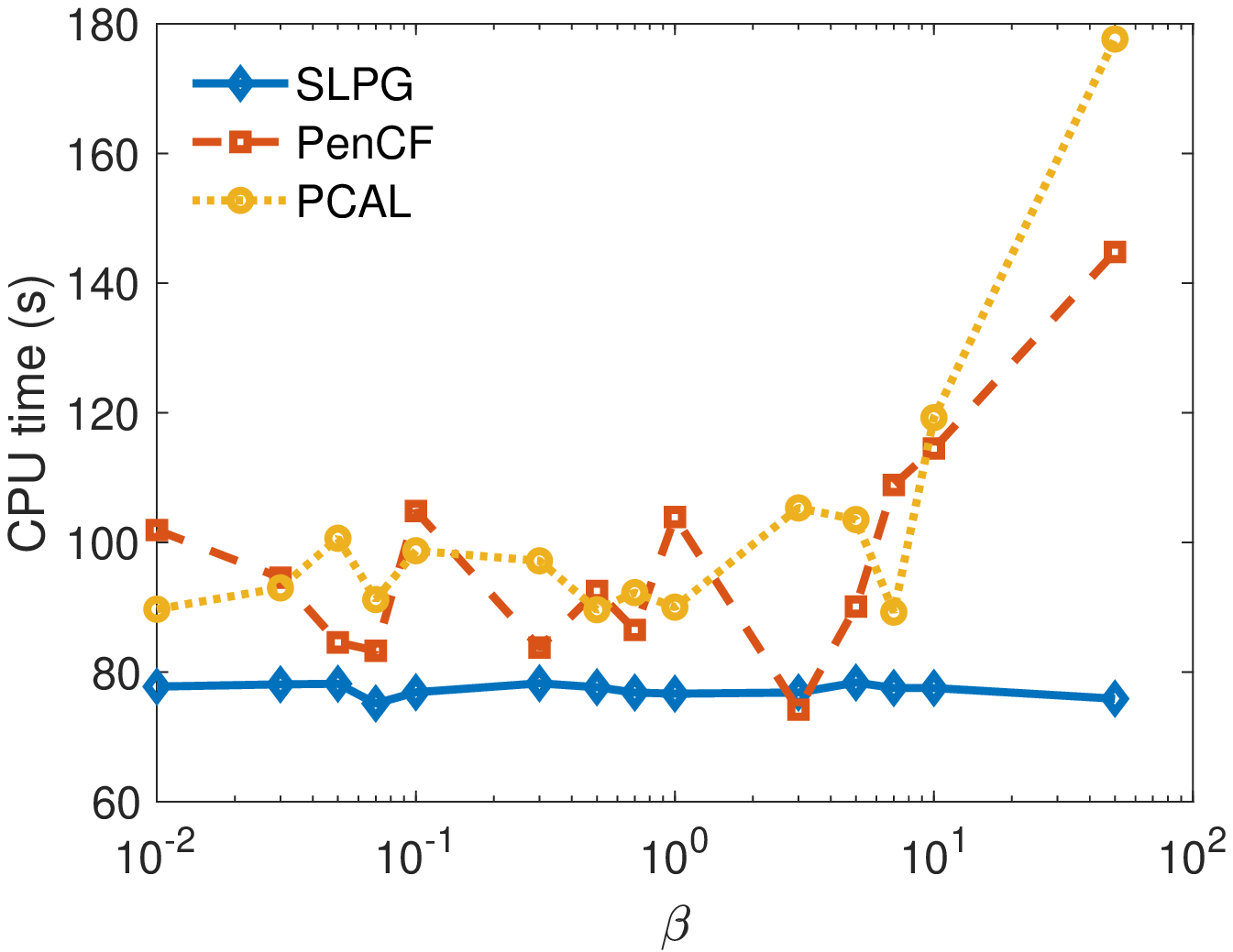}
				\label{Fig:KS_Prob3_item_2}
			\end{minipage}
		}%
		
		\subfigure[graphene16]{
			\begin{minipage}[t]{0.33\linewidth}
				\centering
				\includegraphics[width=\linewidth,height=0.18\textheight]{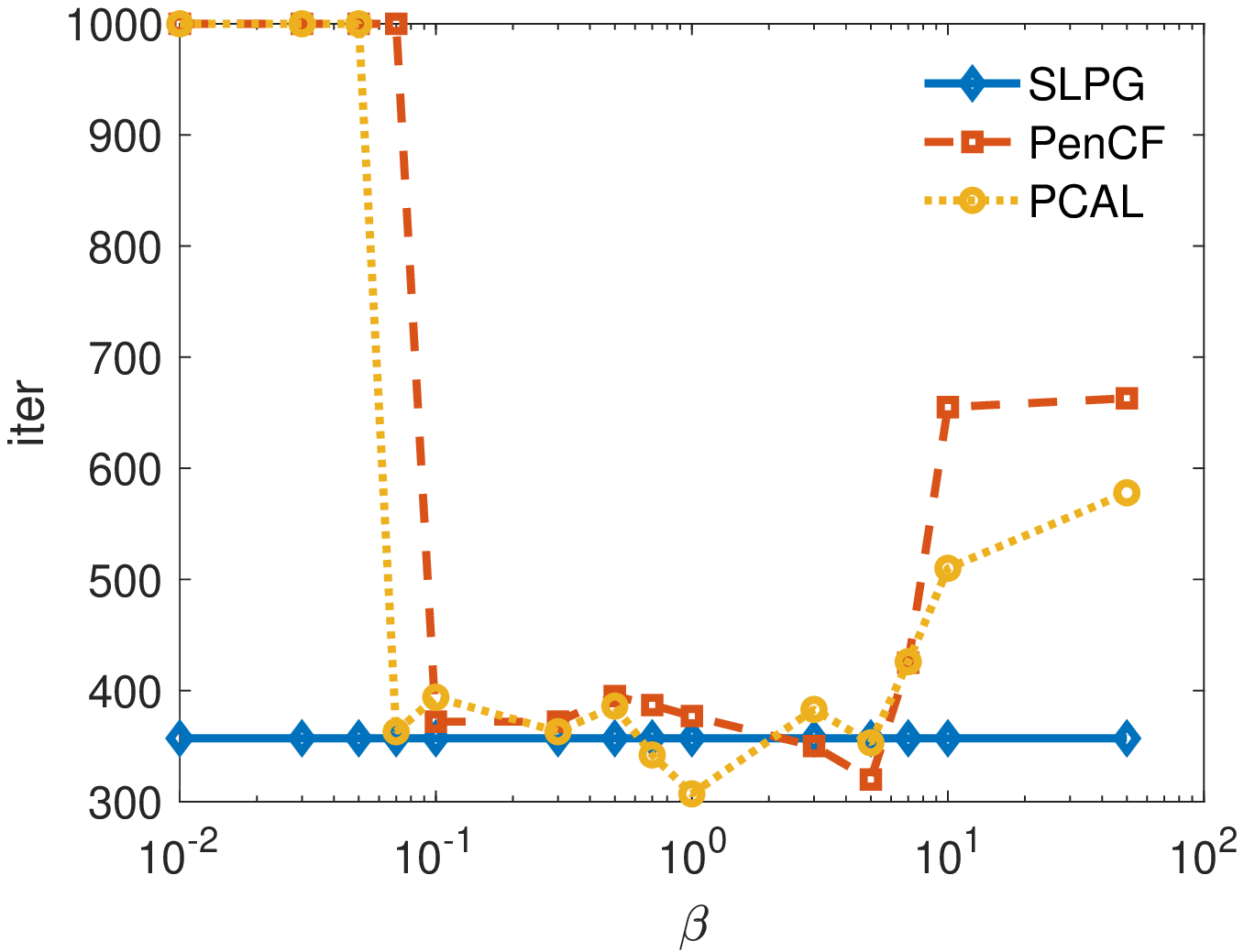}
				\label{Fig:KS_Prob4_item_1}
			\end{minipage}%
		}%
		\subfigure[C12H26]{
			\begin{minipage}[t]{0.33\linewidth}
				\centering
				\includegraphics[width=\linewidth,height=0.18\textheight]{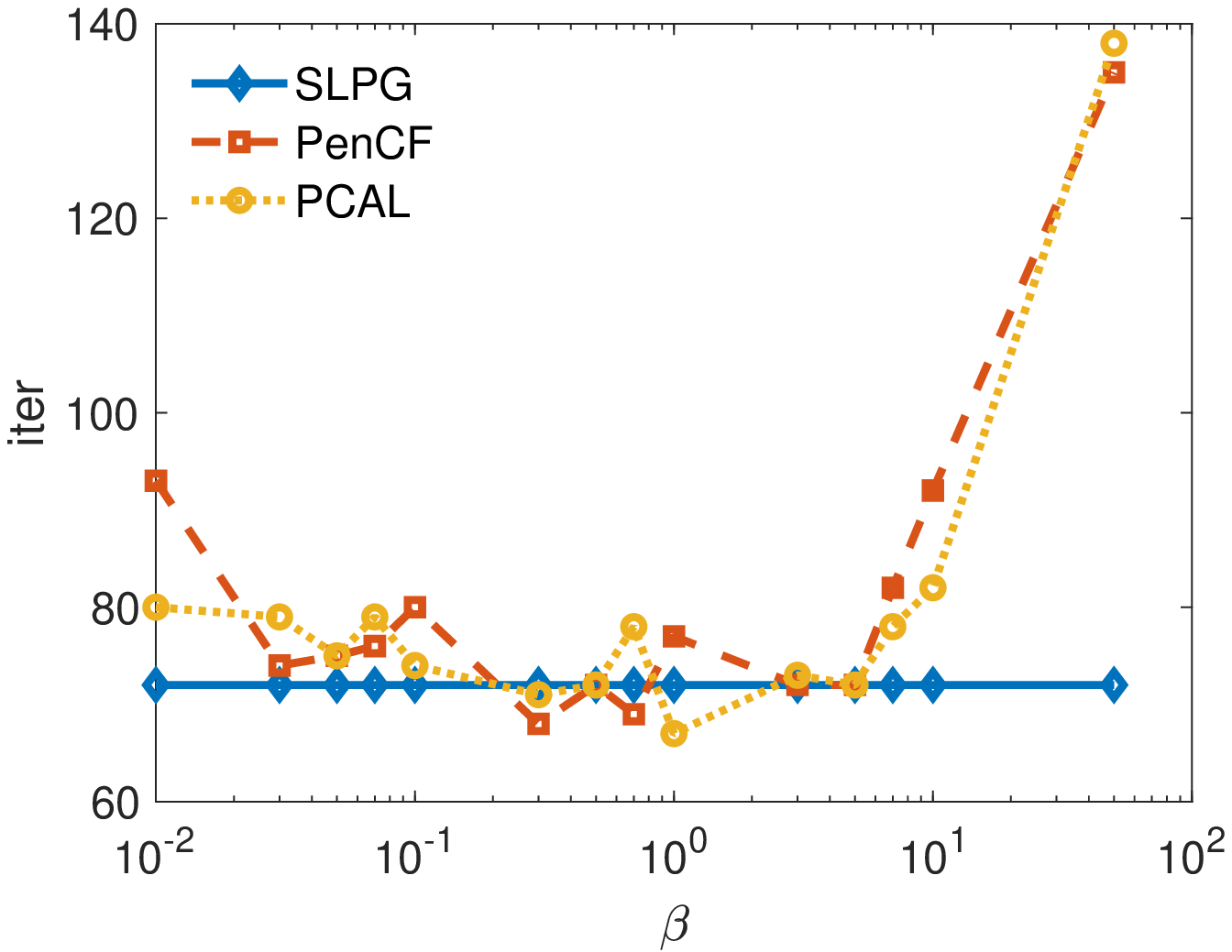}
				\label{Fig:KS_Prob5_item_1}
			\end{minipage}%
		}%
		\subfigure[pentacene]{
			\begin{minipage}[t]{0.33\linewidth}
				\centering
				\includegraphics[width=\linewidth,height=0.18\textheight]{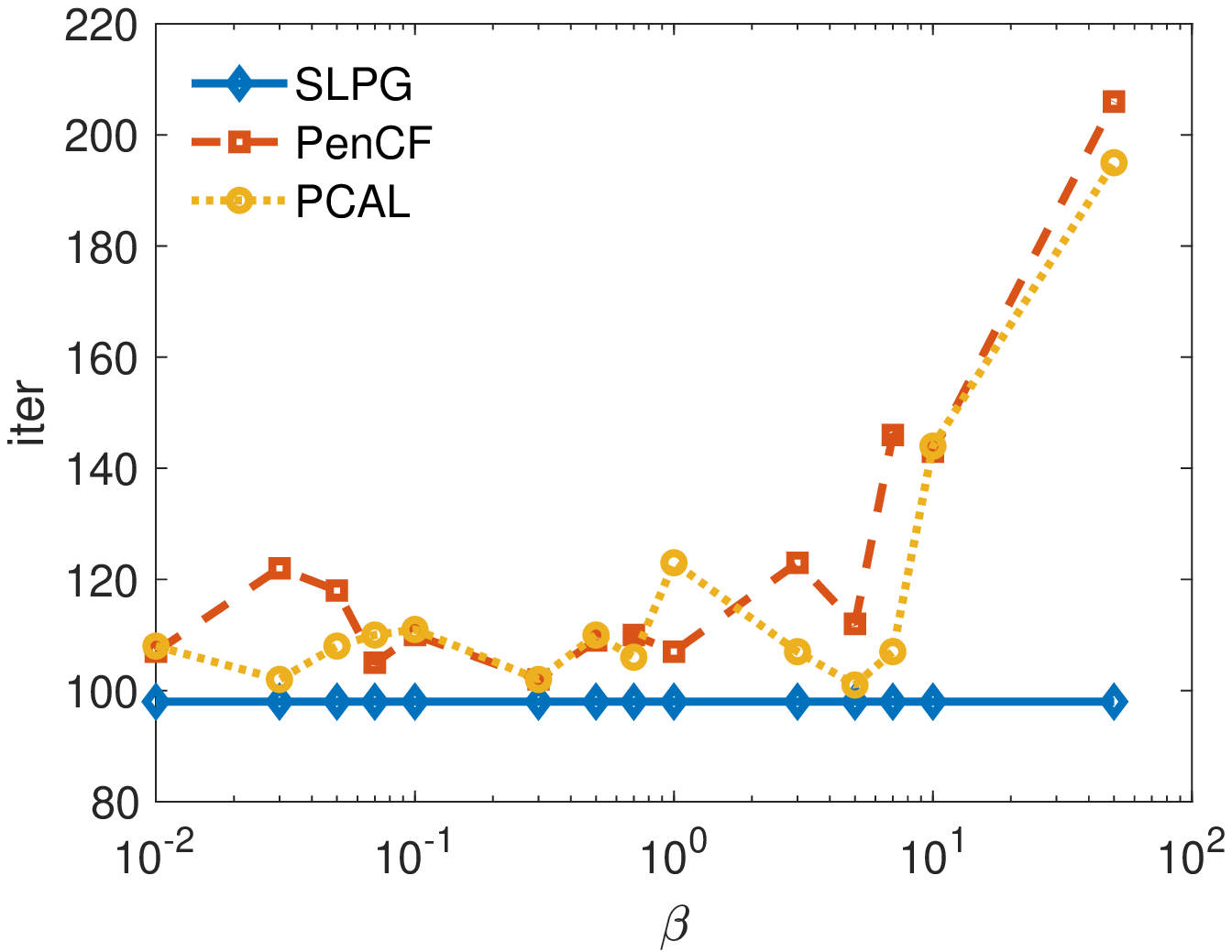}
				\label{Fig:KS_Prob6_item_1}
			\end{minipage}
		}%
		
		\subfigure[graphene16]{
			\begin{minipage}[t]{0.33\linewidth}
				\centering
				\includegraphics[width=\linewidth,height=0.18\textheight]{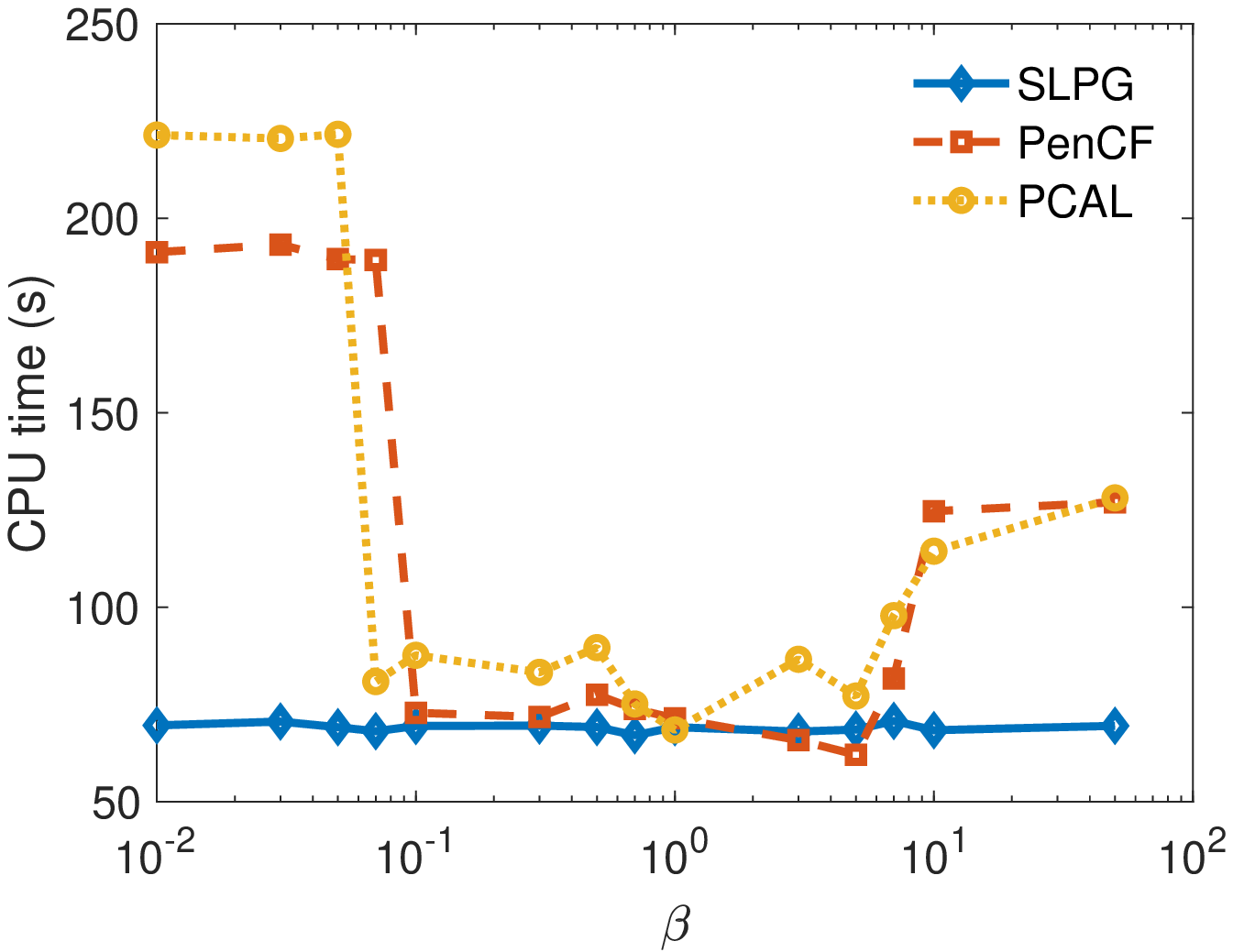}
				\label{Fig:KS_Prob4_item_2}
			\end{minipage}%
		}%
		\subfigure[C12H26]{
			\begin{minipage}[t]{0.33\linewidth}
				\centering
				\includegraphics[width=\linewidth,height=0.18\textheight]{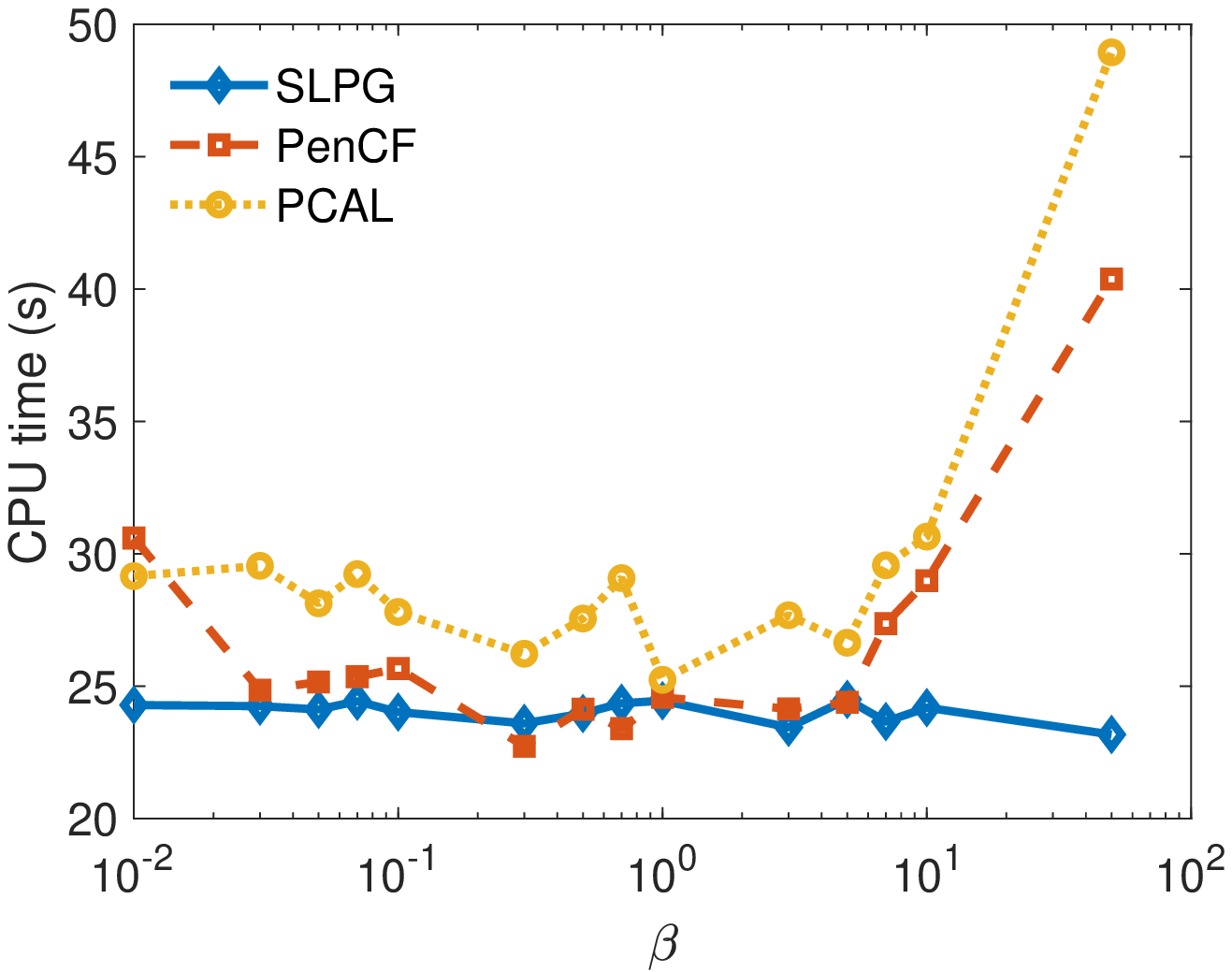}
				\label{Fig:KS_Prob5_item_2}
			\end{minipage}%
		}%
		\subfigure[pentacene]{
			\begin{minipage}[t]{0.33\linewidth}
				\centering
				\includegraphics[width=\linewidth,height=0.18\textheight]{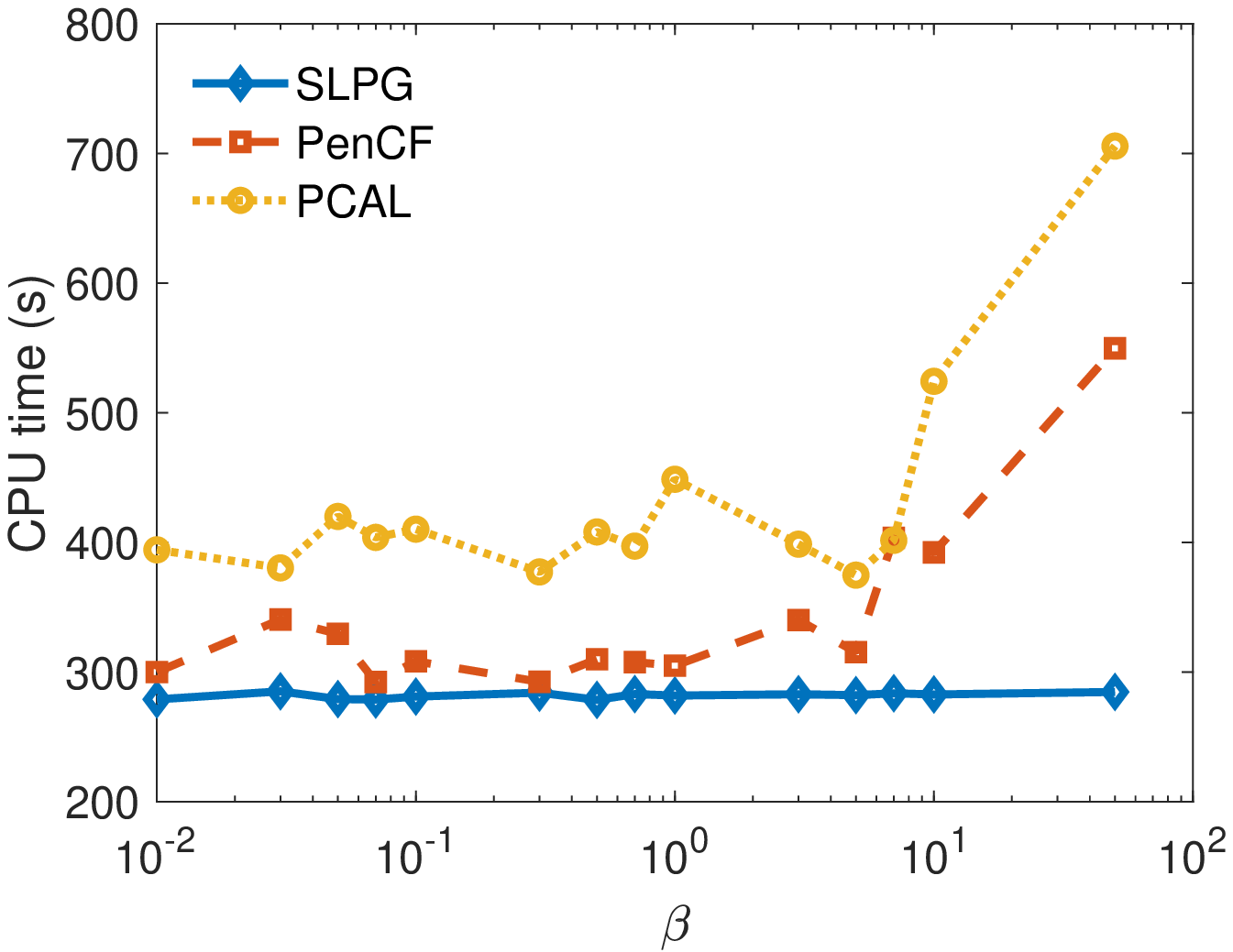}
				\label{Fig:KS_Prob6_item_2}
			\end{minipage}
		}%

		\caption{A detailed comparison  on the iterations and CPU time taken by \SLPG, PenCF and PCAL in KSSOLV.}
		
		\label{Fig_KS_iter}
	\end{figure}

	\begin{longtable}{@{}cccccc@{}}
		\toprule
		Solver                        & $E_{tot}$    & Substationarity & Iteration & Feasibility violation & CPU time(s) \\* \midrule
		\endfirsthead
		\multicolumn{6}{c}%
		{{\bfseries Table \thetable\ continued from previous page}} \\
		\endhead
		\bottomrule
		\endfoot
		\endlastfoot
		\multicolumn{6}{c}{alanine, $(n,p) = (12671, 18)$}    \\ \midrule 
		\multicolumn{1}{c|}{ManOptQR}  & -6.11e+01 &  9.88e-08 &    80 & 2.01e-15 &    24.16  \\ 
		\multicolumn{1}{c|}{OptM}  & -6.11e+01 &  2.15e-08 &    87 & 4.22e-14 &    26.35  \\ 
		\multicolumn{1}{c|}{PCAL}  & -6.11e+01 &  7.16e-08 &    97 & 2.80e-15 &    30.56  \\ 
		\multicolumn{1}{c|}{PenCF}  & -6.11e+01 &  2.83e-08 &    83 & 1.90e-15 &    24.10  \\  
		\multicolumn{1}{c|}{SLPG}  & -6.11e+01 &  7.15e-08 &    73 & 6.65e-16 &    21.27  \\  \midrule 
		\multicolumn{6}{c}{benzene, $(n,p) = (8407, 15)$}    \\ \midrule 
		\multicolumn{1}{c|}{ManOptQR}  & -3.72e+01 &  8.48e-08 &   163 & 2.07e-15 &    27.52  \\ 
		\multicolumn{1}{c|}{OptM}  & -3.72e+01 &  1.19e-08 &    82 & 2.46e-14 &    14.91  \\ 
		\multicolumn{1}{c|}{PCAL}  & -3.72e+01 &  6.88e-08 &    67 & 2.35e-15 &    13.17  \\ 
		\multicolumn{1}{c|}{PenCF}  & -3.72e+01 &  7.44e-08 &    67 & 2.57e-15 &    12.15  \\  
		\multicolumn{1}{c|}{SLPG}  & -3.72e+01 &  1.68e-08 &    66 & 8.66e-16 &    11.83  \\  \midrule 
		\multicolumn{6}{c}{c12h26, $(n,p) = (5709, 37)$}    \\ \midrule 
		\multicolumn{1}{c|}{ManOptQR}  & -8.15e+01 &  8.85e-08 &   439 & 5.06e-15 &   131.12  \\ 
		\multicolumn{1}{c|}{OptM}  & -8.15e+01 &  2.49e-08 &   105 & 8.50e-14 &    33.71  \\ 
		\multicolumn{1}{c|}{PCAL}  & -8.15e+01 &  8.46e-08 &    66 & 4.68e-15 &    25.23  \\ 
		\multicolumn{1}{c|}{PenCF}  & -8.15e+01 &  7.70e-08 &    81 & 4.72e-15 &    25.92  \\  
		\multicolumn{1}{c|}{SLPG}  & -8.15e+01 &  8.83e-08 &    72 & 1.34e-15 &    24.26  \\  \midrule 
		\multicolumn{6}{c}{ctube661, $(n,p) = (12599, 48)$}    \\ \midrule 
		\multicolumn{1}{c|}{ManOptQR}  & 2.51e+01 &  2.45e+01 &  1000 & 4.22e-15 &   743.02  \\ 
		\multicolumn{1}{c|}{OptM}  & -1.34e+02 &  8.64e-09 &   108 & 4.89e-15 &    90.88  \\ 
		\multicolumn{1}{c|}{PCAL}  & -1.34e+02 &  9.54e-08 &    73 & 4.80e-15 &    68.42  \\ 
		\multicolumn{1}{c|}{PenCF}  & -1.34e+02 &  5.04e-08 &    76 & 5.04e-15 &    60.32  \\  
		\multicolumn{1}{c|}{SLPG}  & -1.34e+02 &  7.67e-08 &    69 & 1.36e-15 &    56.94  \\  \midrule 
		\multicolumn{6}{c}{glutamine, $(n, p) = (16517, 29)$}    \\ \midrule 
		\multicolumn{1}{c|}{ManOptQR}  & -9.18e+01 &  7.03e-08 &   180 & 3.20e-15 &   138.08  \\ 
		\multicolumn{1}{c|}{OptM}  & -9.18e+01 &  1.39e-08 &   129 & 3.29e-15 &   102.58  \\ 
		\multicolumn{1}{c|}{PCAL}  & -9.18e+01 &  6.67e-08 &   108 & 3.08e-15 &    88.68  \\ 
		\multicolumn{1}{c|}{PenCF}  & -9.18e+01 &  9.35e-08 &   109 & 3.16e-15 &    83.01  \\  
		\multicolumn{1}{c|}{SLPG}  & -9.18e+01 &  9.06e-08 &   104 & 9.50e-16 &    78.98  \\  \midrule 
		\multicolumn{6}{c}{graphene16, $(n,p) = (3071, 37)$}    \\ \midrule 
		\multicolumn{1}{c|}{ManOptQR}  & -9.40e+01 &  8.74e-08 &   326 & 4.24e-15 &    67.13  \\ 
		\multicolumn{1}{c|}{OptM}  & -9.40e+01 &  2.34e-08 &   313 & 4.34e-15 &    66.44  \\ 
		\multicolumn{1}{c|}{PCAL}  & -9.40e+01 &  4.81e-08 &   416 & 4.41e-15 &    94.70  \\ 
		\multicolumn{1}{c|}{PenCF}  & -9.40e+01 &  3.74e-08 &   327 & 4.18e-15 &    65.00  \\  
		\multicolumn{1}{c|}{SLPG}  & -9.40e+01 &  9.53e-08 &   286 & 1.22e-15 &    56.71  \\  \midrule 
		\multicolumn{6}{c}{pentacene, $(n,p) = (44791, 51)$}    \\ \midrule 
		\multicolumn{1}{c|}{ManOptQR}  & -1.31e+02 &  9.20e-08 &   150 & 4.79e-15 &   425.62  \\ 
		\multicolumn{1}{c|}{OptM}  & -1.31e+02 &  2.37e-08 &   126 & 4.52e-15 &   374.43  \\ 
		\multicolumn{1}{c|}{PCAL}  & -1.31e+02 &  8.62e-08 &   111 & 4.06e-15 &   365.80  \\ 
		\multicolumn{1}{c|}{PenCF}  & -1.31e+02 &  6.40e-08 &   109 & 4.56e-15 &   301.13  \\  
		\multicolumn{1}{c|}{SLPG}  & -1.31e+02 &  9.58e-08 &   113 & 1.25e-15 &   308.88  \\  \midrule 
		\multicolumn{6}{c}{ptnio, $(n,p) = (4069, 43)$}    \\ \midrule 
		\multicolumn{1}{c|}{ManOptQR}  & -2.26e+02 &  8.62e-08 &   661 & 3.90e-15 &   166.45  \\ 
		\multicolumn{1}{c|}{OptM}  & -2.26e+02 &  2.49e-08 &   662 & 3.96e-15 &   171.86  \\ 
		\multicolumn{1}{c|}{PCAL}  & -2.26e+02 &  9.53e-08 &   596 & 3.75e-15 &   169.96  \\ 
		\multicolumn{1}{c|}{PenCF}  & -2.26e+02 &  8.54e-08 &   508 & 3.98e-15 &   123.19  \\  
		\multicolumn{1}{c|}{SLPG}  & -2.26e+02 &  6.70e-08 &   706 & 1.19e-15 &   169.87  \\  \midrule 
		\multicolumn{6}{c}{si64, $(n,p) = (6451, 128)$}    \\ \midrule 
		\multicolumn{1}{c|}{ManOptQR}  & 1.58e+02 &  2.88e+01 &  1000 & 8.02e-15 &  2413.09  \\ 
		\multicolumn{1}{c|}{OptM}  & -2.53e+02 &  2.18e-08 &   124 & 1.03e-14 &   328.57  \\ 
		\multicolumn{1}{c|}{PCAL}  & -2.53e+02 &  9.97e-08 &    74 & 1.02e-14 &   227.23  \\ 
		\multicolumn{1}{c|}{PenCF}  & -2.53e+02 &  7.48e-08 &    68 & 1.03e-14 &   179.30  \\  
		\multicolumn{1}{c|}{\SLPG}  & -2.53e+02 &  9.73e-08 &    74 & 2.25e-15 &   194.01  \\  \midrule 
		\multicolumn{6}{c}{si8, $(n,p) = (799, 16)$}    \\ \midrule 
		\multicolumn{1}{c|}{ManOptQR}  & -3.13e+01 &  9.39e-08 &   394 & 2.33e-15 &    38.27  \\ 
		\multicolumn{1}{c|}{OptM}  & -3.13e+01 &  2.20e-08 &   182 & 1.73e-15 &    18.66  \\ 
		\multicolumn{1}{c|}{PCAL}  & -3.13e+01 &  8.23e-08 &    73 & 2.16e-15 &     8.27  \\ 
		\multicolumn{1}{c|}{PenCF}  & -3.13e+01 &  9.14e-08 &    90 & 1.73e-15 &     9.22  \\  
		\multicolumn{1}{c|}{\SLPG}  & -3.13e+01 &  8.54e-08 &    72 & 5.81e-16 &     7.60  \\  \midrule
		\bottomrule[.4mm] 
		\caption{The results in Kohn-Sham total energy minimization}
		\label{Table_KS}
	\end{longtable}

	\section{Conclusion}
	
	In this paper, we have presented a penalty-free infeasible approach 
	called  \SLPGs for solving optimization problems
	over the Stiefel manifold with possibly nonsmooth objective functions.
	Our \SLPGs has two main steps. The first step is to solve a linearized proximal approximation in an affine subspace, which reduces to a tangent space
	of the Stiefel manifold
	if the iterate is feasible. We suggest to adopt
	a fixed point iteration to solve this tangential subproblem.
	Particularly, when the objective function is smooth or of $\ell_{2,1}$ regularization term, we can adopt
	an empirical direct approach to inexactly solve the tangential subproblem instead of the fixed point iteration. 
	The other step is to approximate the orthonormalization procedure by a cheap normal step, which is inspired from
	the Taylor expansion of the polar decomposition. The main advantages of our approach lie in
	the following three aspects. Firstly, we adopt an infeasible framework which is of better scalability than those 
	manifold-based approaches. Secondly, compared with the existing infeasible approaches, 
	\SLPGs does not invoke any penalty function, and hence the sensitivity of the performance to the choice of 
	penalty parameters is naturally eliminated. Thirdly, numerical experiments demonstrate the great potential of
	\SLPGs in solving \eqref{Prob_Ori} with both smooth and nonsmooth objective functions. In addition, we 
	have established the global convergence results for \SLPG.

	\bibliography{ref}

\begin{thebibliography}{51}
\providecommand{\natexlab}[1]{#1}
\providecommand{\url}[1]{\texttt{#1}}
\expandafter\ifx\csname urlstyle\endcsname\relax
  \providecommand{\doi}[1]{doi: #1}\else
  \providecommand{\doi}{doi: \begingroup \urlstyle{rm}\Url}\fi

\bibitem[Abrudan et~al.(2009)Abrudan, Eriksson, and
  Koivunen]{abrudan2009conjugate}
Traian Abrudan, Jan Eriksson, and Visa Koivunen.
\newblock Conjugate gradient algorithm for optimization under unitary matrix
  constraint.
\newblock \emph{Signal Processing}, 89\penalty0 (9):\penalty0 1704--1714, 2009.

\bibitem[Abrudan et~al.(2008)Abrudan, Eriksson, and
  Koivunen]{abrudan2008steepest}
Traian~E Abrudan, Jan Eriksson, and Visa Koivunen.
\newblock Steepest descent algorithms for optimization under unitary matrix
  constraint.
\newblock \emph{IEEE Transactions on Signal Processing}, 56\penalty0
  (3):\penalty0 1134--1147, 2008.

\bibitem[Absil et~al.(2007)Absil, Baker, and Gallivan]{absil2007trust}
P-A Absil, Christopher~G Baker, and Kyle~A Gallivan.
\newblock Trust-region methods on riemannian manifolds.
\newblock \emph{Foundations of Computational Mathematics}, 7\penalty0
  (3):\penalty0 303--330, 2007.

\bibitem[Absil et~al.(2009)Absil, Mahony, and Sepulchre]{Absil2009optimization}
P-A Absil, Robert Mahony, and Rodolphe Sepulchre.
\newblock \emph{Optimization algorithms on matrix manifolds}.
\newblock Princeton University Press, 2009.

\bibitem[Barzilai and Borwein(1988)]{barzilai1988two}
Jonathan Barzilai and Jonathan~M Borwein.
\newblock Two-point step size gradient methods.
\newblock \emph{IMA journal of numerical analysis}, 8\penalty0 (1):\penalty0
  141--148, 1988.

\bibitem[{Beale} et~al.(1959){Beale}, {Arrow}, {Hurwicz}, {Uzawa}, {Chenery},
  {Johnson}, {Karlin}, {Marschak}, and {Solow}]{beale1959studies}
E.~M.~L. {Beale}, Kenneth~J. {Arrow}, Leonid {Hurwicz}, Hirofumi {Uzawa},
  Hollis~B. {Chenery}, Selmer~M. {Johnson}, Samuel {Karlin}, Thomas {Marschak},
  and Robert~M. {Solow}.
\newblock Studies in linear and non-linear programming.
\newblock \emph{Journal of the Royal Statistical Society. Series A (General)},
  122\penalty0 (3):\penalty0 381, 1959.

\bibitem[Berge(1963)]{berge1963topological}
Claude Berge.
\newblock Topological spaces, oliver and boyed, edinburg-london.
\newblock \emph{1st English edition}, 1963.

\bibitem[Boumal et~al.(2014)Boumal, Mishra, Absil, and
  Sepulchre]{boumal2014manopt}
Nicolas Boumal, Bamdev Mishra, P-A Absil, and Rodolphe Sepulchre.
\newblock Manopt, a matlab toolbox for optimization on manifolds.
\newblock \emph{The Journal of Machine Learning Research}, 15\penalty0
  (1):\penalty0 1455--1459, 2014.

\bibitem[Cai et~al.(2013)Cai, Ma, Wu, et~al.]{cai2013sparse}
T~Tony Cai, Zongming Ma, Yihong Wu, et~al.
\newblock Sparse pca: Optimal rates and adaptive estimation.
\newblock \emph{The Annals of Statistics}, 41\penalty0 (6):\penalty0
  3074--3110, 2013.

\bibitem[{Chambolle} and {Pock}(2011)]{chambolle2011a}
Antonin {Chambolle} and Thomas {Pock}.
\newblock A first-order primal-dual algorithm for convex problems with
  applications to imaging.
\newblock \emph{Journal of Mathematical Imaging and Vision}, 40\penalty0
  (1):\penalty0 120--145, 2011.

\bibitem[Chen et~al.(2019)Chen, Deng, Ma, and So]{chen2019manifold}
Shixiang Chen, Zengde Deng, Shiqian Ma, and Anthony Man-Cho So.
\newblock Manifold proximal point algorithms for dual principal component
  pursuit and orthogonal dictionary learning.
\newblock In \emph{Asilomar Conference on Signals, Systems, and Computers},
  2019.

\bibitem[Chen et~al.(2020)Chen, Ma, Man-Cho~So, and Zhang]{chen2018proximal}
Shixiang Chen, Shiqian Ma, Anthony Man-Cho~So, and Tong Zhang.
\newblock Proximal gradient method for nonsmooth optimization over the stiefel
  manifold.
\newblock \emph{SIAM Journal on Optimization}, 30\penalty0 (1):\penalty0
  210--239, 2020.

\bibitem[Chen et~al.(2016)Chen, Ji, and You]{chen2016augmented}
Weiqiang Chen, Hui Ji, and Yanfei You.
\newblock An augmented lagrangian method for 1-regularized optimization
  problems with orthogonality constraints.
\newblock \emph{SIAM Journal on Scientific Computing}, 38\penalty0
  (4):\penalty0 B570--B592, 2016.

\bibitem[Chen et~al.(2010)Chen, Zou, Cook, et~al.]{chen2010coordinate}
Xin Chen, Changliang Zou, R~Dennis Cook, et~al.
\newblock Coordinate-independent sparse sufficient dimension reduction and
  variable selection.
\newblock \emph{The Annals of Statistics}, 38\penalty0 (6):\penalty0
  3696--3723, 2010.

\bibitem[{Chen} et~al.(2019){Chen}, {Dai}, and {Liu}]{chen2019a}
Zhongwen {Chen}, Yu-Hong {Dai}, and Jiangyan {Liu}.
\newblock A penalty-free method with superlinear convergence for equality
  constrained optimization.
\newblock \emph{Computational Optimization and Applications}, pages 1--33,
  2019.

\bibitem[Clarke(1990)]{clarke1990optimization}
Frank~H Clarke.
\newblock \emph{Optimization and nonsmooth analysis}, volume~5.
\newblock Siam, 1990.

\bibitem[Dai et~al.(2019)Dai, Zhang, and Zhou]{dai2019adaptive}
Xiaoying Dai, Liwei Zhang, and Aihui Zhou.
\newblock Adaptive step size strategy for orthogonality constrained line search
  methods.
\newblock \emph{arXiv preprint arXiv:1906.02883}, 2019.

\bibitem[Edelman et~al.(1998)Edelman, Arias, and Smith]{edelman1998geometry}
Alan Edelman, Tom{\'a}s~A Arias, and Steven~T Smith.
\newblock The geometry of algorithms with orthogonality constraints.
\newblock \emph{SIAM journal on Matrix Analysis and Applications}, 20\penalty0
  (2):\penalty0 303--353, 1998.

\bibitem[Gao et~al.(2018)Gao, Liu, Chen, and Yuan]{gao2018new}
Bin Gao, Xin Liu, Xiaojun Chen, and Ya-xiang Yuan.
\newblock A new first-order algorithmic framework for optimization problems
  with orthogonality constraints.
\newblock \emph{SIAM Journal on Optimization}, 28\penalty0 (1):\penalty0
  302--332, 2018.

\bibitem[Gao et~al.(2019)Gao, Liu, and Yuan]{gao2019parallelizable}
Bin Gao, Xin Liu, and Ya-xiang Yuan.
\newblock Parallelizable algorithms for optimization problems with
  orthogonality constraints.
\newblock \emph{SIAM Journal on Scientific Computing}, 41\penalty0
  (3):\penalty0 A1949--A1983, 2019.

\bibitem[Gao et~al.(2020)Gao, Hu, Kuang, and Liu]{GHKL2020}
Bin Gao, Guanghui Hu, Yang Kuang, and Xin Liu.
\newblock An orthogonalization-free parallelizable framework for all-electron
  calculations in density funcitonal theory.
\newblock \emph{arXiv preprint arXiv:2007.14228}, 2020.

\bibitem[Gao et~al.(2017)Gao, Ma, Zhou, et~al.]{gao2017sparse}
Chao Gao, Zongming Ma, Harrison~H Zhou, et~al.
\newblock Sparse cca: Adaptive estimation and computational barriers.
\newblock \emph{The Annals of Statistics}, 45\penalty0 (5):\penalty0
  2074--2101, 2017.

\bibitem[{Gould} and {Toint}(2010)]{gould2010nonlinear}
N.~I.~M. {Gould} and Ph.~L. {Toint}.
\newblock Nonlinear programming without a penalty function or a filter.
\newblock \emph{Mathematical Programming}, 122\penalty0 (1):\penalty0 155--196,
  2010.

\bibitem[{He} et~al.(2014){He}, {You}, and {Yuan}]{he2014on}
Bingsheng {He}, Yanfei {You}, and Xiaoming {Yuan}.
\newblock On the convergence of primal-dual hybrid gradient algorithm.
\newblock \emph{Siam Journal on Imaging Sciences}, 7\penalty0 (4):\penalty0
  2526--2537, 2014.

\bibitem[Higham and Papadimitriou(1994)]{higham1994APA}
Nicholas~J. Higham and Pythagoras Papadimitriou.
\newblock A parallel algorithm for computing the polar decomposition.
\newblock \emph{Parallel Computing}, 20:\penalty0 1161--1173, 1994.

\bibitem[Hiriart-Urruty and Lemar{\'e}chal(2013)]{hiriart2013convex}
Jean-Baptiste Hiriart-Urruty and Claude Lemar{\'e}chal.
\newblock \emph{Convex analysis and minimization algorithms I: Fundamentals},
  volume 305.
\newblock Springer science \& business media, 2013.

\bibitem[Hu et~al.(2018)Hu, Milzarek, Wen, and Yuan]{hu2018adaptive}
Jiang Hu, Andre Milzarek, Zaiwen Wen, and Yaxiang Yuan.
\newblock Adaptive quadratically regularized newton method for riemannian
  optimization.
\newblock \emph{SIAM Journal on Matrix Analysis and Applications}, 39\penalty0
  (3):\penalty0 1181--1207, 2018.

\bibitem[Hu et~al.(2020)Hu, Liu, Wen, and Yuan]{hu2020}
Jiang Hu, Xin Liu, Zaiwen Wen, and Ya-xiang Yuan.
\newblock A brief introduction to manifold optimization.
\newblock \emph{Journal of the Operations Research Society of China}, \penalty0
  (8):\penalty0 199--248, 2020.

\bibitem[{Hu} and {Liu}(2020)]{hu2020anefficiency}
Xiaoyin {Hu} and Xin {Liu}.
\newblock An efficient orthonormalization-free approach for sparse dictionary
  learning and dual principal component pursuit.
\newblock \emph{Sensors}, 20\penalty0 (3041), 2020.

\bibitem[Huang and Wei(2019{\natexlab{a}})]{huang2019extending}
Wen Huang and Ke~Wei.
\newblock Extending fista to riemannian optimization for sparse pca.
\newblock \emph{arXiv preprint arXiv:1909.05485}, 2019{\natexlab{a}}.

\bibitem[Huang and Wei(2019{\natexlab{b}})]{huang2019riemannian}
Wen Huang and Ke~Wei.
\newblock Riemannian proximal gradient methods.
\newblock \emph{arXiv preprint arXiv:1909.06065}, 2019{\natexlab{b}}.

\bibitem[Jiang and Dai(2015)]{jiang2015framework}
Bo~Jiang and Yu-Hong Dai.
\newblock A framework of constraint preserving update schemes for optimization
  on stiefel manifold.
\newblock \emph{Mathematical Programming}, 153\penalty0 (2):\penalty0 535--575,
  2015.

\bibitem[Kohn and Sham(1965)]{Kohn1965Self}
Walter Kohn and Lu~Jeu Sham.
\newblock Self-consistent equations including exchange and correlation effects.
\newblock \emph{Physical review}, 140\penalty0 (4A):\penalty0 A1133, 1965.

\bibitem[Lai and Osher(2014)]{lai2014splitting}
Rongjie Lai and Stanley Osher.
\newblock A splitting method for orthogonality constrained problems.
\newblock \emph{Journal of Scientific Computing}, 58\penalty0 (2):\penalty0
  431--449, 2014.

\bibitem[{Liu} and {Yuan}(2011)]{liu2011a}
Xinwei {Liu} and Yaxiang {Yuan}.
\newblock A sequential quadratic programming method without a penalty function
  or a filter for nonlinear equality constrained optimization.
\newblock \emph{Siam Journal on Optimization}, 21\penalty0 (2):\penalty0
  545--571, 2011.

\bibitem[Ma et~al.(2013)]{ma2013sparse}
Zongming Ma et~al.
\newblock Sparse principal component analysis and iterative thresholding.
\newblock \emph{The Annals of Statistics}, 41\penalty0 (2):\penalty0 772--801,
  2013.

\bibitem[Manton(2002)]{manton2002optimization}
Jonathan~H Manton.
\newblock Optimization algorithms exploiting unitary constraints.
\newblock \emph{IEEE Transactions on Signal Processing}, 50\penalty0
  (3):\penalty0 635--650, 2002.

\bibitem[{Martinez}(2001)]{martinez2001inexact}
J.~M. {Martinez}.
\newblock Inexact-restoration method with lagrangian tangent decrease and new
  merit function for nonlinear programming.
\newblock \emph{Journal of Optimization Theory and Applications}, 111\penalty0
  (1):\penalty0 39--58, 2001.

\bibitem[Nishimori and Akaho(2005)]{nishimori2005learning}
Yasunori Nishimori and Shotaro Akaho.
\newblock Learning algorithms utilizing quasi-geodesic flows on the stiefel
  manifold.
\newblock \emph{Neurocomputing}, 67:\penalty0 106--135, 2005.

\bibitem[Rockafellar and Wets(2009)]{rockafellar2009variational}
R~Tyrrell Rockafellar and Roger J-B Wets.
\newblock \emph{Variational analysis}, volume 317.
\newblock Springer Science \& Business Media, 2009.

\bibitem[Rosman et~al.(2014)Rosman, Tai, Kimmel, and
  Bruckstein]{rosman2014augmented-lagrangian}
Guy Rosman, Xuecheng Tai, Ron Kimmel, and Alfred~M Bruckstein.
\newblock Augmented-lagrangian regularization of matrix-valued maps.
\newblock \emph{Methods and applications of analysis}, 21\penalty0
  (1):\penalty0 105--122, 2014.

\bibitem[{Shen} et~al.(2012){Shen}, {Leyffer}, and {Fletcher}]{shen2012a}
Chungen {Shen}, Sven {Leyffer}, and Roger {Fletcher}.
\newblock A nonmonotone filter method for nonlinear optimization.
\newblock \emph{Computational Optimization and Applications}, 52\penalty0
  (3):\penalty0 583--607, 2012.

\bibitem[Sun and Sun(2002)]{Sun2002}
Defeng Sun and Jie Sun.
\newblock Semismooth matrix-valued functions.
\newblock \emph{Mathematics of Operations Research}, 27:\penalty0 150--169,
  2002.

\bibitem[{Ulbrich} and {Ulbrich}(2003)]{ulbrich2003non}
Michael {Ulbrich} and Stefan {Ulbrich}.
\newblock Non-monotone trust region methods for nonlinear equality constrained
  optimization without a penalty function.
\newblock \emph{Mathematical Programming}, 95\penalty0 (1):\penalty0 103--135,
  2003.

\bibitem[{Ulbrich}(2004)]{ulbrich2004on}
Stefan {Ulbrich}.
\newblock On the superlinear local convergence of a filter-sqp method.
\newblock \emph{Mathematical Programming}, 100\penalty0 (1):\penalty0 217--245,
  2004.

\bibitem[Ulfarsson and Solo(2008)]{ulfarsson2008sparse}
Magnus~O Ulfarsson and Victor Solo.
\newblock Sparse variable pca using geodesic steepest descent.
\newblock \emph{IEEE Transactions on Signal Processing}, 56\penalty0
  (12):\penalty0 5823--5832, 2008.

\bibitem[Wen and Yin(2013)]{wen2013feasible}
Zaiwen Wen and Wotao Yin.
\newblock A feasible method for optimization with orthogonality constraints.
\newblock \emph{Mathematical Programming}, 142\penalty0 (1-2):\penalty0
  397--434, 2013.

\bibitem[Wen et~al.(2010)Wen, Yin, Goldfarb, and Zhang]{wen2010fast}
Zaiwen Wen, Wotao Yin, Donald Goldfarb, and Yin Zhang.
\newblock A fast algorithm for sparse reconstruction based on shrinkage,
  subspace optimization, and continuation.
\newblock \emph{SIAM Journal on Scientific Computing}, 32\penalty0
  (4):\penalty0 1832--1857, 2010.

\bibitem[Xiao et~al.(2020{\natexlab{a}})Xiao, Liu, and Yuan]{xiao2020class}
Nachuan Xiao, Xin Liu, and Ya-xiang Yuan.
\newblock A class of smooth exact penalty function methods for optimization
  problems with orthogonality constraints.
\newblock \emph{Optimization Methods and Software}, 2020{\natexlab{a}}.

\bibitem[Xiao et~al.(2020{\natexlab{b}})Xiao, Liu, and Yuan]{xiao2020l21}
Nachuan Xiao, Xin Liu, and Ya-xiang Yuan.
\newblock Exact penalty function for $\ell_{2,1}$ norm minimization with
  orthogonality constraints.
\newblock \emph{Optimization Online preprint:2020/07/7908}, 2020{\natexlab{b}}.

\bibitem[Yang et~al.(2009)Yang, Meza, Lee, and Wang]{yang2009kssolv}
Chao Yang, Juan~C Meza, Byounghak Lee, and Lin-Wang Wang.
\newblock Kssolv---a matlab toolbox for solving the kohn-sham equations.
\newblock \emph{ACM Transactions on Mathematical Software (TOMS)}, 36\penalty0
  (2):\penalty0 10, 2009.

\end{thebibliography}
	\bibliographystyle{plainnat}

\end{document}